\documentclass[11pt]{amsart}
\usepackage{amsmath,amscd,amssymb,enumerate,verbatim,parskip}
\usepackage[T1]{fontenc}
\usepackage{hyperref}
\usepackage{graphicx}
\usepackage{adjustbox}
\usepackage{pinlabel}
\usepackage{tikz}
\usepackage{mathrsfs} 
\usepackage{xypic}
\usepackage{mathrsfs}
\usepackage{enumitem}
\usetikzlibrary{matrix}
\usepackage{color}

\newcommand{\HH}{\mathbb{H}}

\newcommand{\C}{\mathbb{C}}

\newcommand{\R}{\mathbb{R}}

\newcommand{\Z}{\mathbb{Z}}

\newcommand{\id}{\mathrm{Id}}

\newcommand{\OP}{\operatorname}

\numberwithin{equation}{section}

\newtheorem{thm}[equation]{Theorem}
\newtheorem{dfn}[equation]{Definition}
\newtheorem{lma}[equation]{Lemma}
\newtheorem{prp}[equation]{Proposition}
\newtheorem{cor}[equation]{Corollary}
\newtheorem{ex}[equation]{Example}
\newtheorem{rmk}[equation]{Remark}

\newtheoremstyle{TheoremNum}
    {\topsep}{\topsep}              
    {\itshape}                      
    {}                              
    {\bfseries}                     
    {.}                             
    { }                             
    {\thmname{#1}\thmnote{ \bfseries #3}}
\theoremstyle{TheoremNum}

\begingroup 
\makeatletter 
\@for\theoremstyle:=definition,remark,plain,TheoremNum\do{%
\expandafter\g@addto@macro\csname th@\theoremstyle\endcsname{%
\addtolength\thm@preskip\parskip 
}%
} 
\endgroup

\author{Georgios Dimitroglou Rizell}
\address{Department of Mathematics\\
Uppsala University\\
Box 480\\
SE-751 06 Uppsala\\
SWEDEN}
\email{georgios.dimitroglou@math.uu.se}

\author{Michael G. Sullivan}
\address{Department of Mathematics\\
University of Massachusetts\\
Amherst\\
MA 01002\\
USA}
\email{sullivan@math.umass.edu}

\title{An Energy-Capacity inequality for Legendrian submanifolds}
\thanks{The first author is supported by grant KAW 2013.0321 from the Knut and Alice Wallenberg Foundation.
The second author is supported by grant 317469 from the Simons Foundation and thanks the Centre de Recherches Mathematiques for hosting him while some of this work was done. Both authors would like to thank the anonymous referee who read the paper very thoroughly and provided many valuable suggestions.}

\begin{document}

\begin{abstract}
We prove that the number of Reeb chords between a Legendrian submanifold and its contact Hamiltonian push-off is at least the sum of the $\Z_2$-Betti numbers of the submanifold, provided that the contact isotopy is sufficiently small when compared to the smallest Reeb chord on the Legendrian. Moreover, the established invariance enables us to use two different contact forms: one for the count of Reeb chords and another for the measure of the smallest length, under the assumption that there is a suitable symplectic cobordism from the latter to the former. The size of the contact isotopy is measured in terms of the oscillation of the contact Hamiltonian, together with the maximal factor by which the contact form is shrunk during the isotopy. 
The main tool used is a Mayer--Vietoris sequence for Lagrangian Floer homology, obtained by ``neck-stretching'' and ``splashing.''
\end{abstract}

\maketitle

\section{Introduction}
\label{sec:Introduction}

A by-now famous Energy-Capacity inequality for closed Lagrangian submanifolds of tame symplectic manifolds was obtained in \cite{DisplacementLagrangian} by Polterovich. Roughly speaking, this inequality provides a lower bound for the displacement energy of a closed Lagrangian submanifold $L$ (this is an expression involving the Hofer norm \cite{HoferNorm}) in terms of the minimal area of a pseudoholomorphic disc with boundary on $L \subset (X,\omega)$ of a tame symplectic manifold. In \cite{Chekanov} an even stronger statement was established by Chekanov: the number of intersections $L \cap \phi(L)$ is bounded from below by $\dim H_*(L;\Z_2)$ whenever $\phi$ is a Hamiltonian diffeomorphism whose Hofer norm is less than this minimal area of a pseudoholomorphic disc, under the additional assumption that the intersection is transverse; see Theorem \ref{thm:chekanov} below for a precise formulation. Our goal is to provide a contact geometric analogue of the latter result. There have been previous generalizations of Chekanov's result in this direction, albeit not in the full generality that is considered here; see Section \ref{sec:previous} below for an overview.

A {\bf contact manifold} $(Y,\xi)$ with a contact form is a $(2n+1)$-dimensional manifold with a maximally non-integrable field of tangent hyperplanes $\xi \subset TY$, and a {\bf Legendrian submanifold} is an $n$-dimensional submanifold $\Lambda \subset (Y,\xi)$ which is tangent to $\xi$. We will assume that $\xi$ is co-orientable and choose a contact form $\alpha$, i.e.~a one-form satisfying $\xi=\ker \alpha$, which gives rise to the Reeb vector field $R_\alpha$ (whose dynamics depends heavily on $\alpha$). A choice of contact form induces a bijective correspondence between contact Hamiltonians $H_t \colon Y \to \R$ and contact isotopies $\phi^t_{\alpha,H_t} \colon (Y,\xi)\to(Y,\xi)$ starting at the identity (note that a contact isotopy need not preserve $\alpha$). We refer to Section \ref{sec:basics} for more details.

Two Legendrian submanifolds do not generically intersect. However, generically there are integral curves of $R_\alpha$ connecting them: so-called {\bf Reeb chords}. In fact, if a Legendrian isotopy $\Lambda_t$ is \emph{sufficiently} $C^0$-small for $t\in[0,1]$, it follows by the classical result due to Laudenbach--Sikorav \cite{Persistance} and Chekanov \cite{QuasiFunctions} that there are at least a number $\dim H_*(\Lambda;\Z_2)$ of Reeb chords between $\Lambda_0$ and $\Lambda_1$ for a fixed contact form (in the case when the isotopy is $C^1$-small the bound can be obtained by elementary results from differential topology).

A $C^0$-small Legendrian push-off of $\Lambda \subset (Y,\alpha)$ is contained inside a standard contact neighborhood as shown in \cite[Theorem 6.2.2]{Geiges}; this is a neighborhood that can be identified with a neighborhood of the zero-section of $(J^1\Lambda,dz+\theta_\Lambda)$ while preserving the contact forms, and identifying $\Lambda$ with the zero section. Here $\theta_\Lambda$ denotes the so-called Liouville form on $T^*\Lambda.$ The aforementioned Reeb chords between $\Lambda$ and a $C^0$-close push-off are in bijective correspondence with the intersection points of the images of the same under the canonical projection $J^1\Lambda \to T^*\Lambda$ (here we have used the above standard neighborhood), under which the Legendrian submanifolds become Lagrangian immersions. The contact-geometric counterpart of Chekanov's result would therefore answer the following question:
\begin{center}
\boxed{\parbox[t]{10.2cm}{ Given a Legendrian submanifold $\Lambda \subset (Y,\alpha)$, how ``large'' (in an appropriate sense) of a contact Hamiltonian $H_t$ is needed in order for the number of $\alpha$-Reeb chords that go either from $\Lambda$ to $\phi^1_{\alpha,H_t}(\Lambda),$ or vice versa, to be \emph{less} than $\dim H_*(\Lambda,\Z_2)$?}}
\end{center}
Note that the non-relative counterpart of this question concerns the number of so-called translated points of a contactomorphism. This question has received somewhat more attention in comparison with the relative case studied here. We refer to the work of Albers--Frauenfelder \cite{AlbersLeafwise}, Albers--Fuchs--Merry \cite{Orderability}, \cite{Positive}, Albers--Hein \cite{AlbersHein}, Sandon \cite{Sandon}, and Shelukhin in \cite{HoferContactomorphism} for related results concerning the existence of translated points.

\subsection{Results}
\label{sec:Results}
Given a pair Legendrian submanifold $\Lambda_0,\Lambda_1 \subset (Y,\alpha)$ by
$$\mathcal{Q}_\alpha(\Lambda_0,\Lambda_1;a,b), \:\: -\infty \le a \le b \le +\infty,$$
we denote the union of all Reeb chords for the contact form $\alpha$ either
\begin{itemize}
\item starting on $\Lambda_0$, ending on $\Lambda_1$, and being of length $0<\ell \in [a,b],$ or
\item starting on $\Lambda_1$, ending on $\Lambda_0$, and being of length $0<\ell \in [-b,-a].$
\end{itemize}
We also write $\mathcal{Q}_\alpha(\Lambda_0,\Lambda_1):=\mathcal{Q}_\alpha(\Lambda_0,\Lambda_1;-\infty,+\infty)$. Our goal is to obtain a lower bound on these subsets in the case when $\Lambda_0=\Lambda$ and $\Lambda_1=\phi^1_{\alpha,H_t}(\Lambda_0)$ whenever $H_t$ is sufficient small in an appropriate sense.
\begin{rmk}
\label{rmk:qa}
Interchanging the order of the two Legendrians we obtain
$$\mathcal{Q}_\alpha(\Lambda_0,\Lambda_1;a,b)=\mathcal{Q}_\alpha(\Lambda_1,\Lambda_0;-b,-a)$$
(in particular, the Reeb chords that are considered are allowed to start on either of the Legendrians).
\end{rmk}

Fix a contact Hamiltonian $H_t \colon Y \to \R.$ Since $\phi^t_{\alpha, H_t}$ preserves the contact distribution $\xi,$ we have
$$(\phi^t_{\alpha, H_t})^*\alpha=e^{-\tau^t_{\alpha,H_t}}\alpha,$$
where $e^{-\tau^t_{\alpha,H_t}}$ is called the {\bf conformal factor} of the contactomorphism.
To the contact isotopy $\phi^t_{\alpha,H_t}$ we can then associate the numbers
\begin{equation}
\label{eq:norms}
\boxed{
\begin{array}{rcl}
A&:=& \max_{t \in [0,1] \atop y \in Y} \tau^t_{\alpha,H_t}(y) \ge 0,\\
B&:=& -\min_{t \in [0,1] \atop y \in Y} \tau^t_{\alpha,H_t}(y) \ge 0,\\
M_+&:=& -e^A \int_0^1 \min_{y \in Y} H_t dt,\\
M_-&:=& -e^A \int_0^1 \max_{y \in Y} H_t dt,\\
\|H\|_{\OP{osc}} &:=& \int_0^1 (\max_{y \in Y} H_t-\min_{y \in Y}H_t) dt \ge 0.
\end{array}
}
\end{equation}
(The signs and the notation $M_\pm$ may look odd at first sight, but this notation will become useful in Section \ref{sec:Proof} when proving our main theorem.)

Instead of the minimal area of pseudoholomorphic discs, here we consider the following related geometric quantity:
\[ \boxed{\sigma(\alpha, \Lambda) := \parbox[t]{9cm}{The minimal $\alpha$-length of Reeb chords and periodic Reeb orbits $\gamma$ satisfying $[\gamma]=0 \in \pi_1(Y,\Lambda).$}}
\]
Note that $\sigma(\alpha, \Lambda)=+\infty$ holds when the set of contractible chords and orbits is empty.

In the following we assume that the Legendrian submanifold $\Lambda \subset (Y,\alpha)$ and the contact manifolds are closed, and that $\alpha$ is generically chosen, making both the periodic Reeb orbits and Reeb chords on $\Lambda$ non-degenerate. We defer the notion of a Lagrangian concordance to Section \ref{sec:basics}, but we note that $\Lambda \subset (Y,\alpha)$ is Lagrangian concordant to itself.
\begin{thm}
\label{thm:main}
Let $ L \subset (X,d\eta)$ be a Lagrangian concordance from $\Lambda_- \subset (Y_-,\alpha_-)$ to $\Lambda \subset (Y,\alpha),$ i.e.~$(Y_-,\alpha_-)$ and $(Y,\alpha)$ are the concave and convex contact ends of the symplectic cobordism $(X,d\eta),$ and consider a contact Hamiltonian $H_t \colon Y \to \R$ which satisfies
\begin{equation}
\label{eq:nobreaking}
e^A\|H\|_{\OP{osc}} < \sigma(\alpha_-, \Lambda_-).
\end{equation}
Then
$$ |\mathcal{Q}_\alpha (\Lambda,\phi^1_{\alpha,H_t}(\Lambda);-M_+,-M_-)| \ge \sum_{i=0}^{\dim \Lambda} \dim H_i(\Lambda;\Z_2),$$
assuming that the latter Reeb chords are transverse. 
\end{thm}
We emphasize that the obstruction is measured for the contact form $\alpha_-$ on the concave end of the symplectic cobordism $(X,d\eta)$, which need not be the same as the contact form $\alpha$ inducing the Reeb chords which we count.

\begin{rmk}
 Without the implementation in Section \ref{sec:usher} of Usher's trick in the setting of contact Hamiltonians, as we learned from the work \cite{HoferContactomorphism} by Shelukhin, we would have to assume
$$e^{A+B}\|H\|_{\OP{osc}} < \sigma(\alpha_-, \Lambda_-).$$
 instead of condition (\ref{eq:nobreaking}).
We believe condition (\ref{eq:nobreaking}) can be improved to
$$\|H\|_{\OP{osc}} < \sigma(\alpha_-, \Lambda_-),$$
i.e.~a quantity that only involves the $C^0$-properties of the contact Hamiltonian. Unfortunately, we were unable to show this with our techniques.
\end{rmk}

One nontrivial consequence of Theorem \ref{thm:main} is that the result can be used as an obstruction to the existence of Lagrangian concordances (cylindrical cobordisms) inside symplectic cobordisms. We refer to Corollary \ref{cor:JoshLisa} below for the case of exact Lagrangian concordances embedded in the symplectization. In addition, note that Lagrangian concordances also arise naturally when we interpolate between two contact forms, as described by Example \ref{ex:conc}. In the case of a ``displaceable'' Legendrian embedding, we hence immediately conclude that:
\begin{cor}
\label{cor:main}
Assume that
$$ |\mathcal{Q}_\alpha (\Lambda,\phi^1_{\alpha,H_t}(\Lambda);-M_+,-M_-)| < \sum_{i=0}^{\dim \Lambda} \dim H_i(\Lambda;\Z_2),$$
is satisfied for the contact Hamiltonian $H_t,$ where all chords are assumed to be transverse. Then for any smooth function $f \colon Y \to (0,1],$ the contact form $\alpha_-:=e^f\alpha$ satisfies
$$ e^A\|H\|_{\OP{osc}} > \sigma(\alpha_-, \Lambda_-).$$
In particular, there exists either a contractible chord or orbit for the contact form $\alpha_-$ satisfying a fixed bound of its length.
\end{cor}

With a refined version of the non-bubbling theorem proven in Section \ref{sec:CompactnessProof} it should be possible to replace in the hypothesis of Theorem \ref{thm:main}, $\sigma(\alpha,\Lambda)$ by the minimal $d\alpha_-$-energy of either a pseudoholomorphic plane in $\R \times Y_-,$ or a one-punctured pseudoholomorphic disc with boundary on $\R \times \Lambda_-.$ Such a plane or disc is asymptotic to either a periodic Reeb orbit or a Reeb chord on $\Lambda_-.$ Together with the fact that the $d\alpha_-$-energy of these pseudoholomorphic curves is equal to the length of the asymptotic orbit by Stokes' theorem, it now follows that the latter quantity is less than or equal to $\sigma(\alpha,\Lambda).$

Theorem \ref{thm:main} also extends to certain non-compact contact manifolds, such as one-jet spaces, if a standard contact form is used outside of a compact subset; we refer to Section \ref{sec:noncomp} for more details.

\begin{cor}
\label{cor:nondisp}
Let $\Lambda \subset (Y,\xi)$ be a Legendrian submanifold and let $\alpha_0$ be a (not necessarily generic) contact form on $(Y,\xi)$ which is {\bf relatively hypertight}, i.e.~for which $\sigma(\alpha_0,\Lambda)=+\infty.$ Suppose that $\Lambda'$ is Legendrian isotopic to $\Lambda$. For {\em any} choice of contact form $\alpha,$ we then have the bound
\[ |\mathcal{Q}_\alpha (\Lambda,\phi^1_{\alpha,H_t}(\Lambda);-M_+,-M_-)| \ge \sum_{i=0}^{\dim \Lambda} \dim H_i(\Lambda;\Z_2),\]
given that the latter chords are transverse.
\end{cor}
\begin{proof}

The contact form $\alpha$ can be written as $\alpha = e^f\alpha_0$ for some real-valued function $f \colon Y \to \R.$ Choose any constant $m > -\min_Y f.$ It follows that $\alpha_-:=e^{-m}\alpha_0$ also is a relatively hypertight contact form, while $\alpha_- = e^{-(f+m)}\alpha$ with $-(f+m)<-\min f+\min f=0.$ After a sufficiently small perturbation of the contact form $\alpha_-$ it may be assumed to be generic, while the inequality $ e^A\|H\|_{\OP{osc}} < \sigma(\alpha_-, \Lambda)$ is still satisfied for the contact Hamiltonian generating the isotopy from $\Lambda$ to $\phi^1_{\alpha,H_t}(\Lambda).$ Since it is possible to assume that $\alpha_-=e^g \alpha$ holds for some smooth function $g \colon Y \to (0,1)$ also after the perturbation, the result now follows directly from Corollary \ref{cor:main}.
\end{proof}

\begin{ex}
The following are well-known examples of Legendrian submanifolds satisfying $\sigma(\alpha_0,\Lambda) = +\infty$ to which the above corollary can be applied.
\begin{enumerate}
\item The zero section of the one-jet space $(J^1M,dz-\lambda_M)$ of a closed manifold $M$ endowed with its canonical contact form $\alpha_{\OP{std}}=dz-\lambda_M;$ i.e.~$\lambda_M$ is the Liouville form on $T^*M$ and $z$ is the canonical coordinate on the $\R$-factor of $J^1M=T^*M\times\R;$
\item A fiber of the unit cotangent bundle $S^*M \subset T^*M$ with contact form $\alpha_0=\lambda_M|_{T(S^*M)}$ induced by a Riemannian metric on $M$ having non-positive sectional curvature (recall that such a Riemannian manifold has no closed contractible geodesics, even without assuming periodicity); and
\item The conormal lift of a sub-torus $(S^1)^k \times \{1\}^{n-k} \subset (S^1)^n$ inside the unit cotangent bundle $S^*(S^1)^n,$ with contact form $\alpha_0$ induced by the canonical flat metric on $(S^1)^n=\R^n /\Z^n.$
\end{enumerate}
\end{ex}
In another direction, define the {\em{$\alpha$-displacement energy}} of a Legendrian $\Lambda$ in $(Y, \mbox{ker} \alpha)$ to be
 \[
 \OP{disp}(\alpha, \Lambda) := \inf_{H_t} e^A \|H\|_{\OP{osc}}
 \]
 where the infimum is taken over all contact Hamiltonians such that there are no $\alpha$-Reeb chords between $\phi^1_{\alpha, H_t}(\Lambda)$ and $\Lambda$ or vice versa. (Set this infimum to $+\infty$ if no such Hamiltonian exists.)
Suppose $\Lambda_-, \Lambda_+ \subset Y$ are two Legendrians which are Lagrangian concordant in the symplectization $(\R \times Y,d(e^r\alpha)).$ 
Without loss of generality, assume that the concordance is cylindrical outside of $[S,0] \times Y,$ where $S <0.$
 \begin{cor}
 \label{cor:JoshLisa}
Under the above assumptions, we have
$$ |S| \ge \ln\left(\frac{\sigma(\alpha, \Lambda_-)}{\OP{disp}(\alpha, \Lambda_+)}\right).$$
In particular if the length of the shortest $\alpha$-chord of $\Lambda_-$ is greater than the $\alpha$-displacement energy of $\Lambda_+,$ we achieve a non-trivial lower bound on the length of the Lagrangian concordance.
 \end{cor}
 
\begin{proof}
Suppose $H$ displaces $\Lambda^+$ is the above sense.
Theorem \ref{thm:main} implies
$$\sigma(\alpha_-, \Lambda_-) \le \OP{disp}(\alpha, \Lambda) \le e^A \|H\|_{\OP{osc}}$$
where $\alpha_- = e^S \alpha$ and hence $\sigma(\alpha, \Lambda_-)=e^{-S}\sigma(\alpha_-, \Lambda_-).$
\end{proof}

\subsection{Previous results in the symplectic setting}
\label{ssec:PreviousSymplectic}

In the following we assume that $(X,\omega)$ is a tame symplectic manifold.

\subsubsection{Results related to Corollary \ref{cor:nondisp}}
Lower bounds for the number of intersections between a Lagrangian submanifold $L$ and its image $\phi^1_{H_t}(L)$ under a Hamiltonian isotopy has been a major topic in symplectic topology. In certain cases it has been shown that
\[ |L \pitchfork \phi^1_{H_t}(L)| \ge \dim H_*(L;\Z_2) \] 
given that the intersection is transverse. Here we present the two most classical such results.
In \cite{Persistance} Laudenbach--Sikorav used generating family techniques to prove the statement for $(X,\omega)=(T^*M,d\theta_M)$ and $L=0_M$. 
Floer homology was introduced in \cite{Floer:MorseTheoryLagrangian} by Floer, based upon Gromov's technique of pseudoholomorphic curves \cite{Gromov}. The most basic version of Floer homology can handle the case $L\subset (X,\omega)$, under the assumption that there is no non-constant pseudoholomorphic representative of any element in $\pi_2(X,L)$. The work due to Floer proves the lower bound in this setting.
Finally, we note that a far-reaching generalization of Floer homology has been constructed, which can be used to determine such non-trivial lower bounds in many cases \cite{FloerAnomalyI}, \cite{FloerAnomalyII}.

\subsubsection{Results related to Theorem \ref{thm:main}}
In \cite{Chekanov} Chekanov provided the following refinement of the Energy-Capacity inequality due to Polterovich \cite{DisplacementLagrangian}. Recall the definition
\begin{eqnarray*}
\|\phi\| &:=& \inf_{\left\{H_t;\: \phi=\phi^1_{H_t}\right\}} \| H \|_{\OP{osc}}\\
\|H\|_{\OP{osc}} &:=& \int_0^1\left(\max_X H_t-\min_X H_t\right)dt,
\end{eqnarray*}
of the Hofer norm of a Hamiltonian diffeomorphism, as well as the definition of the holomorphic disc capacity
\[0 \le \sigma_\omega(L):= \sup_{J \in \mathcal{J}(X,\omega)} \inf_{u \in \mathcal{M}(L;J)}\int_u \omega \le +\infty,\]
where $\mathcal{M}(L;J)$ denotes the moduli space consisting of non-constant $J$-holomorphic representatives of elements in $\pi_2(X,L)$ and $\mathcal{J}(X,\omega)$ is the contractible set of tame almost complex structures on $(X,\omega)$.
\begin{thm}[\cite{Chekanov}]
\label{thm:chekanov}
Suppose that $\phi \colon (X,\omega) \to (X,\omega)$ is a Hamiltonian diffeomorphism of a tame symplectic manifold satisfying the inequality $\|\phi \| < \sigma_\omega(L)$ for a closed Lagrangian submanifold $L \subset (X,\omega)$. It follows that
\[ |L \cap \phi(L)| \ge \sum_{i=0}^{\dim L} \dim H_i (L;\Z_2),\]
given that the intersection $L \pitchfork \phi(L)$ is transverse.
\end{thm}

\subsection{Previous results in the contact setting}
\label{sec:previous}

\subsubsection{Results related to Corollary \ref{cor:nondisp}}
Chekanov's refinement \cite{QuasiFunctions} of the aforementioned result \cite{Persistance} by Laudenbach--Sikorav establishes the persistence of Reeb chords under general contact isotopies of the zero-section of $J^1M$ -- a set of transformations that obviously is much larger than its subset consisting of the lifts of Hamiltonian isotopies of $T^*M$. 
This can be used to deduce Corollary \ref{cor:nondisp} in the special case when $\Lambda=0_M \subset J^1M$ and when the contact forms are all taken to be standard, i.e.~when
$$\alpha=\alpha_0=\alpha_{\OP{std}}.$$
Recall that in this case the Reeb chords between $\Lambda$ and its push-off are in bijective correspondence with intersection points of their images under the canonical projection $ J^1M \to T^*M.$ Our result can therefore be seen as a generalization of this result to general contact forms. (Recall that we still have to make the requirement that the contact forms coincide with $\alpha_{\OP{std}}$ outside of a compact subset.)

In the more general case when $\alpha=\alpha_0,$ and $\sigma(\alpha_0,\Lambda)=+\infty,$ but under the additional assumption that
\[ \phi^\R_{R_\alpha}(\Lambda) := \bigcup_{t\in\R} \phi^t_{R_\alpha}(\Lambda) \subset Y \]
is a closed submanifold, the conclusion of Corollary \ref{cor:nondisp} also follows from work of Eliashberg--Hofer--Salamon \cite[Theorem 2.5.4]{LagCont}.

There have also been results inside certain prequantization spaces. Consider the standard Legendrian $\R P^n \subset \R P^{2n+1}$, i.e.~the image of
\[ \mathfrak{Re}\C^{n+1} \cap S^{2n+1} \subset \left(S^{2n+1},\mathrm{ker} \sum_{i=1}^{2n+1}(x_i dy_i-y_idx_i)\right)\]
under the canonical projection, where the contact structure on $\R P^{2n+1}$ is induced by the standard contact structure on $S^{n+1}$. A result by Givental \cite{Givental} shows that there must exist a Reeb chord between $\R P^n$ and $\phi^1_{\alpha,H_t}(\R P^n)$ for any choice of contact form $\alpha$ and contact isotopy; also, see related results \cite{Borman} by Borman--Zapolsky. Note that this result cannot be deduced from Corollary \ref{cor:nondisp}, due to the presence of contractible periodic orbits and chords. We expect that the results can be recovered by our methods as well, after a more refined Floer theoretic invariance has been established.

\subsubsection{Results related to Theorem \ref{thm:main}}
There have also been previous results along the lines of Theorem \ref{thm:main}, i.e.~taking quantitative properties of the contact Hamiltonian into account.

Notably, in the case when $\alpha=\alpha_-$ is the $S^1$-invariant contact form on a prequantization space
\[S^1 \to (P,\alpha) \to (M,\omega),\]
and when $\Lambda=\Lambda_-$ is the lift of an \emph{embedded} Lagrangian submanifold in $(M,\omega)$, such results were obtained in \cite{Her} by Her generalizing previous work by Ono \cite{Ono}. For a result in contactizations of Liouville domains with the standard contact form we also refer to the more recent result \cite{Akaho2} by Akaho.

For a Legendrian submanifold of a general contact manifold (again under the assumption that $\alpha=\alpha_-$ and $\Lambda=\Lambda_-$), Theorem \ref{thm:main} can also be seen to follow from a result by Akaho \cite{Akaho}, again under the additional assumption that
\[ \phi^\R_{R_\alpha}(\Lambda) \subset Y\]
is a \emph{closed} submanifold satisfying some additional topological constraints. Note that the latter behavior is non-generic and imposes severe restrictions on the contact form.

Finally, we mention the result \cite[Remark 1.14]{Entov:Tetragons} due to Entov and Polterovich which is relevant here. They use completely different techniques to show that, under suitable assumptions, a Legendrian and its image under the Reeb flow (for a fixed contact form) admit a Reeb chord from one component to the other for all choices of contact forms.

\subsubsection{Results related to Corollary \ref{cor:JoshLisa}}

Sabloff and Traynor prove a similar result when the Legendrian contact homology DGAs have augmentations \cite{SabloffTraynor15}, which in turned inspired this corollary.
Their hypotheses are stronger as many Legendrian contact homology DGAs do not have augmentations.
But, on the other hand, they have an improved bound where their numerator is not just a chord of minimal length, but runs over a collection of (possibly longer) chords which represent certain canonical classes in the so-called linearized Legendrian contact homology (this is a chain complex associated to the Legendrian which is generated by its Reeb chords). They also consider general Lagrangian cobordisms in the symplectization, as opposed to just concordances.

\subsection{Overview of paper}
\label{sec:Overview}
In Section \ref{sec:basics} we review background definitions of Lagrangian cobordisms. In Section \ref{sec:action} we introduce a Language that will later allow us to consider different Hofer-type energies for a single pseudoholomorphic curve in a fixed symplectic cobordism; roughly speaking, instead of the symplectic area induced by the symplectic form $d\eta$ we will also consider the ``area'' induced by the (not necessarily symplectic) two-form $d(\varphi \eta)$ for different functions $\varphi$ subject to certain conditions. These Hofer-type energies are important for studying the compactness properties of the moduli spaces, but later we will also utilize this language in order to formulate a version of neck-stretching.

In Section \ref{sec:FloerHomology}, we introduce a Floer theory for Lagrangian cobordisms. The Floer strips with Hamiltonian perturbation terms are of the type usually considered in Lagrangian Floer homology. (However, some of the formulations might not be ``mainstream'' in that we absorb the parameters of the continuation maps into our $\overline{\partial}$-equation.) The main new ingredient here is that we use the previously defined two-form $d(\varphi \eta)$ to define new Hofer-type energies, and then we consider the action properties of the Floer complexes for these different choices. These Hofer-type energies allow us to formulate refined conditions for when compactness holds, that is, when Floer strips have no pieces which escape into to the concave end of the symplectic cobordism. (Ordinary symplectic energy does not capture the pieces of a curve that disappear into the concave end since the symplectic area of these pieces vanishes.) In addition, we formulate condition that ensure a weak form of ``filtered invariance'' of our Floer complexes with respect to our Hofer-type energies, which later will be crucial; see Section \ref{sec:invariance}.

 The aforementioned compactness property is proven in Section \ref{sec:CompactnessProof}, which is somewhat technical Gromov non-bubbling result for strips satisfying a Cauchy-Riemann equation with a Hamiltonian perturbation term. While the result is not surprising, to our knowledge it does not follow from the compactness results which exist in the literature. The strategy is to show that noncompactness implies that a pseudoholomorphic bubble asymptotic to Reeb chords is forming in the concave end; the typical example is a pseudoholomorphic plane or half-disc asymptotic to a contractible orbit or chord. The asymptotic convergence to Reeb chords, together with Stokes' theorem, then enables us to extract some quantum $\hbar>0$ of $d\alpha$-area (i.e.~symplectic area in the contact planes) concentrated inside the concave end, for any sequence of curves in which parts disappear down into the concave end. Here it is crucial that, by Stokes' theorem, the constant $\hbar$ is expressed in terms of the length of the smallest contractible Reeb chord and orbit. 

Sections \ref{sec:usher} and \ref{sec:SplashNeck} discuss several methods to study pseudoholomorphic curves using Reeb chord actions: the version of Usher trick for contact Hamiltonians, as implemented by Shelukhin in \cite{HoferContactomorphism}, together with neck-stretching and ``splashing.'' The splashing construction in this context is analogous to the wrapping considered in \cite{SymplecticEilenbergSteenrod} by Cieliebak--Oancea in order to prove their Mayer--Vietoris long exact sequence for symplectic homology. Our version of neck-stretching is somewhat less technical, as it only consists of a consideration of different Hofer-type energies $d(\varphi\eta)$ for a given pseudoholomorphic curve. When combined with an action argument, this then enables us to exclude the curve from passing through a certain hypersurface. The advantage is that we never need to deform the complex structure by taking an SFT-type neck stretching limit in order to reach our conclusion. We anticipate other applications, as it provides a relatively easy way to decompose a Floer complexes into subcomplexes.

Up until this point, our constructions and definitions apply to arbitrary exact Lagrangian cobordisms between Legendrians. In Section \ref{sec:PushOff} we provide a technical push-off 
 $L_{\epsilon, N,s}$ 
of our Lagrangian cobordism 
 $L$ 
assuming it is a concordance. 
 
The main feature of this push-off is that, for $s=\sigma>0$,  the intersection points $L_{\epsilon,N,s} \cap L$ that are close to the level sets $\{r=-N\}$ and $\{r=N\}$ have a difference of action of magnitude roughly equal to $\sigma.$  This is later used in Proposition \ref{prp:refinedinvariance} to prove that the Floer complex for these two Lagrangians still has a mapping cone structure after turning on a Hamiltonian perturbation of oscillation just slightly less than $\sigma.$

 We prove the main theorem in Section \ref{sec:Proof}, putting all the previous sections together to compute various Floer differentials and chain maps for the Lagrangian concordance and its push-off, as well as for other related pairs of Lagrangians. We outline this proof below.

%
First we show in Section \ref{sec:FloerHomologyPush} that the push-off of the cobordism constructed in Section \ref{sec:PushOff} has a trivial Floer complex. We also show that the complex corresponds to the following Morse cohomology complex (i.e.~with a differential counting \emph{positive} gradient flow lines): Start with the Morse function on the trivial cobordism $\R \times \Lambda$ given as the canonical projection to the $\R$-factor. Add two canceling critical manifolds -- one at the top (some positive $\R$-coordinate) corresponding to the convex end of the concordance, and one at the bottom (some negative $\R$-coordinate) corresponding to the concave end of the concordance. Then perform a perturbation from the Morse-Bott to the Morse setting. It is important that the complex at the top and at the bottom are nontrivial; this is indeed the case, since they both compute the homology of $\Lambda$ (possibly with an index shift). In conclusion, the acyclicity of the total complex implies that the differential must be nontrivial, and moreover that there exists rigid gradient flow lines going from the bottom complex to the top complex.

We next ``turn on" a Hamiltonian perturbation $G_t$ in the middle of the concordance. Here $G_t$ is a compactly supported Hamiltonian on $\R \times Y$ constructed, using Section \ref{sec:usher}, out of the (non-compactly supported) cylindrical lift $e^rH_t$ of our contact Hamiltonian of interest. With some care, we ensure that the Lagrangian is cylindrical over the Legendrian $\phi^1_{\alpha,H_t}(\Lambda)$ in a small slice
$$\phi^1_{G_t}(L) \cap ([A+\delta/2,A+3\delta/2]\times Y)=[A+\delta/2,A+3\delta/2] \times \phi^1_{H_t}(\Lambda).$$ 
We splash and neck-stretch as in this cylindrical slice to create a ``barrier'' at $\{A+2\delta/2\} \times Y$ through which no Floer strip can pass. See Section \ref{sec:TurnSplash} for more details.

The total complex remains acyclic, as we show in Proposition \ref{prp:refinedinvariance} and Corollary \ref{cor:invariance} using a partial ``filtered invariance'' of our Floer complex. Here the $e^{A}\|H\|_{\OP{osc}} < \sigma(\alpha_-, \Lambda_-)$ assumption in Theorem \ref{thm:main} plays a significant role. The homological algebra of Section \ref{sec:TurnSplash}, based upon the Mayer-Vietoris type decomposition from Section \ref{sec:SplashNeck}, then shows that applying ``splashing'' to the slice $[A+\delta/2,A+3\delta/2] \times \phi^1_{H_t}(\Lambda)$ must create new generators. These generators are realized by intersections of the splashed Lagrangian with the original cylindrical concordance. Very roughly speaking, the complex on the top and on the bottom contain classes that must be killed in homology (either by \emph{being} the boundary of something or by \emph{having} a nontrivial boundary in the larger complex). Since they cannot kill each other after the splashing and neck-stretching (no strip can pass through the middle of the cobordism), they must now interact with the newly created generators inside $[A+\delta/2,A+3\delta/2]\times Y.$
 The latter intersection points correspond to Reeb chords between $\Lambda$ and $\phi^1_{H_t}(\Lambda)$ starting on either Legendrian, i.e.~analogously to the generators of Rabinowitz Floer homology. Their existence finally gives us the sought number of chords.

\subsection{Relation to existing techniques}
\label{sec:Comparison}
It has come to our attention that the authors in \cite{Positive}, when proving the existence of translated points, use techniques that are similar to some of ours: they use Shelukhin's contact version of Usher's trick when computing a certain oscillatory norm; they introduce several actions on the loop space and study continuation maps that relate chords; and they show that similar a priori bounds on energy can prevent certain bubbling in the sense of symplectic field theory. On the other hand, they work with the action functional used in the definition of Rabinowitz Floer homology, instead of the more classical setup of Floer homology used here.

It should be the case that Rabinowitz Floer homology for Lagrangian submanifolds as defined in \cite{Merry}, \cite{LagrangianRabinowitz} by Merry -- this is a Floer complex with differential counting gradient trajectories of the Rabinowitz action functional -- is a suitable framework also for studying the questions here. In the case when the obstruction contact form $\alpha_-$ and the contact form $\alpha$ used for counting the orbits are taken to coincide, our results also appear very naturally from this perspective, taking the standard invariance properties of this Floer theory into account; see Remark \ref{rmk:rabinowitz}. However, we would like to stress that the full result in the case when $\alpha_-$ and $\alpha$ are related by a symplectic cobordism would require a new form of invariance for Rabinowitz Floer homology, e.g.~one which also allows deformations of the contact form while working within some suitable action range. The invariance result here is proven by carefully controlling our continuation maps via ``splashing'' and ``neck-stretching,'' in order to produce our Mayer--Vietoris sequence \eqref{eq:MVS}. Finally, compared to \cite{Positive} our analysis of SFT bubbling also has more cases to consider due to the Lagrangian boundary condition, which makes the situation more involved.

In addition, we make the technical remark concerning the relations between the ``splashing'' and ``neck-stretching'' that we perform here and the related techniques from \cite{SymplecticEilenbergSteenrod} that aim at the same results. (For instance, they also obtain a Mayer--Vietoris sequence in a similar setting.) The approach taken here, however, is more basic since we never rely on the full SFT compactness theorem for excluding the existence of Floer strips, but rather attain this by mere means of action computations.

\section{Basic definitions}
\label{sec:basics}

\subsection{Symplectic geometry}
A {\bf symplectic manifold} $(X,\omega)$ is a smooth $2n$-dimensional manifold $X$ together with a closed non-degenerate two-form $\omega$. A symplectic manifold is called {\bf exact} given that the symplectic form exact, i.e.~$\omega=d\eta$ for some one-form $\eta.$ An $n$-dimensional submanifold $L \subset (X,\omega)$ is called {\bf Lagrangian} if $\omega|_{TL}$ vanishes and, given that $\omega=d\eta$ is exact with a choice of primitive $\eta,$ it is called {\bf exact Lagrangian} if $\eta|_{TL}$ is an exact one-form on $L$.

A (time-dependent) Hamiltonian $H \colon X \times [0,1] \to \R$, usually written as $H_t \colon X \to \R$ where $t \in [0,1]$, gives rise to the so-called Hamiltonian vector field $X_{H_t}$ via Hamilton's equations
\[ \omega(\cdot,X_{H_t})=dH_t(\cdot).\]
The corresponding Hamiltonian flow $\phi^t_{H_t} \colon (X,\omega) \to (X,\omega)$ with infinitesimal generator $X_{H_t}$ preserves the symplectic form.

For an exact symplectic manifold $(X,d\eta)$ with a choice of primitive $\eta$ of the symplectic form, recall the definition of the {\bf Liouville vector field} $\zeta$, which is determined uniquely by $\iota_\zeta d\eta=\eta$.

By a {\bf Liouville manifold} $(P,d\theta)$ we mean an open exact symplectic manifold satisfying the following properties. The Liouville vector field $\zeta$ is transverse and outward-pointing to the boundary of a smooth compact domain $\overline{P} \subset P,$ such that $\zeta$ moreover defines a complete (forward-time) flow in the subset $P \setminus \mathrm{int} \overline{P}.$ The compact domain $(\overline{P},d\theta)$ is called a {\bf Liouville domain}.

\subsection{Contact geometry}
Recall that a {\bf contact manifold} is a smooth $(2n+1)$-dimensional manifold $(Y,\xi)$ with a maximally non-degenerate hyperplane distribution $\xi \subset TY$ called the {\bf contact distribution}. A {\bf Legendrian submanifold} is a smooth $n$-dimensional submanifold $\Lambda \subset (Y,\xi)$ for which $T\Lambda \subset \xi$.

For us all contact manifolds will be assumed to have {\bf coorientable} contact distributions, which is equivalent to the existence of a {\bf contact form} $\alpha \in \Omega^1(Y)$ satisfying $\xi=\ker \alpha$. A choice of contact form determines the {\bf Reeb vector field} $R_\alpha$ on $Y$ via the equations
\[ \iota_{R_\alpha} d\alpha=0 \:\:\: \text{and} \:\:\: \iota_{R_\alpha}(\alpha)=1.\]
This vector field then gives rise to the {\bf Reeb flow} $\phi^t_{R_\alpha} \colon (Y,\alpha)\to (Y,\alpha)$, which can be seen to preserve $\alpha$.

We assume Legendrian submanifolds are closed unless stated otherwise. We also assume that the contact manifold $(Y,\xi)$ is closed, or that it is the {\bf contactization} of a Liouville manifold $(P,d\theta)$, i.e.
\begin{gather*}
(P \times \R, \ker \alpha_{\OP{std}}),\:\: \alpha_{\OP{std}}:=dz+\theta,
\end{gather*}
where $z$ denotes the coordinate of the $\R$-factor. Observe that the canonical contact form $\alpha_{\OP{std}}$ on the contactization induces the Reeb vector field $R_{\alpha_{\OP{std}}}=\partial_z$. When considering a more general contact form $\alpha$ on a contactization, we will always assume that $\alpha=\alpha_{\OP{std}}$ holds outside of a compact subset.

Periodic solutions to the Reeb vector field are called {\bf periodic Reeb orbits}. 
Each periodic Reeb orbit $\gamma$ has a {\bf length} given by \[ \ell(\gamma):=\int_\gamma \alpha >0.\]

A non-trivial integral curve of $R_\alpha$ having end-points on an embedded Legendrian submanifold $\Lambda$ is called a {\bf Reeb chord on $\Lambda$}, and its length is defined in the same way as that of a periodic Reeb orbit. 
We denote $\mathcal{Q}_\alpha(\Lambda)$ to be the set of Reeb chords on $\Lambda$ and Reeb periodic orbits, and $\mathcal{Q}^0_\alpha(\Lambda) \subset \mathcal{Q}_\alpha(\Lambda)$ to be the subset consisting of those Reeb chords and periodic orbits which define trivial elements in $\pi_1(Y,\Lambda)$.

\begin{ex}
The jet space $J^1M=T^*M\times \R$ has a canonical contact form $\alpha_{\OP{std}}:=dz+\theta_M$, where $\theta_M=p\,dq$ is the so-called Liouville form on $T^*M$ and $z$ is the standard coordinate on the $\R$-factor. This is also an example of a contactization of a Liouville manifold. The zero-section $0_M \subset J^1M$ is a Legendrian submanifold. Observe that $R_{\alpha_{\OP{std}}}=\partial_z$ and, hence, that 
$\mathcal{Q}_{\alpha_{\OP{std}}}(0_M)=\emptyset$.
\end{ex}
\begin{rmk} While two Legendrian submanifolds generically are \emph{disjoint}, there are typically (a discrete space of) Reeb chords with endpoints on them. In the case of $(J^1M,\alpha_{\OP{std}})$, two Reeb chords between $\Lambda_0, \Lambda_1 \subset J^1M$ correspond bijectively to intersection points of their images under the canonical projection $\pi\colon J^1M \to T^*M$. Observe that $\pi(\Lambda_i) \subset (T^*M,d\theta_M)$, $i=0,1$, are exact Lagrangian immersions, i.e.~immersions for which the pull-back of the one-form $\theta_M$ is exact.
\end{rmk}
The above remark relates the phenomenon of intersection points in symplectic geometry with that of Reeb chords in contact geometry. The following passage from a contact manifold to a cylindrical symplectic manifold (and from a Legendrian submanifold to a cylindrical Lagrangian submanifold) will also provide such a correspondence. The exact symplectic manifold
\[ (\R \times Y,d(e^r \alpha)) \]
associated to a contact manifold is called the {\bf symplectization of $(Y,\alpha)$}, where $r$ denotes the standard coordinate on the $\R$-factor. Observe that the cylinder $\R \times \Lambda \subset (\R \times Y,d(e^r \alpha))$ is (exact) Lagrangian if and only if $\Lambda \subset (Y,\alpha)$ is Legendrian.

\subsection{Contact Hamiltonians}
\label{ssec:ContactHamiltonians}

A {\bf contact Hamiltonian} is a smooth function
\[H \colon Y \times [0,1] \to \R,\]
which usually will be considered as a family of functions $H_t \colon Y \to \R$, $t \in [0,1]$. The Hamiltonian $e^r H_t$ on the symplectization can be seen to have a Hamiltonian flow of the form
\[\phi^t_{e^r H_t}(r,x)=(r +\tau^t_{\alpha, H_t}(x),\phi^t_{\alpha, H_t}(x)), \:\: (r,x) \in \R \times Y.\]
The translation $\tau^t_{\alpha, H_t}(x)$ is determined by the so-called conformal factor via the formula
$$(\phi^t_{\alpha,H_t})^*\alpha=e^{-\tau^t_{\alpha,H_t}}\alpha.$$
The contact Hamiltonian can be recovered from the formula
$$ H_t(\phi^t_{\alpha,H_t}(x)) = \alpha\left(\frac{d}{dt}\phi^t_{\alpha,H_t}(x)\right).$$
Observe that $\tau^t_{\alpha,H_t} \colon Y \to \R$ is indefinite for each $t \in [0,1]$, and that it vanishes identically if and only if the contactomorphism preserves the contact form; such contactomorphisms are usually called {\em strict contactomorphisms}. We have
\[ \phi^{g(t)}_{R_\alpha}=\phi^t_{\alpha,\dot{g}(t)},\]
where $\dot{g}(t) \colon Y \to \R$ is seen as family of constant functions, and where the left hand side is the Reeb flow.

A standard result implies that each one-parameter family of contactomorphisms starting at $\id_Y$ is induced by a contact Hamiltonian (see e.g.~\cite{Geiges}). It is clear that the contact Hamiltonian depends on the choice of contact form.

 We say that a (contact) Hamiltonian $G_t$ is {\bf indefinite} if, for each $t \in [0,1],$ its image is \emph{not} contained inside either of the two subsets $(-\infty,0),(0,+\infty) \subset \R.$

We will need the following basic fact.
\begin{lma} \label{lma:factor} A contact isotopy $\phi^t_{\alpha,H_t}$ can be uniquely factorized as
\[ \phi^t_{\alpha,H_t} =\phi^t_{\alpha,c_t}\circ \phi^t_{\alpha,G_t},\]
for the contact Hamiltonians $c_t,G_t \colon Y \to \R$ given by
\begin{eqnarray*}
& & c_t := (\max_Y H_t + \min_Y H_t)/2, \\
& & G_t := H_t \circ \phi^t_{\alpha,c_t} - c_t.
\end{eqnarray*}
In particular, $c_t$ is a family of constant functions and $\phi^t_{\alpha,c_t}=\phi^{\int_0^t c_s ds}_{R_\alpha}$, while $G_t$ is indefinite for each $t \in [0,1]$ and satisfies $\|G_t\|_{\OP{osc}}=\|H_t\|_{\OP{osc}}.$
\end{lma}
\begin{proof}
The claim that $G_t$ is indefinite, as well as the equivalence between the oscillatory norms, is immediate by construction.

The equality $\phi^t_{\alpha,H_t} =\phi^t_{\alpha,c_t}\circ \phi^t_{\alpha,G_t}$ of the two flows can be seen by the standard computation
\begin{eqnarray*}
\lefteqn{\alpha\left(\frac{d}{dt}(\phi^t_{\alpha,c_t}\circ \phi^t_{\alpha,G_t})\right)= } \\
& = & \alpha(c_t R_\alpha(\phi^t_{\alpha,c_t}\circ \phi^t_{\alpha,G_t}))+\alpha\left(D\phi^t_{\alpha,c_t} \left(\frac{d}{dt}\phi^t_{\alpha,G_t}\right) \right) \\
& = & c_t+(\phi^t_{\alpha,c_t})^*\alpha\left(\frac{d}{dt}(\phi^t_{\alpha,G_t})\right) \\
& = & c_t+G_t \circ \phi^t_{\alpha,G_t} \\
& = & H_t \circ (\phi^t_{\alpha,c_t}\circ \phi^t_{\alpha,G_t}),
\end{eqnarray*}
where the fourth equality follows since $(\phi^t_{\alpha,c_t})^*\alpha=\alpha,$ and where the last equality holds by the construction of $G_t=H_t \circ \phi^t_{\alpha,c_t} - c_t.$
\end{proof}

\subsection{Symplectic cobordisms}
\label{sec:SymplecticCobordisms}
Consider a compact exact symplectic manifold $(\overline{X},d\eta)$ with contact boundary $\partial \overline{X}=Y_- \sqcup Y_+$, where $\eta$ restricts to the contact form $\alpha_\pm$ on $Y_\pm$, and where the Liouville vector field determined by $\eta$ is transverse to the boundary, inwards-pointing along $Y_-$, and outwards-pointing along $Y_+$. We will call this a {\bf compact exact symplectic cobordism from $(Y_-,\alpha_-)$ to $(Y_+,\alpha_+)$}. A compact exact symplectic cobordism can be completed by adjoining half symplectizations of the form
$$((-\infty,0] \times Y_-,d(e^r\alpha_-)) \:\:\: \text{and} \:\:\: ([0,+\infty) \times Y_+,d(e^r\alpha_+)),$$
given that we use appropriate coordinates near $\partial \overline{X}$ (see e.g.~the standard symplectic neighborhood theorem in \cite{SympTop}).
\begin{dfn}
The exact symplectic manifold $(X,d\eta)$ together with the choice of embedding $\overline{X} \subset X$, is called a {\bf (complete) exact symplectic cobordism}. The non-compact subsets $(-\infty,0] \times Y_-$ and $[0,+\infty) \times Y_+ \subset X$ are called its {\bf concave} and {\bf convex} cylindrical ends, respectively.
\end{dfn}
Note that the identifications of these cylindrical ends are part of the data of a complete exact symplectic cobordism.
\begin{ex}
\label{ex:cob}
For any two smooth functions $f_\pm \colon Y \to \R$ satisfying $f_-(y) < f_+(y),$ the subset
$$\{ (r,y); \: f_-(y) \le r \le f_+(y) \} \subset (\R \times Y,d(e^r\alpha))$$
of the symplectization is a compact exact symplectic cobordism from $(Y,e^{f_-}\alpha)$ to $(Y,e^{f_+}\alpha).$ Note that its completion again is symplectomorphic to the symplectization $(\R \times Y,d(e^r\alpha)),$ but that the canonical cylindrical structure provided by this symplectization differs from the ones induced by the data of our symplectic cobordism.
\end{ex}
We will also allow the non-compact case $(\overline{X},d\eta)=([a,b] \times P \times \R,d(e^r\alpha_{\OP{std}}))$ when the contact manifold is a contactization of a Liouville manifold.

Let $(X,d\eta)$ be the completion of $\overline{X} \subset (X,d\eta)$ as above. Observe that $(X,d\eta)$ equivalently can be obtained as the completion of the domain
$$\overline{X}_{T_-,T_+}:= [T_-,0) \times Y_- \:\: \sqcup \:\: \overline{X} \:\: \sqcup \:\: (0,T_+] \times Y_+ \subset (X,d\eta)$$
with smooth contact boundary, given any choices of $T_- \le 0 \le T_+$. The latter domain is a compact exact symplectic cobordism from $(Y_-,e^{T_-}\alpha_-)$ to $(Y_+,e^{T_+}\alpha_+)$.

\subsection{Lagrangian cobordisms and concordances}
\label{sec:LagrangianCobordismsAndConcordances}
Here we develop the notion of both a compact and as well as a complete (typically non-compact) Lagrangian cobordism. First, by a {\bf (compact) Lagrangian cobordism} $\overline{L} \subset (\overline{X},d\eta)$ from the Legendrian submanifold $\Lambda^- \subset (Y_-,\alpha_-)$ to $\Lambda^+ \subset (Y_+,\alpha_+)$ we mean a Lagrangian embedding $\overline{L} \hookrightarrow (\overline{X},d\eta)$ for which the following holds.
\begin{itemize}
\item The boundary satisfies $\partial \overline{L} \subset Y_- \sqcup Y_+ = \partial\overline{X}$, where moreover
\begin{itemize}
\item $\overline{L} \cap Y_\pm = \Lambda^\pm,$ and
\item $\overline{L}$ is invariant under the Liouville flow of $\eta$ near the boundary,
\end{itemize}
are satisfied.
\item The one-form $\eta|_{T\overline{L}}$ has a primitive which is globally constant when restricted to either of $\Lambda^-,\Lambda^+ \subset \partial\overline{L}$. (This condition is empty whenever both of $\Lambda^\pm$ are connected.)
\end{itemize}
In the above situation we will also say that $\Lambda^- \subset (Y,\alpha_-)$ is {\bf Lagrangian cobordant} to $\Lambda^+ \subset (Y_+,\alpha_+)$. A Lagrangian cobordism $\overline{L} \subset (\overline{X},d\eta)$ can be completed to a properly embedded Lagrangian $L \subset (X,d\eta)$ inside the completion of the symplectic cobordism, by adjoining the non-compact cylindrical ends
\begin{eqnarray*}
&& (-\infty,0] \times \Lambda_- \subset ((-\infty,0] \times Y_-,d(e^r\alpha_-)),\\
&& [0,+\infty) \times \Lambda_+ \subset ([0,+\infty) \times Y_+,d(e^r\alpha_+)).
\end{eqnarray*}
If $\overline{L}$ is an exact Lagrangian, then so is the resulting submanifold $L.$ 
\begin{dfn}
\label{dfn:LagCob}
An exact Lagrangian submanifold $L \subset (X,d\eta)$ is called a {\bf (complete) exact Lagrangian cobordism} if there exists $T_- \le 0 \le T_+$ such that $L$ is obtained by completing a compact exact Lagrangian cobordism $\overline{L} \subset (\overline{X}_{T_-,T_+},d\eta)$ from $\Lambda_- \subset (Y_-,e^{T_-}\alpha_-)$ to $\Lambda_+ \subset (Y_+,e^{T_+}\alpha_+)$. The Legendrian submanifolds $\Lambda_\pm \subset (Y_\pm,\alpha_\pm)$ will be called the {\bf $\pm$-ends} of $L$.
\end{dfn}

 Throughout this article, denote $[0,1]$ by $I.$ In the case when $\overline{L}$ is diffeomorphic to $I \times \Lambda,$ we call both $\overline{L}$ and $L$ a {\bf Lagrangian concordance}, and we say that $\Lambda^- \subset (Y,\alpha_-)$ is {\bf Lagrangian concordant} to $\Lambda^+ \subset (Y,\alpha_+)$.

\begin{ex}
\label{ex:conc}
Inside the symplectic cobordism $\{ f_-(y) \le r \le f_+(y) \} \subset (\R \times Y,d(e^r\alpha))$ from Example \ref{ex:cob}, any Lagrangian cylinder $\R \times \Lambda$ intersected with this domain is a Lagrangian concordance.
\end{ex}

With the above definition, there is a Lagrangian cobordism from $\Lambda \subset (Y,\alpha)$ to $\Lambda \subset (Y,e^s \alpha)$ if and only if $s>0.$ In fact, this follows from the basic fact that no exact symplectic cobordism from $(Y,\alpha)$ to $(Y,\alpha)$ exists. (E.g.~such a symplectic cobordism can be used to construct a closed exact symplectic manifold, which cannot exist by Stokes' theorem). However, by convention, we will also prescribe that \emph{any Legendrian submanifold $\Lambda \subset (Y,\alpha)$ is Lagrangian concordant to itself}. It was shown in \cite{LagConc} that a Legendrian isotopy $\Lambda_t \hookrightarrow (Y,\alpha)$, $i \in [0,1]$, (i.e.~an isotopy through Legendrian submanifolds) induces a Lagrangian concordance from $(\Lambda_0,\alpha)$ to $(\Lambda_1,e^C\alpha)$ for some constant $C\ge 0$.

\subsection{Primitives and action of Hamiltonian chords}
\label{sec:action}

Assume that we are given an \emph{ordered} pair of exact Lagrangian cobordisms $L_0,L_1 \subset (X,\eta)$. 

In order to define a Hofer-type symplectic energy and a Hofer-type Floer energy for the pseudoholomorphic strips in Section \ref{sec:Moduli} we will need to specify a one-form on $X$ which is exact when pulled back to $L_i$, $i=0,1.$ First, note that $\eta$ pulls back to an exact one-form on $L_i$ by assumption. It will later be important to exploit the fact that the one-form $\varphi\eta$ pulls back to an exact one-form also for a large class of continuous piecewise smooth functions $\varphi \colon X \to \R_{> 0}$. 

The above energies are important when defining the Floer chain complex for a pair of exact Lagrangian cobordisms. In Section \ref{sec:FloerHomology} we introduce the Floer complex in this setting, associated to a pair of Lagrangian submanifolds together with a compactly supported time-dependent Hamiltonian $G_t \colon X \to \R.$

We also need to study the so-called Floer ``continuation maps,'' between Floer complexes, which provide a morphism from the Floer theory of one set-up to the Floer theory of another. For this reason, we need to generalize our set-up from $G_t,$ a time-dependent Hamiltonian on $(X,d\eta)$ as above, to a one-parameter family
$$G_{s,t} \colon X \to \R, \:\: (s,t) \in \R \times [0,1],$$
of compactly supported time-dependent Hamiltonians parametrized by $s \in \R.$

We are now ready to define the class of functions which will be used for defining the Hofer-type energies associated to the triple $(L_0,L_1,G_{s,t})$.

\begin{dfn}
\label{dfn:varphi}
Let $\mathcal{C}(L_0,L_1,G_{s,t})$ denote the class of continuous piecewise smooth functions $\varphi \colon X \to \R_{> 0}$ satisfying the following properties.
\begin{enumerate}
\item $\varphi|_{\overline{X}}$ is constant;
\item On the cylindrical ends the function $\varphi|_{X \setminus \overline{X}}$ depends only on $r$, and it satisfies
$$\varphi'(r)+ \varphi(r) \ge 0$$ wherever it is differentiable (this ensures that $d(\varphi \eta)$ is non-negative on any $J$-complex two-planes whenever $J$ is cylindrical);
\item $\varphi'(r) \equiv 0$ in some (possibly empty) neighborhood of the (possibly empty) subset
$$\left\{ \begin{array}{l|l} \{r\} \times Y_\pm & (L_0 \cup L_1) \cap (\{r\} \times Y_\pm) \subset (Y_\pm ,\alpha_\pm)\\
& \text{is \underline{not} an embedded Legendrian link} \end{array}\right\} \subset X$$
of the cylindrical ends;
\item On the concave end we either have $\varphi(r)=e^{-r-r_\varphi}$ for some $r_\varphi \ge 0$ or $\varphi(r) \equiv 1$ for all $r \ll 0$ sufficiently small, in which case we set $r_\varphi:=+\infty$; and
\item In the subset
$$\{ x \in X \: | \: G_{s,t}(x) \neq 0 \:\text{for some}\: (s,t) \in \R \times [0,1]\}$$
we have $\varphi(r) \equiv 1$.
\end{enumerate}
 We call such a function $\varphi \in \mathcal{C}(L_0,L_1,G_{s,t})$ a {\bf{Hofer function}}.
\end{dfn}
 Part (3) allows us to have $\varphi' \neq 0$ in any subset $I \times Y_\pm$ inside which both $L_0$ and $L_1$ are traces of the Reeb flow applied to a Legendrian submanifold. This will later be used to ``stretch the neck'' of the symplectic form in such a region. (See Section \ref{sec:SplashNeck}.) Part (5) above is needed in order to prove the energy estimates given by Lemmas \ref{lma:varphienergy} and \ref{lma:action}. Note that, in particular, the constant function $1 \in \mathcal{C}(L_0,L_1,G_{s,t})$ is always contained in the above subset.

The following basic computation will be used multiple times, so we formulate it as a lemma.

\begin{lma}
\label{lma:ActionComp}
 Suppose $L \subset [r_0,r_1] \times Y$ is a Lagrangian submanifold intersecting each $(\{r\} \times Y,\alpha)$ $r \in [r_0,r_1],$ transversely in a Legendrian submanifold.
\begin{enumerate}
\item The submanifold $L$ is of the form $L=\phi^1_{\sigma(r)}([r_0,r_1] \times \Lambda)$ for some fixed Legendrian $\Lambda \subset (Y,\alpha)$ and for some smooth autonomous Hamiltonian $\sigma \colon [r_0,r_1] \to \R$ with Hamiltonian vector field $X_\sigma(r)=e^{-r}\sigma'(r)R_\alpha.$ 
\item The one-form $\varphi(r)e^r\alpha$ pulled back to $L$ has a primitive
$$ f(r)=\int_{r_0}^r \varphi(r)e^r\frac{d}{dr}(e^{-r}\sigma'(r))dr=\int_{r_0}^r \varphi(r)(\sigma''(r)-\sigma'(r))dr $$
on $L \subset [r_0,r_1] \times Y$ depending only on the $r$-coordinate. In particular, if $\varphi(r) \equiv 1,$ then one can take $f(r)=\sigma'(r)-\sigma(r)$ as a primitive.
\end{enumerate}
\end{lma}
\begin{proof}
(1): The Legendrian condition of the slices implies that $e^r\alpha$ pulled back to $L$ is of the form $g(r)dr$ where $g(r_0)=0$ without loss of generality. It can be seen that $\Lambda:=L \cap (\{r_0\} \times Y ),$ and the Hamiltonian $\sigma(r)$ satisfying the ODE $\sigma'(r)-\sigma(r)=g(r),$ $\sigma(0)=0,$ are as needed. To that end, we use e.g.~the Weinstein Lagrangian neighborhood theorem together with the fact that a graphical Lagrangian inside the cotangent bundle is determined uniquely by its potential function.

(2): This calculation is straightforward.
\end{proof}

\begin{lma}
\label{lma:crucial}
The continuous and piecewise smooth one-form $\varphi\eta$ for $\varphi \in \mathcal{C}(L_0,L_1,G_{s,t})$ is exact when pulled back to $L_i,$ $i=0,1.$ The primitive on $L_i$ is moreover locally constant when restricted to any slice $\{r = r_0\}$ contained inside the subset
$$V:=\left\{ \begin{array}{l|l} \{r\} \times Y_\pm & L_i \cap (\{r\} \times Y_\pm) \subset (Y_\pm ,\alpha_\pm)\\
& \text{is a Legendrian submanifold} \end{array}\right\} \subset X$$
of the cylindrical ends. (In regions where $L_i$ is cylindrical the pullback of $\varphi\eta$ obviously vanishes.)
\end{lma}
\begin{proof}
We establish the existence of primitives on each piece $L_i \cap V$ and $L_i \setminus V$ separately, and then proceed to show that these primitives can be combined to form a globally defined primitive.

 The pull-back to $L_i \cap V$ of $ \varphi\eta$ is exact by Part (2) of Lemma \ref{lma:ActionComp} with a primitive of the form $g(r)dr.$ 

The assumptions imply that $\varphi'(r)\equiv 0$ holds inside the level-sets of $r$ contained in $X \setminus V$. 
A primitive of the pull-back of $\varphi\eta$ to $L_i \cap (X \setminus V)$ can thus be taken to be a locally constant rescaling of the primitive of the pull-back of $\eta$ there (the latter primitives exist since $L_i$ are assumed to be exact).

 It now straightforward to combine the primitives defined in the different regions above to a globally defined primitive on $L_i.$
\end{proof}
The above lemma implies that a primitive of $\varphi\eta$ pulled back to $L_i$ is locally constant outside of a compact subset. The following lemma shows that it is possible to assume that this constant vanishes on both of the ends for one of the cobordisms in the pair.
\begin{lma}
\label{lma:equalaction}
Assume that the primitive of $\varphi\eta$ pulled back to $L_i$ which vanishes on the negative end and takes the value $C_i$ on the positive end, for $i=0,1.$ After a Hamiltonian isotopy of the union $\overline{L}_0 \cup \overline{L}_1 \subset (\overline{X},d\eta)$ of Lagrangian cobordisms supported inside the interior $(\overline{X} \setminus \partial \overline{X},d\eta)$, the primitives on the positive end can be assumed to be given by $C_i+C$, $i=0,1,$ for an arbitrary constant $C \in \R,$ while the primitives still vanish on the negative ends.
\end{lma}
\begin{proof}
The claim follows by considering a suitable Hamiltonian diffeomorphism
$\phi^1_\sigma \colon (\overline{X},d\eta) \to (\overline{X},d\eta),$ applied to $\overline{L}_0 \cup \overline{L}_1$, where the autonomous Hamiltonian $\sigma$ has support in some collar neighborhood $((-\epsilon,0] \times Y_+,d(e^r\alpha_+))$ of the boundary $\partial \overline{X} \cong \{0\} \times Y_+.$ Moreover, we may require that
\begin{itemize}
\item $\sigma=\sigma(r)$ only depends on the symplectisation coordinate,
\item $\sigma'(r)=0$ holds in some neighborhood of the boundary.
\end{itemize}
Note that the computation in Part (2) of Lemma \ref{lma:ActionComp} shows that $C=\sigma'(0)-\sigma(0)=-\sigma(0).$

\end{proof}
Fix $s = s_0.$
A Hamiltonian chord $p$ of $\phi^t_{G_{s_0,t}}$ from $L_0$ to $L_1$ is a path
\begin{gather*}
([0,1],\{0\},\{1\}) \to (X,L_0,L_1),\\
t \mapsto \phi^t_{G_{s_0,t}}(x),
\end{gather*}
with starting point given by $x \in L_0$ and endpoint given by $\phi^1_{G_{s_0,t}}(x) \in L_1$. Observe that intersections $L_0 \cap L_1$ are Hamiltonian chords from $L_0$ to $L_1$ for a constant Hamiltonian flow.

Take primitives $f_i \colon L_i \to \R$ and $f^\varphi_i \colon L_i \to \R$ of $\eta$ and $\varphi\eta$ pulled back to $L_i$, respectively, uniquely determined by the requirement that they vanish for all $r \ll 0$ sufficiently small on the concave end.

To a Hamiltonian chord $p$ of $G_{s_0,t}$ starting on $x \in L_0$ and ending on $\phi^1_{G_{s_0,t}}(x) \in L_1$, we associate its so-called {\bf action} defined by
\begin{eqnarray}
\label{eq:a_varphi}
\mathfrak{a}_\varphi(p)& :=& f^\varphi_0(x)-f^\varphi_1(\phi^1_{G_{s_0,t}}(x))+\int_0^1(\eta(\dot{\phi}^t_{G_{s_0,t}}(x))-G_{s_0,t}(\phi^t_{G_{s_0,t}}(x)))dt.
\end{eqnarray}
\begin{rmk}
\label{rmk:action}
The following two claims are straightforward.
\begin{enumerate}
\item In the case when $G_{s_0,t} \equiv 0$, and hence $p(t) \equiv x \in L_0 \cap L_1$ is an intersection point, these actions specialize to the potential differences
\begin{eqnarray*}
\mathfrak{a}_\varphi(p)&:=&f^\varphi_0(x)-f^\varphi_1(x)
\end{eqnarray*}
at the intersection. When $\varphi \equiv 1$ we also write
$$\mathfrak{a}(p):=\mathfrak{a}_1(p)$$
for the induced action.
\item If $p \in L_0 \cap L_1$ is an intersection point contained outside of the support of $G_{s_0,t}$ for all $t \in [0,1],$ then it can be considered either as a $0$-Hamiltonian or a $G_{s_0,t}$-Hamiltonian chord. We claim that, for any fixed choice of
$$\varphi \in \mathcal{C}(L_0,L_1,G_{s,t}) \subset \mathcal{C}(L_0,L_1,0),$$
the action of this intersection point obtained in the two different cases coincide.
\end{enumerate}
\end{rmk}

\section{The Floer homology for a pair of Lagrangian cobordisms}
\label{sec:FloerHomology}

In this section we introduce the Floer chain complex whose differential is defined with a Hamiltonian perturbation term, and the so-called continuation chain map between such Floer chain complexes needed for proving invariance. Floer homology was originally introduced by Floer in \cite{Floer:MorseTheoryLagrangian} for a pair of compact Lagrangian submanifolds, and has since then seen a lot of development. We refer to \cite{OhSympTop1} and \cite{OhSympTop2} for a thorough and modern treatment, including the incorporation of the Hamiltonian term in Section 14. In this section we describe a version of Floer homology in the present context, i.e.~for a pair of non-compact exact Lagrangian cobordisms inside a symplectic cobordism with a concave end. We also refer to \cite{Cthulhu} for a previous construction of Floer homology in a similar setting.

In the following we let $L_0,L_1 \subset (X,d\eta)$ be complete exact Lagrangian cobordisms of a complete exact symplectic cobordism. We denote by $\Lambda_i^\pm \subset (Y_\pm,\alpha_\pm)$, $i=0,1$, the Legendrian submanifolds being the $\pm$-ends of $L_i$. We also fix a choice of a one-parameter family $G_{s,t} \colon X \to \R$ of compactly supported and time-dependent Hamiltonians. The most general case we consider, when defining continuation maps between Floer complexes, is
\begin{equation}
\label{eq:Gst}
G_{s,t} = \rho(s) G_t, \mbox{where} \,\,\rho(s): \R \rightarrow [0,1] \,\, \mbox{and} \,\, \mbox{supp}(\rho') \,\, \mbox{is compact.}
\end{equation}
We can thus write $G_{+,t}$ and $G_{-,t}$ to denote the Hamiltonian $G_{s_0,t}$ with $s_0 \gg 0$ and $s_0 \ll 0$, respectively. 

\subsection{Admissible almost complex structures}
\label{sec:AdmissibleACS}

 An $(s,t)$-dependent almost complex structure $J=J_{s,t}$, $(s,t) \in \R \times [0,1]$ on a symplectic manifold $(X,\omega)$, i.e.~a smooth one-parameter family of time-dependent almost complex structures parametrised by $s \in \R$, is said to be {\bf tamed by $\omega$} if $\omega(v,Jv)>0$ holds whenever $v \neq 0$. 
We assume that for all families $J_{s,t}$ there exists some $K \ge 0$ such that 
\begin{equation}
\label{eq:Jst}
J_{s,t} = J_{\pm, t} \quad \mbox{if} \,\, \pm s \ge K,
\end{equation}
 for the tame $t$-dependent almost complex structures $J_{\pm, t}.$ An almost complex structure $J$ on a symplectization
\[ (\R \times Y,d(e^r \alpha))\]
is said to be {\bf cylindrical} if
\begin{itemize}
\item $J\partial_r=R_\alpha$,
\item $J \xi =\xi$, where $\xi := \ker\alpha \subset TY,$ and $J|_\xi$ is tamed by $d\alpha,$ and
\item $J$ is invariant under translations of the coordinate $r$.
\end{itemize}
It automatically follows that a cylindrical almost complex structure is tame.

Consider a choice of function $\varphi \colon X \to \R$, $\varphi \in \mathcal{C}(L_0,L_1,G_{s,t})$, as defined in Section \ref{sec:action}. We call a tame almost complex structure on $(X,d\eta)$ {\bf admissible with respect to $\varphi$} if it coincides with a cylindrical almost complex structure in each component of
$$\{ r ; \: \varphi'(r) \neq 0 \} \subset X \setminus \overline{X}$$
in the cylindrical ends. The space of admissible almost complex structures will be denoted by
$$\mathcal{J}(X,d\eta,\varphi).$$
Observe that this is a non-empty and contractible space by \cite{Gromov}.

In the next section we will see that an almost complex structure which is admissible with respect to $\varphi$ satisfies the property that $d(\varphi\eta)$ is non-negative on the corresponding complex tangent planes. Loosely speaking, this means that the piecewise smooth two-form $d(\varphi\eta)$ behaves like a symplectic form when using it to define the energy of a pseudoholomorphic curve.

\subsection{Moduli spaces of pseudoholomorphic strips and their energy estimates}
\label{sec:Moduli}

We define here the pseudoholomorphic curves which we will need in the next subsection for our maps (differentials and chain maps). We also define and relate a number of different energies of these curves expressed in terms of the actions introduced in Section \ref{sec:action}. All definitions and results here are standard, with only one minor variation: the energies are induced by the two-form $d(\varphi\eta)$ which is not necessarily a symplectic form everywhere.

Let $p_\pm(t)$ be $G_{\pm,t}$-Hamiltonian chords from $L_0$ to $L_1$. (Recall in the case when $G_{\pm,t} \equiv 0$ these are intersection points $L_0 \cap L_1$.) Define the moduli-space of pseudoholomorphic strips
\[\mathcal{M}_{p_+,p_-}(L_0,L_1;G_{s,t})\]
to be the set of smooth maps $u \colon \R \times [0,1] \to X$ satisfying
\begin{equation}
\label{eq:strip}
\begin{cases}
u(s,0) \in L_0, \quad u(s,1) \in L_1, \quad \lim_{s \to \pm \infty}u(s,t)=p_\pm(t), \\
\partial_s u(s,t)+J_{s,t}(\partial_t u(s,t)-X_{G_{s,t}}(u(s,t)))=0,
\end{cases}
\end{equation}
 where the above limits are 
 uniform in $t.$ 
 This last condition we state as: $u$ satisfies the Cauchy-Riemann equation with a Hamiltonian perturbation term. However, for short these solutions will often be referred to simply as {\bf pseudoholomorphic strips}. The chord $p_-$ will be called the {\bf input} while $p_+$ will be called the {\bf output}.

The {\bf Floer energy} of a strip is the quantity
\[ E_{d(\varphi\eta),J_{s,t}}(u):=\int_{-\infty}^\infty \int_0^1 d(\varphi\eta)(\partial_s u(s,t),J_{s,t} \partial_s u(s,t))\, dt \,ds\ge 0.\]

This quantity is a priori non-negative on any strip, since $d(\varphi\eta)$ is non-negative on any $J_{s,t}$-complex tangency. Here it is crucial that $J_{s,t}$ is cylindrical wherever $\varphi'(r) \neq 0,$ and that $\varphi(r)+\varphi'(r) \ge 0$; see Definition \ref{dfn:varphi} and Section \ref{sec:AdmissibleACS}.

We define the {\bf $d(\varphi\eta)$-energy} to be given by
\[ E_{d(\varphi\eta)}(u):=\int_u d(\varphi\eta).\]
Stokes' theorem together with the exactness of $L_i,$ $i=0,1,$ implies that this quantity only depends on the asymptotics of the strip; more precisely, we have
\begin{equation}
\label{eq:sympen}
E_{d(\varphi\eta)}(u)= \mathfrak{a}_\varphi(p_+)-\mathfrak{a}_\varphi(p_-) + \int_0^1G_{+,t}(p_+(t))dt-\int_0^1G_{-,t}(p_-(t))dt.
\end{equation}
 Note that $\varphi$ only is piecewise smooth, but that Stokes' theorem still applies since $\varphi$ is continuous.

The obvious generalizations of the above formulas also enable us to consider the different energies in the case when the map $u$ only is defined on a subset of the strip $\R \times [0,1].$ 

In the case when $G_{s,t} \equiv G_t$ only depends on the $t$-coordinate we have the following precise expressions for the Floer energy which is standard; see e.g.~\cite[Section 12.3]{OhSympTop2}.
\begin{lma}
\label{lma:varphienergy}
Consider $\varphi \in \mathcal{C}(L_0,L_1,G_t)$ and $J_{s,t} \in \mathcal{J}(X,d\eta,\varphi)$, and a strip $u \in \mathcal{M}_{p_+,p_-}(L_0,L_1;G_t)$. The energy can be expressed as
\[E_{d(\varphi\eta),J_{s,t}}(u)=\mathfrak{a}_\varphi(p_+)-\mathfrak{a}_\varphi(p_-) \ge 0,\]
with equality if and only if $u$ is contained in a single Hamiltonian chord.
\end{lma}
\begin{proof}
Since $u$ satisfies the Cauchy-Riemann equation with Hamiltonian perturbation term, we compute
\begin{eqnarray*}
\lefteqn{E_{d(\varphi\eta),J_{s,t}}(u) =}\\
& =& \int_u d(\varphi\eta) + \int_{-\infty}^{+\infty} \int_0^1 d(\varphi\eta)(\partial_s u(s,t),-X_{G_t}(u(s,t)))dt\,ds \\
& =& \int_u d(\varphi\eta) - \int_0^1 \int_{-\infty}^{+\infty} \partial_sG_t(u(s,t)) ds\,dt,
\end{eqnarray*}
where we have used the property $\varphi \in \mathcal{C}(L_0,L_1,G_t)$ in order to infer that
$$d(\varphi\eta)(\cdot,-X_{G_t}(u(s,t)))=d\eta(\cdot,-X_{G_t}(u(s,t)))=-d_{u(s,t)}G_t(\cdot)$$
is satisfied.

The expressions of the energies in terms of the actions now follow by elementary applications of Stokes' theorem together with Equality \eqref{eq:sympen}.

The assumptions on $\varphi$ and $J_{s,t}$ imply that the energies are non-negative; here we have used the assumption that $d(\varphi\eta)$ is non-negative on any $J_{s,t}$-complex line. Since, moreover, $d(\varphi\eta)=Cd\eta$ is a \emph{symplectic} form near the Hamiltonian chords by the assumption $\varphi \in \mathcal{C}(L_0,L_1,G_t),$ it follows that a non-trivial pseudoholomorphic strip in fact must have \emph{positive} energy.
\end{proof}

The following lemma gives an action estimate in the case when the family $G_{s,t}$ of Hamiltonians is allowed to depend on $s \in \R$ in a very controlled way. These estimates are straightforward adaptations of the estimates in \cite[Section 14.4]{OhSympTop2} to the current setting. Take $G_{s,t}=\rho(s)G_t$ for some smooth $\rho \colon \R \to [0,1],$ and let $p_+,p_-$ denote a Hamiltonian chord of $G_{+,t}$ and $G_{-,t},$ respectively. Moreover, we assume that $\rho'(s)$ has compact support and satisfies the property that each of the integrals $\int_{(\rho')^{-1}([0,+\infty))} \rho'(s) ds$ and $\int_{(\rho')^{-1}((-\infty,0])} \rho'(s) ds$ are equal to either $0$ or $\pm 1$. Again, we take $\varphi \in \mathcal{C}(L_0,L_1,G_{s,t})$ and $J_{s,t} \in \mathcal{J}(X,d\eta,\varphi)$.

Consider a strip $u \in \mathcal{M}_{p_+,p_-}(L_0,L_1;G_{s,t})$ together with a (possibly empty) open subset $U \subset \R \times [0,1]$ satisfying the conditions that:
\begin{enumerate}[label=(C.\arabic{*}), ref=(C.\arabic{*})]
\item The subset $U$ is disjoint from $[T,+\infty) \times [0,1]$ for some number $T \gg 0;$
\item \label{C.2} There is a number $\mathfrak{a}^-_\varphi \in \R$ determined as follows: we either require $U$ to be precompact, in which case we set
$$\mathfrak{a}^-_\varphi:=\mathfrak{a}_\varphi(p_-),$$
or that it contains a subset of the form $(-\infty,T] \times [0,1]$ for some $T \ll 0,$ in which case we set
$$\mathfrak{a}^-_\varphi:=\mathfrak{a}_\varphi(p_-)+\int_0^1 G_{-,t}(p_-(t)) dt;$$
\item $G_{s,t} \circ u$ vanishes in some neighborhood of $\overline{U} \setminus U \subset \R \times [0,1],$ where we note that this subset is compact by the previous assumptions.
\end{enumerate}

\begin{lma}
\label{lma:action}
Let the strip $u \in \mathcal{M}_{p_+,p_-}(L_0,L_1;G_{s,t})$ be a solution of the perturbed Cauchy-Riemann equation (\ref{eq:strip}) and consider a (possibly empty) subset $U \subset \R \times [0,1],$ where both $G_{s,t} = \rho(s) G_t$ and $U$ are assumed to satisfy the properties above.
\begin{enumerate}
\item[(0)] If $\rho(s) \equiv 1$ or $\rho(s) \equiv 0$ then
$$ 0 \le E_{d(\varphi\eta),J_{s,t}}(u|_{\R \times [0,1] \setminus U})\le \mathfrak{a}_\varphi(p_+)-\mathfrak{a}^-_\varphi-E_{d(\varphi\eta)}(u|_U);$$
\item[(1)] If $\rho(s)$ is non-constant and $\rho(s)=0$ (resp. $\rho(s)=1$) whenever $|s| \gg 0$ is sufficiently large, i.e.~$G_{-, t} \equiv G_{+,t} \equiv 0$ (resp. $G_{-, t} \equiv G_{+,t}\equiv G_t$) then
$$ 0 \le E_{d(\varphi\eta),J_{s,t}}(u|_{\R \times [0,1] \setminus U})\le \mathfrak{a}_\varphi(p_+)-\mathfrak{a}^-_\varphi-E_{d(\varphi\eta)}(u|_U)+ \|G_t\|_{\OP{osc}};$$
\item[(2)] If $\rho(s) = 0$ for $s \ll 0$, $\rho(s)= 1$ for $s \gg 0$, and $\rho'(s) \ge 0$, then
$$0 \le E_{d(\varphi\eta),J_{s,t}}(u|_{\R \times [0,1] \setminus U}) \le \mathfrak{a}_\varphi(p_+)-\mathfrak{a}_\varphi(p_-)-E_{d(\varphi\eta)}(u|_U)+\int_0^1 \max_X G_t dt;$$
\item [(3)]
If $\rho(s) = 1$ for $s \ll 0$, $\rho(s)= 0$ for $s \gg 0$, and $\rho'(s) \le 0$, then
$$0 \le E_{d(\varphi\eta),J_{s,t}}(u|_{\R \times [0,1] \setminus U}) \le \mathfrak{a}_\varphi(p_+)-\mathfrak{a}^-_\varphi-E_{d(\varphi\eta)}(u|_U)-\int_0^1 \min_X G_t dt;$$
\end{enumerate}
\end{lma}
\begin{rmk} The reason for why we need the energy estimates for the complicated domains in Lemma \ref{lma:action} is that we have not established the full symplectic field theory (SFT for short) type compactness result for Floer strips with a Hamiltonian perturbation term as considered here. Ideally, it should be possible to obtain a compactness result which, given that a sequence of strips eventually leaves every compact subset of the target space, extracts a limit ``building''. These buildings are supposed to consist of several levels of curves, defined with or without Hamiltonian perturbation terms, having punctures (in the interior as well as on the boundary) asymptotic to Reeb chords and orbits. In contrast, here we only establish Corollary \ref{cor:RefinedChordExistence} which, roughly speaking, exhibits the behavior of such a strip just prior to the moment when a breaking of SFT type occurs.
\end{rmk}
\begin{proof}
Similar to the proof of Lemma \ref{lma:varphienergy} we compute
\begin{eqnarray*}
\lefteqn{E_{d(\varphi\eta),J_{s,t}}(u|_{\R \times [0,1] \setminus U}) =}\\
& =& E_{d(\varphi\eta)}(u|_{\R \times [0,1]}) - E_{d(\varphi\eta)}(u|_U) + \\
& + & \int_{u|_{\R \times [0,1] \setminus U}} \rho(s)d(\varphi\eta)(\partial_s u(s,t),-X_{G_t}(u(s,t))) dt\,ds.
\end{eqnarray*}
Using the assumption that $G_t \circ u$ vanishes in a neighborhood of the boundary of $U$, the latter term can be computed to be equal to 
\begin{eqnarray*}
\lefteqn{\int_{-\infty}^{+\infty}\int_0^1 \chi(s,t)\rho(s)d(\varphi\eta)(\partial_s u(s,t),-X_{G_t}(u(s,t))) dt\,ds = }\\
& = & - \int_0^1 \int_{-\infty}^{+\infty} \chi(s,t)\rho(s) \partial_s G_t(u(s,t)) ds\,dt\\
& =& - \int_0^1\left[\rho(s) \chi(s,t)G_t(u(s,t))\right]_{s=-\infty}^{+\infty} dt + \\
&+& \int_0^1 \int_{-\infty}^{+\infty}\chi(s,t)\rho'(s)G_t(u(s,t)) ds\,dt + \\
& + & \int_0^1 \int_{-\infty}^{+\infty}\rho(s)G_t(u(s,t))\partial_s\chi(s,t) ds\,dt
\end{eqnarray*}
for some smooth bump function $\chi \colon \R \times [0,1] \to [0,1]$ 
that is equal to one when restricted to $\R \times [0,1] \setminus U,$ and which vanishes on $\OP{supp}(G_t \circ u) \cap U.$ Here we use the fact that there exists disjoint open neighborhoods of $\R \times [0,1] \setminus U$ and $\OP{supp}G_t(u(s,t)) \cap U;$ this is a consequence of the assumptions made on $U.$

In particular, since $\partial_s\chi(s,t)$ vanishes in the subset $\OP{supp}(G_t \circ u) \cap U,$ we conclude the vanishing
$$\int_0^1 \int_{-\infty}^{+\infty}\rho(s)G_t(u(s,t))\partial_s\chi(s,t) ds\,dt=0$$
of the last term in the above expression. Also, note that $\chi(s,t)\equiv 1$ for $s \gg 0$ by construction.

Finally, we use the fact that
$$\max_{\R \times [0,1]} \chi(s,t) G_t(u(s,t)) \le \max_{\R \times [0,1]} G_t(u(s,t)),$$
$$\min_{\R \times [0,1]} \chi(s,t) G_t(u(s,t)) \ge \min_{\R \times [0,1]} G_t(u(s,t)).$$
In combination with the expression of $E_{d(\varphi\eta)}(u|_{\R \times [0,1]})$ given in Equality \eqref{eq:sympen}, all estimates can now be seen to follow.
\end{proof}

\subsection{The boundary map and chain maps}
\label{sec:LagIntHF}
We are now ready to define the Floer complexes along with their boundary maps, as well as continuation maps between them. Assume that we are given exact Lagrangian cobordisms $L_0, L_1 \subset (X,d\eta)$ that are disjoint outside of a compact subset. Furthermore, we consider a compactly supported Hamiltonian diffeomorphism $\phi^1_{G_t} \colon (X,\omega) \to (X,\omega)$ for which the intersection $\phi^1_{G_t}(L_0) \pitchfork L_1$ is transverse.

For any fixed component $\mathfrak{p} \in \pi_0(\Pi(X;L_0,L_1))$ of the space of paths from $L_0$ and $L_1$ in $X$, we define the graded and finite-dimensional vector space
\[ CF_*^{\mathfrak{p}}(L_0,L_1;G_{t}) := \Z_2 \left\langle \begin{array}{l|l} p(t) & p(t)=\phi^t_{G_t}(x),\: t \in [0,1],\\
& p(0) \in L_0, p(1) \in L_1, [p(t)] \in \mathfrak{p}, \end{array}\right\rangle\]
spanned by chords of $\phi^t_{G_t}$ from $L_0$ to $L_1$ in class $\mathfrak{p}$. Up to a global shift, this grading is well-defined modulo the greatest common divisor of the Maslov numbers of the two cobordisms; we refer to \cite{Floer:RelativeMorseIndex} for more details. For our purposes the grading will not play any role.

Given two numbers $m_- \le m_+$ we also define
\[ CF_*^{\mathfrak{p}}(L_0,L_1;G_{t};\varphi)^{m_+}_{m_-} := \Z_2 \left\langle \begin{array}{l|l} & p(t)=\phi^t_{G_t}(x), \: t \in [0,1],\\
p(t) & p(0) \in L_0, p(1) \in L_1, [p(t)] \in \mathfrak{p}, \\
& m_- \le \mathfrak{a}_\varphi(p(t)) 
 < 
m_+, \end{array}\right\rangle\]
which is a vector subspace of $CF_*^{\mathfrak{p}}(L_0,L_1;G_{t})$. It is important to observe that this subspace depends on the choice of function $\varphi \colon X \to \R$ used when defining the action. When the function $\varphi$ is clear from the context, we will sometimes use the notation
$$CF_*^{\mathfrak{p}}(L_0,L_1;G_{t})^{m_+}_{m_-}=CF_*^{\mathfrak{p}}(L_0,L_1;G_{t};\varphi)^{m_+}_{m_-},$$
omitting this choice.

Under certain additional assumptions deferred to Section \ref{sec:invariance} below, we can define the following linear maps.

{\bf{The differential}}: This is a linear map
$$d: CF_*^{\mathfrak{p}}(L_0,L_1;G_{t}) \rightarrow CF_{*-1}^{\mathfrak{p}}(L_0,L_1;G_{t}) $$
that on a generator $q$ is defined via the $\Z_2$-count
\[ d(q) :=\sum_{\dim(\mathcal{M}_{p,q}(L_0,L_1;J,G_t)/\R)=0} \#_2\left(\mathcal{M}_{p,q}(L_0,L_1;G_t)/\R\right) p.\]
Here we have used a fixed almost complex structure $J=J_t$ only depending on the $t$-coordinate, and $\mathcal{M}_{p,q}(L_0,L_1;G_t)/\R$ denotes the solutions up to the action of translation of the $s$-coordinate (this action obviously preserves the solutions). Using Lemma \ref{lma:varphienergy} we see that $\langle d(q),p \rangle \neq 0$ being nonzero implies that $\mathfrak{a}_\varphi(p)>\mathfrak{a}_\varphi(q)$. In other words, our differential \emph{increases} the action. This means that the differential descends to a differential
$$ d^{m_+} \colon CF_*^{\mathfrak{p}}(L_0,L_1;G_{t})^{m_+}_{m_-} \to CF_{*-1}^{\mathfrak{p}}(L_0,L_1;G_{t})^{m_+}_{m_-}$$
in the following manner: consider the differential $d^{m_+}$ induced on the quotient space
$$CF_*^{\mathfrak{p}}(L_0,L_1;G_{t})^{m_+}_{-\infty}=CF_*^{\mathfrak{p}}(L_0,L_1;G_{t})/CF_*^{\mathfrak{p}}(L_0,L_1;G_{t})^{+\infty}_{m_+}$$
and then take its restriction to the subspace
$$CF_*^{\mathfrak{p}}(L_0,L_1;G_{t})^{m_+}_{m_-} \subset CF_*^{\mathfrak{p}}(L_0,L_1;G_{t})^{m_+}_{-\infty}=CF_*^{\mathfrak{p}}(L_0,L_1;G_{t})/CF_*^{\mathfrak{p}}(L_0,L_1;G_{t})^{+\infty}_{m_+}.$$

{\bf{The continuation maps}}: Given a one-parameter family of Hamiltonians $G_{s,t}$ of the form prescribed by Lemma \ref{lma:action}, the induced continuation map
$$\Phi_{G_{-,t}, G_{+,t}}: CF_*(L_0,L_1;G_{-,t}) \rightarrow CF_*(L_0,L_1;G_{+,t})
$$
is defined on a generator $p_-$ of $CF_*(L_0,L_1;G_{-,t})$ via the $\Z_2$-count
\[ \Phi_{G_{-,t}, G_{+,t}}(p_-) :=\sum_{\dim \mathcal{M}_{p_+,p_-}(L_0,L_1;G_{s,t})=0} \#_2\left(\mathcal{M}_{p_+,p_-}(L_0,L_1;G_{s,t})\right) p_+.\]
 Lemma \ref{lma:action} above determines its behavior with respect to the action filtration. Observe that a continuation map does not necessarily increase the action.

The following is the transversality result that we need in order to define the above counts, which is standard. The compactness properties will be dealt with in Section \ref{sec:invariance}.
\begin{prp}
\label{prp:Transversality}
For generic $J_{s,t} \in \mathcal{J}(X,d\eta, \varphi)$ satisfying (\ref{eq:Jst}) the moduli spaces
$$ \bigcup_{\kappa \in \R} \mathcal{M}_{p_+,p_-}(L_0,L_1;\rho_\kappa(s)G_t),$$
where $\rho_\kappa(s)$ depends smoothly on $\kappa \in \R$ and where $\rho_\kappa'(s) \equiv 0$ outside of a compact subset (allowed to depend on $\kappa$), are all transversely cut out. Moreover, the same is true for the moduli spaces $$ \mathcal{M}_{p_+,p_-}(L_0,L_1;G_t)$$ for generic choices of $J_{t} \in \mathcal{J}(X,d\eta, \varphi)$ only depending on $t \in[0,1].$

Furthermore, transversality can be achieved by a perturbation of $J_t$ supported in some arbitrarily small precompact neighborhood of the Hamiltonian chords.
\end{prp}
\begin{proof}
See e.g.~\cite[Proposition 15.5]{OhSympTop2} for the $t$-dependence case, but also \cite[Section 8.6]{AudinDamian} for the analogous result in the closed case (which is similar).

The only note to make is that our almost complex structure must remain admissible after the perturbation (see Section \ref{sec:AdmissibleACS}). To achieve this, note that in Definition \ref{dfn:varphi}
we can consider arbitrary time-dependent perturbations of the tame almost complex structure in some neighborhood of the Hamiltonian chords; clearly the strip must enter this neighborhood by the assumption made on its asymptotics.

Finally, observe that the case when $J$ depends on both $(s,t)$ as opposed to just $t$ is considerably easier, since the ``some-arc-injective'' (weaker than ``somewhere-injective'' which is used when $J$ domain-independent) condition is not needed.

\end{proof}

\subsection{A condition for non-bubbling}
\label{sec:compactness}

A main technical point of this paper establishes conditions for when $\left(CF_*^{\mathfrak{p}}(L_0,L_1;J,G_t)^{m_+}_{m_-},\partial\right)$ is a complex and when the continuation map $\Phi_{G_{-,t}, G_{+,t}}$ is well-defined. To this end, Theorem \ref{thm:compactness} below is the only new ingredient needed in our Floer theory set-up, which gives a condition for when strips are confined to some given compact subset.

In this section we will consider a fixed Hofer function $\varphi \in \mathcal{C}(L_0,L_1,G_t)$ for action computations, together with an associated constant $0\le r_\varphi \le +\infty.$ 
Let the set $\mathcal{Q}_{\alpha_-}^{\mathfrak{p}}$ consist of all contractible periodic $(\alpha_-)$-Reeb orbits and all $(\alpha_-)$-Reeb chords having 
\begin{itemize}
\item both ends on $\Lambda^-_0,$ while defining the trivial element in $\pi_1(Y_-,\Lambda^-_0),$
\item both ends on $\Lambda^-_1,$ while defining the trivial element in $\pi_1(Y_-,\Lambda^-_1),$ or
\item starting point on $\Lambda^-_0$ and endpoint on $\Lambda^-_1,$ living in the component $\mathfrak{p} \in \pi_0(\Pi(X,L_0,L_1)).$
\end{itemize}
We define the quantity
\begin{equation}
\label{eq:hbar_def}
\boxed{
$$\hbar(\varphi,\mathfrak{p},\Lambda^-_0,\Lambda^-_1,\alpha_-):= e^{-r_\varphi}\min_{c \in \mathcal{Q}_{\alpha_-}^{\mathfrak{p}}} \int_c \alpha_- > 0,$$
}
\end{equation}
which for short will be referred to simply as $\hbar.$

\begin{rmk}
The above definition of $\hbar$ does not take the length of any Reeb chord starting at $\Lambda^-_1$ and ending at $\Lambda^-_0$ into account. This is important, since in the application that we have in mind, there will be very small such Reeb chords; see the push-off of the Lagrangian cobordism considered in the proof of Theorem \ref{thm:main} in Section \ref{sec:Proof}.
\end{rmk}
In the subsequent Section \ref{sec:CompactnessProof} we prove the following:
\begin{thm}[Non-bubbling for strips]
\label{thm:compactness}
Fix an admissible function $\varphi,$ a compact family of admissible almost complex structures $J \in \mathcal{J}(X,d\eta, \varphi)$ as well as exact Lagrangian submanifolds $L_i,$ and a family $G_{s,t} = \rho(s) G_t$ of Hamiltonians. 
Both the families $J$ and $L_i$ are required to be fixed outside of some precompact subset as in Proposition \ref{prp:Transversality}. 
Then there is a compact subset $K \subset X$ containing all $J$-holomorphic strips $u \in \mathcal{M}_{p_+,p_-}(L_0,L_1;G_{s,t})$ of either of the following types, under the assumption that $p_\pm \in \mathfrak{p}.$
\begin{enumerate}
\item[(0)] In the case when $\rho \equiv 1$ or $\rho \equiv 0$ we require
$$\max\{\mathfrak{a}_\varphi(p_+) - \mathfrak{a}_\varphi(p_-) - \hbar,\mathfrak{a}_\varphi(p_+)-\hbar\} <0;$$
\item[(1)] In the case when $\rho$ is non-constant and $\rho(s)=0$ (resp. $\rho(s)=1$) whenever $|s| \gg 0$ is sufficiently large, i.e.~$G_{-, t} \equiv G_{+,t} \equiv 0$ (resp. $G_{-, t} \equiv G_{+,t}\equiv G_t$), we require 
$$\max\{\mathfrak{a}_\varphi(p_+)-\mathfrak{a}_\varphi(p_-)-\hbar,\mathfrak{a}_\varphi(p_+)-\hbar\} < -\|G_t\|_{\OP{osc}};$$
\item[(2)] In the case when $\rho(s) = 0$ for $s \ll 0$, $\rho(s)= 1$ for $s \gg 0,$ and $\rho'(s) \ge 0,$ we require 
$$\max\{\mathfrak{a}_\varphi(p_+)-\mathfrak{a}_\varphi(p_-)-\hbar,\mathfrak{a}_\varphi(p_+)-\hbar\} < -\int_0^1 \max_X G_t dt;$$
and
\item[(3)] In the case when $\rho(s) = 1$ for $s \ll 0$, $\rho(s)= 0$ for $s \gg 0,$ and $\rho'(s) \le 0,$ we require 
$$\max\{\mathfrak{a}_\varphi(p_+)-\mathfrak{a}_\varphi(p_-)-\hbar,\mathfrak{a}_\varphi(p_+)-\hbar\} < \int_0^1 \min_X G_t dt.$$
\end{enumerate}
\end{thm}

\subsection{Well-definedness and invariance}
\label{sec:invariance}
We proceed to apply Theorem \ref{thm:compactness} in order to define our Floer complexes, and to show the needed invariance properties.

Using the above non-bubbling theorem we obtain a condition for when the differential of the Floer complex is well-defined. Namely, once the strips have been shown to be confined to a given compact subset, the remaining argument is standard.
\begin{prp}[Conditions for a well-defined complex]
\label{prp:welldefcomplex}
Under the assumption that
$$\max\{m_+ - m_--\hbar,m_+-\hbar\}<0$$
is satisfied with $\hbar$ and $\varphi$ as specified above, the map
$$d^{ m_+}\colon CF_*^{\mathfrak{p}}(L_0,L_1;G_{t};\varphi)_{m_-}^{m_+} \rightarrow CF_{*-1}^{\mathfrak{p}}(L_0,L_1;G_{t};\varphi)_{m_-}^{m_+} $$
is well-defined and satisfies $(d^{ m_+})^2=0$ for any generic $J_t \in \mathcal{J}(L_0,L_1,\varphi)$.
\end{prp}
\begin{proof}
All moduli spaces of strips involved in the definition of $d^{ m_+},$ as well as the glued strips involved in the definition of $(d^{ m_+})^2,$ satisfy the energy bounds in Case (0) of the non-bubbling Theorem \ref{thm:compactness}. In view of this result, the relation $(d^{ m_+})^2=0$ now follows as in \cite[Sections 19.2, 19.5]{OhSympTop2} by a gluing argument, taking the one-dimensional moduli spaces into account.
\end{proof}

We only establish a limited form of invariance for our Floer complexes with action filtration that will be sufficient for our needs. We work under assumptions for which the above non-bubbling result is satisfied for the relevant strips, and the rest of the argument is then again classical.
\begin{prp}[Filtered invariance]
\label{prp:invariance}
Let $\hbar$ and $\varphi$ be as specified above. 
 Assume that we are given numbers satisfying
\begin{equation}
\label{eq:inv.0}
m_- +\max\left\{\int_0^1 \max_X G_t dt,0\right\} \le m_-' \le m_+' \le m_+,
\end{equation}
 and that there is an equality

\[ CF_*^{\mathfrak{p}}(L_0,L_1;0;\varphi)=CF_*^{\mathfrak{p}}(L_0,L_1;0;\varphi)_{m_-'}^{m_+'},\]
 i.e.~all generators have action contained in the range 
 $[m_-',m_+').$ 
 Under the additional assumptions that
\begin{eqnarray}
\label{eq:inv.1} && \max\{m_+ -m_--\hbar,m_+-\hbar\} < 0, \\
\label{eq:inv.2} && \max\{m_+'-m_-'-\hbar,m_+'-\hbar\} < -\|G_t\|_{\OP{osc}},\\
\label{eq:inv.3} && \max\{m_+-m_-'-\hbar,m_+-\hbar\}< -\int_0^1 \max_X G_t dt,\\
\label{eq:inv.4} && \max\{m_+'-m_--\hbar,m_+'-\hbar\}< \int_0^1 \min_X G_t dt,
\end{eqnarray}
are satisfied, there are continuation maps
\begin{eqnarray*}
\Phi_{0,G_t}\colon CF_*^{\mathfrak{p}}(L_0,L_1;0;\varphi) \rightarrow CF_*^{\mathfrak{p}}(L_0,L_1;G_t;\varphi)_{m_-}^{m_+},\\
\Phi_{G_t,0}\colon CF_*^{\mathfrak{p}}(L_0,L_1;G_t;\varphi)_{m_-}^{m_+} \rightarrow CF_*^{\mathfrak{p}}(L_0,L_1;0;\varphi),
\end{eqnarray*}
which are chain maps whose composition admits a chain homotopy
\[ \Phi_{G_t,0}\circ \Phi_{0,G_t} \sim \id_{CF_*^{\mathfrak{p}}(L_0,L_1;0)}\]
making $\Phi_{G_t,0}$ a left-sided homotopy inverse of $\Phi_{0,G_t}.$
\end{prp}
\begin{proof}
The well-definedness of the complexes is implied by Proposition \ref{prp:welldefcomplex} above, which shows that the non-bubbling result Theorem \ref{thm:compactness} applies to the Floer strips defining the boundary operators as well as its squares. We need to show that the moduli spaces of strips involved in the chain maps and chain homotopies, as well as the different glued strips appearing in the algebraic relations, all satisfy the corresponding hypotheses of Theorem \ref{thm:compactness}. 
In other words, we need to consider the moduli spaces arising in the definitions of the operations
\begin{itemize}
\item $\Phi_{0,G_t},$ $\Phi_{0,G_t} \circ \partial,$ $\partial \circ \Phi_{0,G_t},$
\item $\Phi_{G_t,0},$ $\Phi_{G_t,0} \circ \partial,$ $\partial \circ \Phi_{G_t,0},$
\item $K_{0,G_t,0},$ $\Phi_{G_t,0} \circ \Phi_{0,G_t}$, $K_{0,G_t,0} \circ \partial$, $\partial \circ K_{0,G_t,0}$.
\end{itemize}
Here $K_{0,G_t,0}$ is the chain homotopy defined by counting rigid Floer strips with a Hamiltonian perturbation term $\rho_\kappa(s)G_t$ for an appropriate one-parameter family of compactly supported functions $\rho_\kappa(s);$ see e.g.~\cite[Chapter 19.4]{OhSympTop2} for more details.

These operations are all solutions of Cauchy-Riemann equations with a perturbation term coming from a family of Hamiltonians of the form $G_{s,t}:=\rho(s)G_t.$ Here, $0\le\rho(s)\le 1$ satisfies the assumptions of Lemma \ref{lma:action}. To show these solution strips satisfy the conditions of Theorem \ref{thm:compactness}, we use the assumption that Inequalities \eqref{eq:inv.0} and \eqref{eq:inv.1} are satisfied (for the condition in Case (0)), together with \eqref{eq:inv.2} (for Case (1)), \eqref{eq:inv.3} (for Case (2)), and \eqref{eq:inv.4} (for Case (3)).


Since we thus have established that the strips needed to define the above maps are contained in some fixed compact subset, the rest of the argument is standard. We refer to \cite[Sections 19.3, 19.5]{OhSympTop2} for the compactness and gluing argument needed for obtaining the sought algebraic relations satisfied by these maps. Verifying these relations consists of analyzing the boundary of one-dimensional moduli spaces of strips. Here we note that only generators in the specified action ranges are involved in the broken strips arising as boundary points of the relevant one-dimensional moduli spaces. Here it is crucial to use the facts that
\begin{itemize}
\item by Lemma \ref{lma:varphienergy} the differential cannot decrease action, while,
\item by Case (2) of Lemma \ref{lma:action} (with $U=\emptyset$) the chain map $\Phi_{0,G_t}$ can decrease action by at most $\max\left\{\int_0^1 \max_X G_t dt,0\right\},$
\end{itemize}
in combination with Inequalities \eqref{eq:inv.0} together with the assumption that all the generators of the complex $ CF_*^{\mathfrak{p}}(L_0,L_1;0;\varphi)=CF_*^{\mathfrak{p}}(L_0,L_1;0;\varphi)_{m_-'}^{m_+'}$ have action in the range $[m_-',m_+' )$ by assumption.
\end{proof}

As a direct application of Lemma \ref{lma:action}, the chain map in the above proposition can be shown to satisfy the following behavior with respect to the action filtration.

\begin{prp}[Filtration properties]
\label{prp:filtrationproperties}
Assume that the hypotheses of Proposition \ref{prp:invariance} are satisfied for \emph{some} choice of $\varphi \in \mathcal{C}(L_0,L_1,G_t)$. Then, for an arbitrary, and possibly different, choice of $\widetilde{\varphi} \in \mathcal{C}(L_0,L_1,G_t),$ the following can be said:
\begin{enumerate}
\item If $\langle \Phi_{0,G_t}(p_-),p_+ \rangle \neq 0$ holds for generators 
$$ p_- \in CF_*^{\mathfrak{p}}(L_0,L_1;0;\varphi)=CF_*^{\mathfrak{p}}(L_0,L_1;0;\varphi)_{m_-'}^{m_+'} \:\:\: \text{and} \:\:\: p_+ \in CF_*^{\mathfrak{p}}(L_0,L_1;G_t;\varphi)_{m_-}^{m_+},$$
then
$$ \mathfrak{a}_{\widetilde{\varphi}}(p_+) > \mathfrak{a}_{\widetilde{\varphi}}(p_-) - \int_0^1 \max_X G_t dt$$
is satisfied.
\item If $\langle \Phi_{G_t,0}(q_-),q_+\rangle \neq 0$ holds for generators
$$ q_- \in CF_*^{\mathfrak{p}}(L_0,L_1;G_t;\varphi)_{m_-}^{m_+} \:\:\: \text{and} \:\:\: q_+ \in CF_*^{\mathfrak{p}}(L_0,L_1;0;\varphi)=CF_*^{\mathfrak{p}}(L_0,L_1;0;\varphi)_{m_-'}^{m_+'} $$
then
$$ \mathfrak{a}_{\widetilde{\varphi}}(q_+) > \mathfrak{a}_{\widetilde{\varphi}}(q_-) +\int_0^1 \min_X G_t dt$$
is satisfied.
\end{enumerate}
\end{prp}

Consider a vector subspace
$$ C_*^\lambda \subset CF_*^{\mathfrak{p}}(L_0,L^\lambda_1;G_t;\varphi) $$
spanned by generators which all are intersection points contained in the complement of the support of the Hamiltonian isotopy $L^\lambda_1,$ $\lambda \in [0,1].$ Note that these intersection points are transverse for all $\lambda$ and, hence, $C_*^\lambda$ are all canonically isomorphic as vector spaces. Consider a compactly supported family $J_t^\lambda,$ $\lambda \in [0,1],$ of admissible $t$-dependent almost complex structures. Assume that any two generators $p,q \in C_*^\lambda$ satisfy
$$\max\{\mathfrak{a}_\varphi(p) - \mathfrak{a}_\varphi(q)-\hbar,\mathfrak{a}_\varphi(p)-\hbar\}<0.$$
Then, by Case (0) of Theorem \ref{thm:compactness}, any Floer strip $u \in \mathcal{M}_{p,q}(L_0,L_1^\lambda;G_t)$ is contained inside a fixed compact subset $K \subset X.$
\begin{prp}(Invariance via bifurcation analysis)
\label{prp:bifurcation}
 
In the above setting, under the additional assumption that there exists no Floer strip $u \in \mathcal{M}_{p,x}(L_0,L_1^\lambda;G_t)$ whose input is a generator $p \in C_*^\lambda$ and output some generator $x \in CF_*^{\mathfrak{p}}(L_0,L^\lambda_1;G_t;\varphi) \setminus C_*^\lambda$ for any $t \in [0,1]$ then:
\begin{enumerate}
\item Each of $C_*^i,$ $i=0,1,$ is a well-defined Floer subcomplex when $J^i_t$ is a generic admissible almost complex structure; and
\item The subcomplexes $C_*^0$ and $C_*^1$ are chain homotopy equivalent.
\end{enumerate}
\end{prp}
\begin{proof}
(1): This is a direct consequence of Proposition \ref{prp:welldefcomplex}.

(2): Floer's original sketch of invariance of Lagrangian Floer theory under Hamiltonian isotopies used bifurcation analysis: as the geometric data changes, the Floer complex changes by stabilizations (births and deaths of pairs of Lagrangian intersections) and handle-slides (the presence of isolated index $-1$ pseudoholomorphic strips in one-parameter families of such) \cite{Floer:MorseTheoryLagrangian}. This sketch was made rigorous in \cite{Sullivan} for Lagrangian Floer theory, as well as in \cite{YJLee} for Hamiltonian Floer theory.

By assumption, there are neither births nor deaths occurring in our family. Moreover, all relevant $(-1)$-strips involved in the definition of the handle-slide maps are all contained in a fixed compact subset. Together with the assumption on the non-existence of the prescribed strips, the compactness and gluing arguments from \cite{Sullivan} again show that the induced algebraic handle-slide maps are well-defined chain isomorphisms of the complexes $C_*^\lambda.$
\end{proof}

The typical situation when the above proposition can be applied is when an energy computation (possibly using an action $\mathfrak{a}_{\widetilde{\varphi}}$ defined with a different $\widetilde{\varphi}$) prevents the existence of the unwanted Floer strips.

\subsection{Naturality}
\label{sec:Naturality}

It will be useful to switch perspectives between Floer complexes defined in terms of intersection points of Lagrangian submanifolds (e.g.~for the neck stretching construction in Section \ref{sec:neckstretching}) and in terms of Hamiltonian chords (e.g.~for defining continuation maps). The following naturality property provides a translation between these two definitions.

 Note that if $L_0^\lambda$ is an exact Lagrangian isotopy, the components $\pi_0(\Pi(X;L_0^\lambda,L_1))$ are canonically identified as $\lambda$ varies. For instance, one can consider the naturally constructed lifts of the paths to the contactization $(X \times \R, dz+\eta),$ such that the endpoints of the paths lie on the Legendrian lifts of $L_0^\lambda$ and $L_1.$ Choosing the $z$--coordinate of the former to be sufficiently large compared to the one of the latter, these Legendrian lifts stay disjoint during the entire lifted isotopy.
\begin{lma}
Let $L_0,L_1 \subset (X,d\eta)$ be exact Lagrangian cobordisms, fix a function $\varphi$ as in Section \ref{sec:action}, and let $G_t \colon X \to \R$ be a Hamiltonian. Fix a time-dependent tame almost complex structure $J_t$ on $(X,d\eta)$. Using action conventions in the same section, there is a canonical action-preserving chain complex isomorphism
\[ CF_*(\phi^1_{G_t}(L_0),L_1;0) \simeq CF_*(L_0,L_1;G_t),\]
where the former complex is defined using $J_t$ and the latter complex is defined using
\[ \widetilde{J}_t := D(\phi^1_{G_t} \circ (\phi^t_{G_t})^{-1})^{-1} \circ J_t \circ D(\phi^1_{G_t}\circ (\phi^t_{G_t})^{-1}),\]
and where an intersection point $p \in \phi^1_{G_t}(L_0) \cap L_1$ is identified with the Hamiltonian chord $t \mapsto \phi^t_{G_t}((\phi^1_{G_t})^{-1}(p))$ from $(\phi^1_{G_t})^{-1}(p) \in L_0$ to $p \in L_1.$

 When $\phi^t_{G_t}$ has compact support, this isomorphism moreover preserves homotopy classes of paths under the previously mentioned canonical identification $\pi_0(\Pi(X;\phi^1_{G_t}(L_0),L_1)) \cong \pi_0(\Pi(X;L_0,L_1))$.

\end{lma}
\begin{proof}
For the differential is suffices to check that the pseudoholomorphic strips $u(s,t)$ in the definition of the differential of $CF_*(\phi^1_{G_t}(L_0),L_1;0)$, i.e.~satisfying the Cauchy--Riemann equation
$$du(\partial_s)+J_tdu(\partial_t)=0,$$
correspond bijectively to the solutions $\widetilde{u}(s,t):=(\phi^1_{G_t} \circ (\phi^t_{G_t})^{-1})^{-1}\circ u(s,t)$ of 
$$d\widetilde{u}(\partial_s)+\widetilde{J}_t(d\widetilde{u}(\partial_t)-X_{G_t}(\widetilde{u}(s,t)))=0,$$
i.e.~the Floer strips in the definition of the differential of $CF_*(L_0,L_1;G_t).$

The fact that the two definitions of action coincide follows from Cartan's formula. Namely, the differential of $(\phi^t_{G_t})^*\eta$ with respect to the variable $t$ is given by
$$ (\phi^t_{G_t})^*(d\iota_{X_{G_t}}\eta+\iota_{X_{G_t}} d\eta)=d(\phi^t_{G_t})^*(\iota_{X_{G_t}}\eta-G_t).$$
Hence, the primitive of $\eta$ pulled back to $\phi^1_{G_t}(L_0)$ can be obtained from the primitive of $\eta$ pulled back to $L_0$ after adding an integral of $(\phi^t_{G_t})^*(\iota_{X_{G_t}}\eta-G_t)$. This shows that the different versions of the action defined for the intersection points and the corresponding Hamiltonian chords actually coincide; roughly speaking, the latter integral provides the additional terms in the definition of the action for the Hamiltonian chords.
\end{proof}

\subsection{A note about the non-compact case}
\label{sec:noncomp}
There are also interesting cases where the contact manifold is non-compact. We restrict attention to the following situation. Suppose that $(P,d\theta)$ is a Liouville manifold, i.e.~a complete exact symplectic null-cobordism with a convex but no concave end. We will consider contact manifolds of the form $Y=P \times \R$ or $Y=P \times S^1.$ The canonical canonical contact form on the latter is given by $\alpha_{\OP{std}}:=dz+\theta$, where $z$ is either a coordinate on the $\R$-factor or the angular coordinate on the $S^1$-factor. Note that the canonical contact form on $(J^1M,dz+\theta_M)$ on the one-jet space of a smooth manifold $M$ is of the former type, where $\theta_M$ is the tautological one-form on the cotangent bundle $T^*M.$

In this setting we only consider contact forms $\alpha$ that coincide with $\alpha_{\OP{std}}$ outside of a compact subset. Moreover, the symplectic cobordisms that we allow will all be required to be of the form
\[ (\overline{X},d\eta)=(\{f \le r \le g\} \subset \R \times Y, d(e^r\alpha)).\]
 Here $r$ is the symplectization coordinate on the $\R$-factor. The functions $f,g\colon Y=P \times \R \to \R$ satisfy $f(y) < g(y)$ for each $y \in Y,$ and are both constant outside of a compact subset. Note that, after a suitable identification, the corresponding completed symplectic cobordism again is of the form $(\R \times Y,d\widetilde{\eta}),$ where
\begin{itemize}
\item $\widetilde{\eta}=e^re^f\alpha$ for $r \ll 0,$ while
\item $\widetilde{\eta}=e^re^g\alpha$ for $r \gg 0.$
\end{itemize}

Due to the (additional) non-compactness in this case, some extra care must be taken when one studies pseudoholomorphic curves in the symplectization in order for the compactness results to apply. We refer to e.g.~\cite[Lemma 4.1]{OnHomologicalRigidity} for a careful treatment of such an argument. The basic idea is to use a fixed cylindrical almost complex structure outside of a compact subset of $(X=\R \times Y,d\widetilde{\eta}),$ which moreover makes the canonical projection $X \to P$ into a pseudoholomorphic map; i.e.~the cylindrical almost complex structure is the lift of an almost complex structure on $(P,d\theta)$. Here we must use the fact that $\alpha = \alpha_{\OP{std}}$ is satisfied outside of a compact subset. The main point in this case is the following. As is established by the aforementioned lemma, all pseudoholomorphic curves with a given bound on the energy are contained inside some a priori given subset of the form $\R \times K \subset \R \times Y =X,$ where $K \subset Y$ is compact.

\section{Proof of Theorem \ref{thm:compactness}}
\label{sec:CompactnessProof}

In this section we prove Theorem \ref{thm:compactness}.
Compactness results in symplectizations have been proved in a number of different set-ups, see \cite{CompAbbas, CompSFT, CieliebakMohnke, Fish} for example.
However, to our knowledge, existing symplectizations results do not explicitly prove compactness for the set of solutions to Equation (\ref{eq:strip}). Rather than extending the full SFT compactness results to our set-up, we prove Proposition \ref{prp:ChordExistence}; this is a simpler result which recovers a certain Reeb chord or orbit, and implies Theorem \ref{thm:compactness}. 

Our argument to prove Proposition \ref{prp:ChordExistence} is essentially a relative (Lagrangian boundary condition) version of \cite[Sections 5 and 6]{Orderability}, which is turn almost entirely relies on \cite[Section 5]{CieliebakMohnke}.
Since \cite{Orderability} is closer to our set-up than \cite{CieliebakMohnke}, we make precise references to \cite{Orderability}.
The interested reader can then use \cite{Orderability} to see which specific result is relevant from the original source \cite{CieliebakMohnke}.

We set-up some notation to be used throughout this section. In the proof of Theorem \ref{thm:compactness} below, we will be interested in studying a sequence $\tilde{u}_k$ of pseudoholomorphic strips with Hamiltonian perturbation term induced by $G_{s,t},$ with fixed positive and negative punctures asymptotic to Hamiltonian chords $p_+$ and $p_-$ from $L_0$ to $L_1,$ respectively. The Hamiltonian $G_{s,t}$ is assumed to have support contained outside of the concave 
 and convex ends, 
and hence both asymptotics satisfy $p_\pm \subset \overline{X}.$ Both chords are, moreover, assumed to live in the component $\mathfrak{p} \in \pi_0(\Pi(X,L_0,L_1)).$

We will restrict these strips to the negative symplectic end $u_k := \tilde{u}_k|_{\tilde{u}_k^{-1}({\{ r \le {\kappa}\} \times Y})},$ where $\kappa<-N-1< -N-\epsilon,$ independent of $k,$ is a regular value of the projections in the $\R$-coordinate for $\tilde{u}_k.$
The terms $N \gg 1 \gg \epsilon>0$ will be defined in the construction of our Lagrangians in Section \ref{sec:PushOff}.
In particular, we will assume for these restrictions $u_k$ that the Hamiltonian term in (\ref{eq:strip}) vanishes
and that the Lagrangian boundary conditions $L_0,L_1$ are cylindrical 
$\{ r \le {\kappa}\} \times \Lambda_0^-, \{ r \le {\kappa}\} \times \Lambda_1^-.$
Let
\begin{equation}
 \label{eq:Z_k}
Z_k := \tilde{u}_k^{-1}({\{ r \le {\kappa}\} \times Y}) \subset \R \times [0,1]
\end{equation}
denote the domain of $u_k,$ which is a Riemann
surface due to the regularity assumption on $\kappa.$ 
Write $u_k = (a_k, f_k)$ where $a_k:Z_k \rightarrow \R$ is the projection of $u_k$ to the $r$-coordinate.

Throughout this section, $J$ (or $J_k$) denotes an almost complex 
 structure 
considered in Section \ref{sec:AdmissibleACS}, restricted to this negative end.
Since for any $u_k$ that we consider the image is compact ( although not a priori uniformly in $k$), the Hofer energy
\begin{equation}
\label{eq:HoferEnergyE}
E(u_k) := \sup_{\{\nu \in C^\infty(\R, [0,1]) \,|\, \nu' \ge 0\}}
\int u_k^*d(\nu \eta) < E = E( p_-, p_+) < \infty.
\end{equation}
can be seen to be bounded by a function of the asymptotics $ p_\pm$ only; see e.g.~\cite[Proposition 3.6]{Cthulhu}.

We begin with relative versions of the Monotonicity Lemma, the Maximum Principle, and the Conformal Modulus,
adopting as much as possible the notation of the absolute versions given in \cite{Orderability}.
Let $g_J$ be the metric on $\R \times Y$ defined by $J$ and the symplectic form.
Let $B_{g_J}(\mbox{center}; \mbox{radius})$ denote a ball defined with that metric.
\begin{lma}[Monotonicity]
\label{lma:monotonicity}
There exists constants $C, 1 > 10 \epsilon >0$ such that for all $0 < \delta < \epsilon,$ for all $J$-holomorphic maps
$$ v: (Z, \partial Z) \rightarrow \left(\{ r \le {\kappa}\} \times Y,(\{\kappa\} \times Y) \cup (\{ r \le {\kappa}\} \times \Lambda_0^-) \cup (\{ r \le {\kappa}\} \times \Lambda_1^-) \right),$$ 
 for all components $Z_0 \subset Z$ with $v$ is nonconstant on $ Z_0$, and for all $y \in v( Z_0)$ with $B_{g_J}(y;\delta) \subset \{ r < \kappa\} \times Y$ we have
$$
\OP{area}_{g_J}(v(Z_0) \cap {B_{g_J}(y;\delta)}) \ge e^{r(y)} C \delta^2.
$$
Moreover, the constants vary continuously with $J$ in the operator-norm topology.
\end{lma}

\begin{proof}
This follows from \cite[Proposition 4.7.2]{AudinLafontaine} after noting that
 the noncompact Lagrangians are cylindrical, and so the constant $C = C(L_0 \cup L_1)$
 in the monotonicity lemma is nonzero.
See also \cite[Proposition 2.69]{CompAbbas}.
\end{proof}

Denote the closed upper half-plane by
\[\HH := \{z \in \C; \:\: \mathfrak{Im}(z)\ge 0 \}\]
and its boundary by $\partial \HH.$

\begin{lma}[Maximum Principle]
\label{lma:MaximumPrinciple}
Let $U \subset \HH$ be a connected neighborhood of $0 \in \HH.$ 
If
\[u \colon (U,U \cap \partial\HH) \to (\R \times Y,\R\times(\Lambda_0^- \cup \Lambda_1^-))\]
is pseudoholomorphic and non-constant
then $r \circ u$ has no local maximum on $U,$ including $U \cap \partial\HH$.
\end{lma}

\begin{proof}
This follows from standard properties of subharmonic functions; see e.g.~\cite[Lemma 5.5]{Khovanov}.
\end{proof}

\begin{lma}[Conformal Modulus]
\label{lma:ConformalModulus}
Let $X$ denote either the interval $[0,1]$ or the unit circle $S^1 = \R/\Z.$
Let $v = (a_v,f_v): X \times [0,L] \rightarrow (-\infty, \kappa] \times Y$ be holomorphic, with Lagrangian boundary conditions $\R \times \Lambda_j^-$ on $\partial X.$
Assume $a_v(t,0) \le R < S \le a_v(t,L)$ for all $t \in X.$
Then $L,$ called the conformal modulus of $X \times [0,L],$ is bounded below by $(S-R)/2E(v).$
\end{lma}

\begin{proof}
The case when $X =S^1$ is proved as \cite[Lemma 6.9]{Orderability}. The proof of the $X =[0,1]$ case can be copied verbatim.
\end{proof}

Let $Z$ be a compact 
 (possibly disconnected) 
Riemann surface with boundary, viewed as a subset of $\R \times [0,1]$ or $\R \times (\R/2\Z).$
Let $w = (a_w,f_w): Z \rightarrow (-\infty, \kappa] \times Y$ be a smooth (not necessarily $J$-holomorphic) map, non-constant on all of its connected components, and where in the case $Z \subset \R \times [0,1]$ we assume Lagrangian conditions
$$w(Z \cap ((-\infty, \kappa] \times \{j\})) \subset \R \times \Lambda_j^-$$
for $j = 0,1.$
Assume $a_w|_{\partial Z} = \kappa,$ or in the case of a Lagrangian boundary condition, $ a_w|_{\partial Z \setminus \R \times \{0,1\}} = \kappa.$ 

 We wish to study topological features of ``tranches" of $w(Z)$, that is, of $w(Z) \cap \{ r_0 \le r \le r_1\} \times Y,$ for various choices of $(r_0,r_1).$ These arise later in this section when replacing $w$ with our sequence $u_k.$ However, $(r_0,r_1)$ is not within our control as we vary $k,$ so a priori, the topology of these tranches can grow out of control. This potential problem motivates allowing to include to the tranches we encounter pieces of $w(Z)$ that lie outside of $ \{ r_0 \le r \le r_1\} \times Y,$ to simplify the tranches' topology (such as number of connected components). We formalize this inclusion process in the below discussion.

Fix $b \in (\kappa - 2 \delta, \kappa - \delta)$ where $b, b \pm \delta$ are regular values of $a_w.$
For $R < S < \kappa-4 \delta$ (which we think of as ``not fixed" compared to $b$) where $R,S, R \pm \delta, S \pm \delta$ are regular values of $a_w$ and $S -R \ge 2 \delta,$
 we will define surfaces $Z^S_R(w), Z^R_S(w)$ as done in \cite[Section 6]{Orderability}.
 Let $\mathcal{C}_R$ be the set of connected components of $(a_w)^{-1}([R, R+ \delta])$ and of $(a_w)^{-1}([R-\delta, R]).$
 Note that components of the first type can intersect components of the second at $(a_w)^{-1}(R).$
 Begin constructing $\mathcal{C}^+_R \subset \mathcal{C}_R$ (resp. $\mathcal{C}^-_R \subset \mathcal{C}_R$) 
 by including all connected components in $ \mathcal{C}_R$ that meet 
 $(a_w)^{-1}(R+\delta)$ (resp. $(a_w)^{-1}(R-\delta)$) 
 and those in $(a_w)^{-1}([R, R+ \delta])$ (resp. $(a_w)^{-1}([R- \delta, R])$) that do not meet $(a_w)^{-1}(R).$
 Since $Z$ can a priori be disconnected, there may be components of this second type. 
Next extend first $\mathcal{C}^+_R$ (resp. second $\mathcal{C}^-_R$) to include those connected components in $\mathcal{C}_R$ which are connected to (i.e.~share a boundary component with) some component in the previously defined $\mathcal{C}^+_R$ (resp. $\mathcal{C}^-_R$), but which is not equal to a connected component indexed by $\mathcal{C}^-_R$ (resp. the extended $\mathcal{C}^+_R$).
Repeat this (finite) process as long as $\mathcal{C}^\pm_R$ increases, after which $\mathcal{C}_R = \mathcal{C}^+_R \cup \mathcal{C}^-_R.$
 The process is finite because $R$ is a regular value of $a_w.$
To abuse notation, we will also refer to $\mathcal{C}^\pm_R$ as the union, taken over the index set $\mathcal{C}^\pm_R$, of connected components. This union is a subset of the domain $Z$ of the map $w$ defined above.

Set 
\begin{eqnarray*}
 Z^S_R(w) &= & (a_w)^{-1}([R+\delta, S-\delta]) \cup \mathcal{C}^+_R \cup \mathcal{C}^-_S,\\
 Z^b(w) & = & (a_w)^{-1}((-\infty, b-\delta]) \cup \mathcal{C}^-_b,\\
 Z_S^R(w) &= &Z^b(w) \setminus Z^S_R(w).
\end{eqnarray*}

Define a subset $P_0 \subset Z$ to be a $\delta$-\emph{essential local minimum} (resp. \emph{maximum}) on level $R_0$ of $w$ if $P_0$ is a connected component of $a^{-1}_w((-\infty, R_0 + \delta])$ (resp. $a^{-1}_w([R_0 - \delta, \infty))$) and $R_0 = \min_{P_0} a_w$ (resp. $R_0 = \max_{P_0} a_w$).
Note $P_0$ need not be a point and could intersect the boundary of $Z.$ 
Define the function $\chi_w: (-\infty, \kappa] \rightarrow \Z$ on the set of regular values of $a_w$ by setting $\chi_w(r) = \chi(Z^b_r(w)).$
 Since $b > r$, we use the definition of $ Z^S_R(w)$ above setting $S = b$ and $R = r.$
Recall $b$ is fixed once $u$ is.
A value $r \in (-\infty, \kappa]$ is called a \emph{jump} if there is a non-zero difference in limits (taken over regular values)
$$
\lim_{S\rightarrow r_+} \sup \chi_w(S) - \lim_{R\rightarrow r_-} \sup \chi_w(R).
$$

Now suppose that $u = (a,f):Z \rightarrow (-\infty, \kappa] \times Y$ is the restriction, to the negative symplectic end $(-\infty, \kappa] \times Y,$ of an arbitrary $J$-holomorphic map $\tilde{u} = (\tilde{a}, \tilde{f})$ in the statement of Theorem \ref{thm:compactness}.
We construct the following ``doubles.'' 
Let $\bar{Z} \subset \R \times [-1,0]$ be the complex conjugate of $Z.$
Let $Z^d = {(\bar{Z} \cup Z)}/{\{ x-i \simeq x+i \}} \subset \R \times (\R/2\Z).$
For any regular values $R \ne S$ of $a,$ define $(Z^S_R(u))^d, (Z^b)^d \subset Z^d$ in the same way.
Construct a smooth (but not $J$-holomorphic) map $u^d= (a^d,f^d): Z^d \rightarrow (-\infty, \kappa] \times Y,$ 
such that for all $r \in \R,$ $(a^d)^{-1}(r) = (a^{-1}(r))^d.$
This implies $(Z^S_R(u))^d = Z^S_R(u^d)$ for all regular values $R \ne S.$
To see more explicitly how $u^d$ can be constructed, fix a constant $\varepsilon >0.$
For $z = x+iy \in Z^d$ with $\varepsilon < |y| < 1 - \varepsilon,$ set $u^d(x+iy) = u(x+i|y|).$
For $z = x+iy \in (a^{-1}(r))^d $ with $|y| < \varepsilon/2$ or $|y| > 1 - \varepsilon/2,$ set $f^d$ to be locally constant
(taking values in our two Legendrians $\Lambda_0^- \cup \Lambda_1^-$).
 For $\varepsilon/2 \le |y| \le \varepsilon$ and $1-\varepsilon \le |y| \le 1 - \varepsilon/2,$ smoothly interpolate.

\begin{prp}\cite[Proposition 6.7]{Orderability}
\label{prp:6.7}
Recall $\delta$ from Lemma \ref{lma:monotonicity} and the uniform bound $E$ on the Hofer energy of any curve $u$ ($=u_k$). Let $g$ be a (uniform) bound on the genus of such curves. 
\begin{enumerate}
\item
There exists an $N_0 = N_0(g, E, \delta)$ such that the number of $\delta$-essential minima and the number of jumps of $\chi_u$ is bounded from above by $N_0.$
\item
If $R,S, R \pm \delta, S \pm \delta$ are regular values of $a,$ if $ \chi ((Z^S_R(u))^d)=0,$ and if there are no
$\delta$-essential minima in $Z^S_R(u),$ then $Z^S_R(u)$ is the union of at most $N_0$ cylinders and/or strips, each running between levels $R$ and $S.$
\end{enumerate}
\end{prp}

\begin{proof}

The absolute version is proved as \cite[Lemmas 6.2-6.6]{Orderability}, which establish upper bounds, say $M_j>0,$ on the number of components, $\delta$-essential minima and absolute value of Euler characteristics of various restrictions of the maps.
 We will see below that the $M_j = M_j(\delta, E, g)$ and can be thought of ``excessively large" since they each contain terms with $\delta^{-2}.$
For the lemmas which require holomorphicity of the map, we indicate how the proof of the absolute case extends to the relative case.
For the lemmas which only require smoothness of the map, we apply their results to $u^d.$ We review each below.

Without any modification, the proof of \cite[Lemma 6.2]{Orderability} implies the number of components of $Z^S_R(u)$ and $Z_S^R(u)$ is at most $M_1 = \frac{8E}{C \delta^2}$ where $E$ is the energy defined in equation (\ref{eq:HoferEnergyE}) and $C$ is the constant from Lemma \ref{lma:monotonicity}.
The relative version of the Monotonicity Lemma \ref{lma:monotonicity} supplements the absolute version.
Adding Lemma \ref{lma:MaximumPrinciple} to preclude boundary maxima, the proof of \cite[Lemma 6.3]{Orderability}
shows the number of $\delta$-essential local minima of $u$ is at most $M_2= \frac{2E}{C \delta^2}.$
Again, the Monotonicity Lemma-based proof carries over verbatim.

Since the proof of \cite[Lemma 6.4]{Orderability} uses only smooth topology, we use it to prove 
\begin{eqnarray*}
\chi(Z^b(u^d)) & \ge & 2-3g - \frac{12 E}{C \delta^2}\\
 \chi(Z_S^R(u^d)), \chi(Z^S_R(u^d)) & \in & \left[ \chi(Z^b(u^d)) - \frac{8 E}{C \delta^2}, \frac{8 E}{C \delta^2}\right].
 \end{eqnarray*}

Similarly, the proof of \cite[Lemma 6.6]{Orderability} uses only smooth topology and so we apply it to prove that the number of jumps of $\chi_{u^d},$ and hence $\chi_{u}$ as well, is bounded above by some $M_4 = 3g + \frac{5 E}{C \delta^2} .$
This proof uses $Z_S^R(u^d)$ which otherwise appears inessential in our presentation.

Setting $N_0 = \max\{M_1,M_2,M_4\}$ gets us the first claim.
It remains to modify \cite[Lemma 6.5]{Orderability} to prove the second claim.
There is a correspondence between $\delta$-essential minima/maxima of $u$ and of $u^d$
(one-to-one on the boundary and one-to-two in the interior).
So if a component of $U$ of $Z^S_R(u^d)$ has $\chi(U) >0,$ then $U$ is a disk and $a^d|_U(\delta U) = S$
since by Lemma \ref{lma:MaximumPrinciple}, $u^d$ has no maxima. 
This implies $u|_{U \cap Z^S_R(u))}$ has a $\delta$-essential minimum.
So if $ \chi (Z^S_R(u^d))=0$ and if there are no $\delta$-essential minima of $u$ in $Z^S_R(u),$
then all components of $Z^S_R(u^d)$ have Euler characteristic 0 and therefore $u^d|_{Z^S_R(u^d)}$ 
is a union of cylinders connecting the $R$ and $S$ levels.
 Hence $N_0 \ge M_1$ implies the second claim.

\end{proof}

\begin{prp}\cite[Proposition 5.4/6.10]{Orderability}
\label{prp:MainCpt}
Let $$u_k = (a_k,f_k): Z_k \rightarrow (-\infty, \kappa] \times Y$$
be the restrictions of $\tilde{u}_k$ introduced at the beginning of this section.
Assume $\inf_k \inf_{Z_k} a_k = -\infty.$ 
Then there exists a 
 subsequence 
of the natural numbers, which we index by $n$, and a subsequence either all of strips or all of cylinders, $C_n \subset Z_{k_n}$ such that an $\R$-shift 
$v_n= (b_n, g_n)$ of the restrictions of $u_{k_n}$ to $C_n$ has the following properties:
\begin{enumerate}
\item
$C_n$ is biholomorphic to $[-l_n, l_n] \times [0,1] \subset \C$ or $[-l_n, l_n] \times S^1,$ for all $n,$ and $l_n \rightarrow \infty$ as $n \rightarrow \infty.$ Here and elsewhere, $S^1 =\R /\Z.$
\item
$\int_{C_n} v_n^* d\alpha \rightarrow 0.$
\item
There is a sequence $\sigma_n \rightarrow \infty$ such that $\pm b_n(\pm l_n, t) \ge \sigma_n$ for each $t \in [0,1].$
\item
If the $C_n$ are strips then either $v_n|_{[-l_n, l_n] \times \{0,1\}} \subset L_j$ for one of $j \in \{0,1\},$ or $v_n|_{[-l_n, l_n] \times \{j\}} \subset L_j$ for each of $j \in \{0,1\}.$ In the latter case we may moreover require that
$$r(v_n(-l_n,0)) < r(v_n(l_n,0))<0$$
is satisfied (observe that the continuous path $ \tilde{u}_n(\cdot,0): \R \rightarrow L_0$ is asymptotic at both of its ends to the starting points of the chords $p_-,p_+$ contained outside of the region $\{ r \le 0\}$).
\end{enumerate}
\end{prp}

\begin{proof}

We review the proof of \cite[Proposition 5.4/6.10]{Orderability}, modifying it to account for possible cylindrical Lagrangian boundary conditions and for Part (4).
A level $\rho \in ( -\infty, \kappa]$ is defined as essential for $u_k$ if any of the following hold: 
$\rho = \kappa;$ $\rho = \min u_k;$ $\rho-\delta$ is a $\delta$-essential minimum; $\chi_{u_k}$ has a jump at $\rho.$
Proposition \ref{prp:6.7} implies the number of essential levels of $u_k$ is bounded independent of $k.$
Since $\lim_k \min a_k = -\infty,$ we can, after possibly passing to a subsequence $n$ of the natural numbers, find an interval $[\rho_n, s_n] \subset (-\infty, \kappa-\delta)$ of length at least $n$ such that the following hold:
$[\rho_n, s_n]$ does not contain any essential levels; 
$\rho_n,s_n, \rho_n \pm \delta, s_n \pm \delta$ are regular values of $a_n;$ and
$\int_{a_n^{-1}([\rho_n, s_n])} u^*_{n} d\alpha \le E/n.$
Proposition \ref{prp:6.7} implies that $ u_n^{-1}([\rho_n, s_n] \times Y)$ is a union of cylinders and strips, each running between
$\rho_n \times Y$ and $s_n \times Y.$

 Recall that $\tilde{u}_n(0,s)$ takes values in the complement of the concave end (actually, inside $\overline{X}$) outside of a compact subset of the domain; this follows by the assumptions concerning the asymptotics of the strips. Using this fact we deduce that $u_n,$ restricted to the connected components of $u^{-1}_n([\rho_n, s_n] \times Y),$ cannot map all the components' boundaries which lie in $Z \cap (\{j\} \times \R)$ to $L_{j'},$ where $\{j,j'\} =\{0,1\}.$ The reason is that this would contradict continuity of $\tilde{u}_n(0,s)$; indeed, using the aforementioned properties, one sees that $r(u_n(0, S_1)), r(u_n(0, S_1))> \rho_n$ holds for some numbers
$$S_0 \le \inf{r^{-1}(-\infty,\rho_n)} \le \sup{r^{-1}(-\infty,\rho_n)} \le S_1.$$
In conclusion, at least one connected component of $u^{-1}_n([\rho_n, s_n] \times Y)$ must be a cylinder, a strip with both boundary conditions on the same Lagrangian $L_j,$
 or with boundary $Z \cap (\{j\} \times \R)$ mapping to $L_j$ for $j = 0,1.$
 Let $C_n$ denote one such component. 
 The proof now follows from Lemma \ref{lma:ConformalModulus}.
\end{proof}

\begin{prp}\cite[Theorem 5.3]{Orderability}
\label{prp:ChordExistence}
Let $u_k = (a_k,f_k): Z_k \rightarrow (-\infty, \kappa] \times Y$ be the restrictions of $\tilde{u}_k$ introduced at the beginning of this section.
Assume that $\inf_k \inf_{Z_k} a_k = -\infty.$
Then there exists a subsequence $k_n$ and a subsequence of all strips or of all cylinders $C_n \subset Z_{k_n},$ biholomorphically equivalent to the standard strip or cylinder, $[-l_n,l_n] \times [0,1] $ or $[-l_n,l_n] \times S^1,$ such that $l_n \rightarrow \infty$ and such that $u_{k_n}|_{C_n}$ converges (up to an $\R$-shift in either $C^\infty_{\tiny{\OP{loc}}}(\R \times [0,1]; \R \times Y)$ or $C^\infty_{\tiny{\OP{loc}}}(\R \times S^1; \R \times Y)$) to a trivial strip or cylinder over a Reeb chord or closed orbit in the class $\mathcal{Q}_{\alpha_-}^{\mathfrak{p}}$ as defined in Section \ref{sec:compactness}.
\end{prp}

\begin{proof}
Recall that because $u_k$ maps to $(-\infty, \kappa] \times Y,$ there is no perturbation in the $J$-holomorphic equations and
the Lagrangian boundary conditions are cylindrical. This is the case for which Abbas' book provides a complete set of details for the 
lemmas used to prove compactness \cite{CompAbbas}.

Aside from a Reeb orbit energy bound in the last sentence of the statement of \cite[Theorem 5.3]{Orderability}, the proof of \cite[Theorem 5.3]{Orderability}, which uses \cite[Proposition 5.4/6.10]{Orderability} allows us to deduce this proposition from Proposition \ref{prp:MainCpt} with the following modifications: 
replace cylinders with cylinders and/or strips;
supplement the absolute Monotonicity Lemma with Lemma \ref{lma:monotonicity};
derive a bound on the gradient not just at interior points but at boundary points as well, which \cite[Proposition 2.56]{CompAbbas} does; 
supplement \cite{Orderability}'s references to \cite{Hofer93} with relative analogues from \cite{CompAbbas}.
For this third modification, \cite[Theorem 5.3]{Orderability} uses (1) \cite[Lemma 28]{Hofer93} and (2) the proof of \cite[Theorem 31]{Hofer93}.
Here (1) states that if $u: \C \rightarrow \R \times Y$ is $J$-holomorphic, $E(u) < \infty$ and $\int_\C u^*d\alpha = 0,$ then $u$ is constant. 
Roughly, (2) states that if $u: \C \rightarrow \R \times Y$ is $J$-holomorphic, $E(u) < \infty,$ and $u$ is non-constant, then there exists a sequence of $s_k \rightarrow \infty$ such that the component of $u(e^{2 \pi (s+it)})$ in $Y$ converges to a $T$-periodic Reeb orbit. 
The relative versions (domains are $\HH$ instead of $\C$) of (1) and (2) are proved as \cite[Proposition 2.55]{CompAbbas} and \cite[Theorem 2.54/2.57]{CompAbbas}, respectively.

The fact that the limiting Reeb chord or closed orbit is an element of $\mathcal{Q}_{\alpha_-}^{\mathfrak{p}}$ follows from topology.
 If the limit is a closed orbit or is a chord with endpoints on the same Legendrian, then the limit of the images of the $u_k$ can serve as the homotopy to prove the chord/orbit is a trivial element in relative/absolute $\pi_1.$
If it is a chord with endpoints on two different Legendrians, Part (4) of Proposition \ref{prp:MainCpt} is also required, to conclude that the chord is homotopic to the intersection points $p_\pm \in \mathfrak{p}$ (viewed as constant paths).
\end{proof}
 
\begin{cor}
\label{cor:RefinedChordExistence}
Fix $\varphi \in \mathcal{C}(L_0,L_1, G_{s,t})$ be a Hofer function such that $r_\varphi < \infty.$ 
 Recall $\hbar$ which is defined in equation (\ref{eq:hbar_def}).
Let $u_k = (a_k,f_k): Z_k \rightarrow (-\infty, \kappa] \times Y$ be the restrictions of $\tilde{u}_k$ introduced at the beginning of this section.
Assume that $\inf_k \inf_{Z_k} a_k = -\infty.$
Then there exist a subsequence $u_{k_n}$ and a sequence of connected open subsets $U_n \subset \R \times [0,1]$ such that one of the following two cases holds.
\begin{enumerate}
\item $U_n \subset Z_{k_n}$ is precompact and satisfies
\begin{enumerate}
\item $\liminf_{n \to \infty} E_{d(\varphi\eta)}(u_{k_n}|_{U_n}) \ge \hbar,$ and
\item $\limsup_{n \to \infty} {a_{k_n}(U_n)} = -\infty.$
\end{enumerate}
\item $U_n$ contains the cylindrical end $\{ s \le - S_n \}$ for some $ S_n \gg 0,$ while it is disjoint from $\{ s \ge S_n \},$ and satisfies
\begin{enumerate}
\item $\liminf_{n \to \infty} E_{d(\varphi\eta)}(\widetilde{u}_{k_n}|_{U_n}) \ge \hbar-\mathfrak{a}_\varphi(p_-)-\int_0^1 G_{-,t}(p_-(t))dt,$
\item $\limsup_{n \to \infty} a_{k_n}(\partial U_n \setminus (\R \times \{0,1\})) = -\infty.$
\end{enumerate}
\end{enumerate}
\end{cor}
\begin{rmk}

By (1.b) and (2.b) the curves $u_n|_{U_n}$ all converge to $-\infty.$ Recall that this is a basic feature of the breaking of pseudoholomorphic curves in the SFT-setting.
\end{rmk}
\begin{proof}
Let $k_n, C_n$ be as in Proposition \ref{prp:ChordExistence}.

{\em Case 1:} Suppose $ f_{k_n}|_{C_n}$ converges to either a closed Reeb orbit or to a Reeb chord starting and ending on the same Legendrian component.
We wish to show that Case (1) applies.
Let $U_n$ be the component of $Z_{k_n} \setminus C_n$ which does not intersect 
$\partial Z_{k_n} \setminus (\R \times \{0,1\}).$ 

By the Maximum Principle Lemma \ref{lma:MaximumPrinciple}, $ a_{k_n}|_{U_n}$ obtains its maximum on $ \nu_n := \partial U_n \setminus (\R \times \{0,1\}).$ 
 The path $\nu_n$ is smooth as a subset of the boundary of the Riemann surface $Z_n$ defined in equation (\ref{eq:Z_k}).
Part (b) of Case (1) then readily follows the convergence established by Proposition \ref{prp:ChordExistence}; namely, the latter convergence is only possible after shifting the symplectisation coordinate $\R$ to $+\infty$ in view of $\pm l_n \rightarrow \pm \infty$ (there exists no finite energy punctured pseudoholomorphic curves in a half symplectisation $(-\infty,r_0] \times Y$).

Let $ \gamma_n$ be the Reeb chord (resp. orbit) from Proposition \ref{prp:ChordExistence}.
For arbitrary $\varepsilon >0$ there exists $N$ such that for all $n \ge N,$
 there exists a smooth rectangle (resp. annulus) $A_n \subset \R \times Y$ which is a bordism from $u_{k_n}( \nu_n)$ to $\{a_{k_n}(0,0)\} \times \gamma_n$ in $(\R \times Y, L_0 \cup L_1)$ for which 
$$ \left| \int_{A_n} d \alpha_-\right| = \left| \int_{\gamma_n} \alpha_- - \int_{f_{k_n}(\nu_n)} \alpha_-\right| < \varepsilon.
$$
 The inequality comes from the strong convergence of $f_{k_n}(\nu_n)$ to $\gamma_n$ in Proposition \ref{prp:ChordExistence}.
 Stokes' theorem then implies 
 $$
 \int_{U_n} u_{k_n}^*d (\varphi \eta) = e^{-r_\varphi} \int_{ \nu_n} u_{k_n}^*\alpha_- \ge e^{-r_\varphi} \left(\int_{ \gamma_n} \alpha_- - \varepsilon\right) \ge \hbar - e^{-r_\varphi} \varepsilon,
 $$
 proving part (a) of Case (1).
 For the last inequality, we use that $ \gamma_n \in \mathcal{Q}_{\alpha_-}^{\mathfrak{p}}$ by Proposition \ref{prp:ChordExistence}.

{\em Case 2:} When $ f_{k_n}|_{C_n}$ converges to a Reeb chord running between $\Lambda_0^-$ and $\Lambda_1^-,$ we set $U_n$ to be the component of $\R \times [0,1] \setminus C_n$ which includes the cylindrical end 
$\{ s \le -S_n \},$ and hence excludes $\{ s \ge S_n\}.$

We again apply the maximum principle (to $ a_{k_n}|_{C_n}$) to conclude part (b) of Case (2). Part (a) of Case (2) follows from the same argument in Case (1): appending a smooth rectangle $A_n$ with $\varepsilon$-small $d \alpha_-$-area to $\tilde{u}_{k_n}(U_n);$ Stokes' theorem; together with $e^{-r_\varphi} \int_{ \gamma_n} \alpha_- \ge \hbar.$ The ($\int_0^1 G_{-,t}dt $)-term arises from the use of Stokes' theorem as in Equality \eqref{eq:sympen}.

\end{proof}

\begin{proof}[Proof of Theorem \ref{thm:compactness}]
We prove Case (2) of the theorem, as cases (0), (1) and (3) are similar.
As discussed at the beginning of this section, we assume the theorem does not hold. Then there exists a sequence of $\tilde{u}_k$ 
(with restrictions $\tilde{u}_k|_{Z_k} = u_k= ({a}_k, {f}_k)$) of solutions to the $G_{s,t}$-perturbed $J_k$-holomorphic equations, such that (after passing to a subsequence and taking $k$ sufficiently large) $\lim_k \min {a}_k = -\infty$ and the $J_k$ are sufficiently close to each other in the operator-norm topology so that the constants in Lemma \ref{lma:monotonicity} can be treated as independent of $k \gg 0.$

 Recall the action $\mathfrak{a}_{\varphi}$ defined in equation (\ref{eq:a_varphi}).
The hypothesis of Case (2) of Theorem \ref{thm:compactness} implies the inequality
\begin{eqnarray*}
0 & > &
 \max\{ \mathfrak{a}_{\varphi}(p_+) - \mathfrak{a}_{\varphi}(p_-) - \hbar, \mathfrak{a}_{\varphi}(p_+) - \hbar\} + \int_0^1 \max G_t dt.
\end{eqnarray*}
We finally reach a contradiction by the below inequalities, derived from the above hypothesis, Cases (1.a) and (2.a) of Corollary \ref{cor:RefinedChordExistence} for $k \gg 0$ sufficiently large, together with Case (2) of Lemma \ref{lma:action} (in that precise order):
\begin{eqnarray*}
0 & > &\mathfrak{a}_{\varphi}(p_+) - \mathfrak{a}_{\varphi}(p_-) - \hbar + \int_0^1 \max G_t dt \\
& \ge & \mathfrak{a}_{\varphi}(p_+)-\mathfrak{a}_\varphi^- - E_{d(\varphi\eta)}(\tilde{u}_k|_{U_k}) + \int_0^1 \max G_t dt \\
& \ge & E_{d(\varphi \eta), J_k} (\tilde{u}_k|_{[0,1] \times \R \setminus U_k}) > 0.
\end{eqnarray*}
\end{proof}

\section{Usher's trick in the symplectization}
\label{sec:usher}
In this section we fix a contact manifold $(Y,\alpha)$ with a contact form $\alpha$. When speaking about the symplectization we will thus always refer to the identification $(\R \times Y,d(e^r\alpha))$. In order to make the distinction between a Hamiltonian on $\R \times Y$ and a contact Hamiltonian on $Y,$
 in this section we use 
the superscript ``$\alpha$,'' i.e.~$H^\alpha_t \colon Y \to \R,$ in the case of a contact Hamiltonian.

The contact Hamiltonian $ H_t^\alpha \colon Y \to \R$ has a lift $e^r H_t^\alpha \colon \R \times Y \to \R$ to the symplectization which, of course, is unbounded unless $H_t^\alpha\equiv 0.$ In order to produce a Hamiltonian that is possible to measure, it will be necessary to cut it off using a smooth bump function. Obviously, this must be done in a way which preserves the Hamiltonian diffeomorphism in some significant subset. In order to perform the cut-off in an efficient manner we will need to use the so-called Usher's trick, which made its first appearance in the proof of \cite[Theorem 1.3]{ObservationsHofer} due to Usher.

That this trick is very useful also for contact Hamiltonians was observed by Shelukhin in \cite{HoferContactomorphism}; see Lemma \ref{lma:usher}. Here we follow his construction when proving the following result.
\begin{prp}
\label{prp:usher}
Assume that we are given an indefinite contact Hamiltonian $ H_t^\alpha$ and any number $\delta>0.$ Let { {$A,B\ge 0$}} be as in (\ref{eq:norms}). Then there exists a compactly supported Hamiltonian $G_t \colon \R \times Y \to \R$ for which

\[ \phi^1_{G_t}|_{(\phi^1_{G_t})^{-1}([A+\delta/2,A+3\delta/2] \times Y)} =\phi^1_{e^r H_t^\alpha},\]
and whose Hofer norm satisfies the bound
\[ e^{A+2\delta} \min_Y H_{1-t}^\alpha < \min_{\R \times Y} G_t \le 0 \le \max_{\R \times Y} G_t < e^{A+2\delta} \max_Y H_{1-t}^\alpha.\]
Moreover, we can take $G_t$ to be of the form $\kappa_t \cdot e^r K_t^\alpha$ for an indefinite contact Hamiltonian $K_t^\alpha \colon Y \to \R$ and a smooth one-parameter family of bump-functions $\kappa_t \colon \R \times Y \to [0,1]$ satisfying:

\begin{enumerate}
\item  The functions $\kappa_t$ are supported inside $[0,A+B+1] \times Y$; 
\item On the subsets
$$\phi^t_{G_t}((\phi^1_{G_t})^{-1}([A+\delta/2,A+3\delta/2] \times Y)), \: t \in [0,1],$$
we have $\kappa_t \equiv 1.$
\end{enumerate}
\end{prp}
 Before proving this proposition we first recall Usher's trick in this setting.

\begin{lma}[Usher]
\label{lma:usher}
Let $ F_t^\alpha \colon Y \to \R$ be a contact Hamiltonian. Then there exists a contact Hamiltonian $\widetilde{F}_t^\alpha \colon Y \to \R$ satisfying the following properties:
\begin{enumerate}
\item $\phi^1_{\alpha, \widetilde{F}_t^\alpha}=\phi^1_{\alpha, F_t^\alpha}$,
\item $(e^r \widetilde{F}_t^\alpha) \circ \phi^t_{e^r \widetilde{F}_t^\alpha} = e^r F^\alpha_{1-t}$,
\item $\min_{t \in [0,1] \atop y \in Y} \tau^t_{\alpha,\widetilde{F}_t^\alpha}(y) = -\max_{t \in [0,1] \atop y \in Y} \tau^t_{\alpha,- F_{1-t}^\alpha}(y)$,
\item $\max_{t \in [0,1] \atop y \in Y} \tau^t_{\alpha,\widetilde{F}_t^\alpha}(y) = -\min_{t \in [0,1] \atop y \in Y} \tau^t_{\alpha,- F_{1-t}^\alpha}(y)$.
\end{enumerate}
In particular, by the second property, $\widetilde{F}_t^\alpha$ is indefinite if and only if $ F_t^\alpha$ is.
\end{lma}
\begin{proof}
For any two Hamiltonians $G_t$ and $H_t$ on a symplectic manifold, it is a general and easily checked fact that
$$ \phi^t_{G_t} \circ \phi^t_{H_t}=\phi^t_{G_t+H_t\circ (\phi^t_{G_t})^{-1}}$$
holds. In particular, the Hamiltonian isotopy $(\phi^t_{H_t})^{-1}$ is generated by the Hamiltonian $G_t$ satisfying
\[ G_t \circ \phi^t_{G_t} = - H_t.\]
 When applying this to the lift $H_t=e^rH_t^\alpha$ of a contact Hamiltonian, note that the produced $G_t=e^rG_t^\alpha$ again is the lift of a contact Hamiltonian.

Recall the standard fact that the contact Hamiltonian $ H_t^\alpha=-F_{1-t}^\alpha$ generates
\[ \phi^t_{\alpha,- F_{1-t}^\alpha} = \phi^{1-t}_{\alpha, F_t^\alpha} \circ (\phi^1_{\alpha, F_t^\alpha})^{-1},\]
 as can be checked by explicitly evaluating the contact form $\alpha$ on the infinitesimal generator. Applying the previously mentioned formula for the inverse of the Hamiltonian isotopy to $ e^rH_t^\alpha=-e^rF_{1-t}^\alpha,$ the produced contact Hamiltonian $\widetilde{F}_t^\alpha:=G_t^\alpha$ generating $(\phi^t_{\alpha,- F_{1-t}^\alpha})^{-1}$ is the one we seek. Indeed, using the aforementioned identities, we immediately compute
\[ \phi^1_{\alpha, \widetilde{F}_t^\alpha}=(\phi^1_{\alpha,- F_{1-t}^\alpha})^{-1}=\phi^1_{\alpha, F_t^\alpha} \circ (\phi^{1-1}_{\alpha, F_t^\alpha})^{-1}=\phi^1_{\alpha, F_t^\alpha}.\]
 Thus we have established Parts (1) and (2).

Parts (3) and (4) are a straightforward consequence of the fact that $\phi^*\alpha=e^\tau\alpha$ if and only if $(\phi^{-1})^*\alpha=e^{-\tau}\alpha.$
\end{proof}

\begin{proof}[Proof of Proposition \ref{prp:usher}]
 As in the proof of the above lemma, recall the standard identity
\begin{equation}
\label{eq:one}
\phi^{1-t}_{\alpha, F_t^\alpha}\circ (\phi^1_{\alpha, F_t^\alpha})^{-1}=\phi^t_{\alpha,- F_{1-t}^\alpha}
\end{equation}
of contact isotopies for any contact Hamiltonian $ F_t^\alpha$.

We now consider the lifted contact Hamiltonian $-e^r H_{1-t}^\alpha$. Lemma \ref{lma:usher} (Usher's trick) applied to the contact Hamiltonian $ F_t^\alpha:=-H_{1-t}^\alpha$ produces an, again, indefinite contact Hamiltonian $ \widetilde{F}_t^\alpha.$ Denoting this contact Hamiltonian by $- K_{1-t}^\alpha:= \widetilde{F}_t^\alpha,$ we obtain
\begin{equation}
\label{eq:two}
\phi^1_{-e^r H_{1-t}^\alpha}=\phi^1_{-e^r K_{1-t}^\alpha},
\end{equation}
(this is Part (1) of that lemma) while for each $t \in [0,1]$ the equalities
\begin{eqnarray*}
\max_{\phi^t_{-e^r K_{1-t}^\alpha}([A,A+2\delta]\times Y)} (-e^r K_{1-t}^\alpha ) &= & \max_{[A,A+2\delta]\times Y)}(-e^r H_t^\alpha),\\
\min_{\phi^t_{-e^r K_{1-t}^\alpha}([A,A+2\delta]\times Y)} (-e^r K_{1-t}^\alpha )&= & \min_{[A,A+2\delta]\times Y)}(-e^r H_t^\alpha),
\end{eqnarray*}
hold (this is Part (2) of that lemma).

We now claim that the corresponding contact Hamiltonian $ K_t^\alpha$ will do the job, and where the cut-off function $ \kappa_{1-t} \colon \R \times Y \to \R_{\ge 0}$ is constructed to have support inside some arbitrarily small neighborhood of $\phi^t_{-e^r K_{1-t}^\alpha}([A+\delta/2,A+3\delta/2]\times Y) \cap \{ r \ge 0 \}$ for each $t \in [0,1],$ in which $ \kappa_{1-t} \equiv 1$ moreover is satisfied.

The statement that $\phi^1_{G_t}$ satisfies
$$\phi^1_{G_t}|_{(\phi^1_{G_t})^{-1}([A+\delta/2,A+3\delta/2] \times Y)} =\phi^1_{e^r H_t^\alpha} $$
in the subset $(\phi^1_{G_t})^{-1}([A+\delta/2,A+3\delta/2] \times Y),$ as required, is a consequence of Property (3) of this proposition, together with
\[ \phi^1_{e^r K_t^\alpha}=(\phi^1_{-e^r K_{1-t}^\alpha})^{-1}=(\phi^1_{-e^r H_{1-t}^\alpha})^{-1}=\phi^1_{e^r H_t^\alpha}\]
as follows from Formulas \eqref{eq:one} and \eqref{eq:two}. It thus remains to prove that Property (3) can be made to hold simultaneously with Properties (1) and (2). 

First, note that the subsets
$$\phi^t_{-e^r K_{1-t}^\alpha}([A+\delta/2,A+3\delta/2] \times Y), \:\: t\in[0,1],$$
all are contained inside $[0,A+B+1) \times Y$, since Parts (3) and (4) of Lemma \ref{lma:usher} shows that the $r$-coordinate of this subset is decreased (resp. increased) by at most the amount $A>0$ (resp. $B>0$) during the isotopy. (Recall that the lemma was applied to $ F_t^\alpha=-H_{1-t}^\alpha$ in order to construct $ \widetilde{F}_t^\alpha=-K_{1-t}^\alpha$.)

Second, consider the equality
$$ \phi^{1-t}_{ G_t} \circ (\phi^1_{ G_t})^{-1}([A+\delta/2,A+3\delta/2] \times Y) = \phi^t_{- G_{1-t}}([A+\delta/2,A+3\delta/2]\times Y) $$
which holds by the Hamiltonian analogue of Formula \eqref{eq:one}. The right-hand side can be seen here to coincide with $\phi^t_{-e^r K_{1-t}^\alpha}([A+\delta/2,A+3\delta/2] \times Y)$ by the construction of the cut-off function $\kappa_t$ together with $G_t=\kappa_t\cdot e^r K_t^\alpha.$ In particular, $\phi^{1-t}_{G_t} \circ (\phi^1_{G_t})^{-1}([A+\delta/2,A+3\delta/2] \times Y)$ is contained in the subset where $ \kappa_{1-t}$ is equal to one.
\end{proof}

The pay-off of this trick is that the Hofer norm for $G_t$ in Proposition \ref{prp:usher} is controlled with a factor approximately equal to $e^{A}$ as opposed to $e^{A+B}.$ Such a larger factor would result using the naive approach setting $G_t = \kappa_t \cdot e^r H_t^\alpha$ instead of $G_t = \kappa_t \cdot e^r K_t^\alpha$ with a contact Hamiltonian $ K_t^\alpha$ as constructed by Lemma \ref{lma:usher}.

\section{Splashing and neck-stretching}
\label{sec:SplashNeck}
Here we use a splitting along a hypersurface of contact type towards two goals. 
In Section \ref{sec:neckstretching}, we review neck-stretching.
In Section \ref{sec:splashing}, we wrap a Lagrangian in a certain way that, combined with neck-stretching, ultimately leads to a Mayer--Vietoris long exact sequence in Floer homology.
For similar decompositions of complexes defined by counts of pseudoholomorphic curves in different settings, see \cite{BorderedI} and \cite{BorderedII}. The approach taken here is also closely related to the Mayer--Vietoris decomposition in symplectic homology and Wrapped Floer homology as constructed in \cite{SymplecticEilenbergSteenrod} by Cieliebak--Oancea.

In this section let $L_0,L_1 \subset (X,d\eta)$ denote two exact Lagrangian submanifolds in an exact symplectic manifold. Consider a dividing hypersurface $(Y,\alpha) \subset (X,d\eta)$ of contact-type, where $\alpha=\eta|_{TY}$. Assume that
\[\Lambda_i := L_i \cap Y \subset (Y,\alpha), \:\: i=0,1,\]
are connected Legendrian submanifolds, where $\Lambda_i \subset L_i$ moreover is dividing and where $\Lambda_0 \cap \Lambda_1=\emptyset.$ 
We will also take a neighborhood of $Y \subset (X,d\eta)$ and its exact symplectomorphism to the subset
\[ ([T-3\varepsilon,T+3\varepsilon] \times Y,d(e^r\alpha)) \]
of the symplectization as part of the data, for some $\varepsilon>0$ and arbitrary $T \in \R$ (such a neighborhood always exists). The image of $L_i$ under this identification is given by the cylindrical Lagrangian submanifold
\[ [T-3\varepsilon,T+3\varepsilon] \times \Lambda_i \subset [T-3\varepsilon,T+3\varepsilon] \times Y, \:\: i=0,1. \]

In addition, we fix here a piecewise smooth Hofer function $\varphi$ (see Definition \ref{dfn:varphi}) which moreover is assumed to satisfy
$$\varphi|_{[T-3\varepsilon,T+3\varepsilon] \times Y} \equiv 1.$$
Recall that for $i = 0,1$ we have chosen primitives $f^\varphi_0,f^\varphi_1$ of the pullbacks of $\varphi \eta$ to $L_0,L_1.$

\subsection{Neck-stretching}
\label{sec:neckstretching}
Neck-stretching can be used to prevent certain pseudoholomorphic curves or strips from crossing a hypersurface. In the setting of Floer homology with a Hamiltonian perturbation term, it appears in e.g.~\cite[Section 2.3]{SymplecticEilenbergSteenrod}, while in the setting of symplectic field theory (SFT for short) we refer to \cite{IntroSFT}, \cite{CompSFT}. However, note that the approach taken here is simpler compared to the one in \cite{SymplecticEilenbergSteenrod}, since our method only uses positivity of energy for strips, and does not rely on the SFT compactness theorem.

Denote by $X_L \sqcup X_R=X \setminus ([T,T+\varepsilon]\times Y)$ the components ``to the left'' and ``to the right'' of the cylindrical region, respectively (here we used the property that the hypersurface is dividing). One way to formulate the neck-stretching is as follows. First, excise the cylindrical pair
$$([T,T+\varepsilon] \times Y,[T,T+\varepsilon] \times (\Lambda_0 \cup \Lambda_1)) \subset (X, L_0 \cup L_1), $$
and replace it with the ``longer pair''
$$([T,T+\varepsilon+\lambda] \times Y,[T,T+\varepsilon+\lambda] \times (\Lambda_0 \cup \Lambda_1)).$$
Rescale the symplectic form to be $e^\lambda d\eta$ on the subset $X_R.$ In this manner we obtain a new pair of exact Lagrangian submanifolds of a new symplectic manifold.

For us we will only be interested in the case when $[T,T+\varepsilon] \times Y$ is a cylindrical subset of a cylindrical end of $X$. By assumption we have $\varphi(r) \equiv 1 \in \R$ for all $r \in [T,T+\varepsilon].$ In this case there is an alternative formulation of the above construction, where instead of replacing the manifold we replace the function $\varphi \colon X \to \R_{>0}$ with the function
$$\varphi_\lambda \colon X \to \R_{>0},$$
determined by
\begin{equation}
\label{eq:VarphiNeckStretch}
\varphi_\lambda := \begin{cases}
\varphi(x), & x \in X_L,\\
e^{\rho_\lambda(r)}\varphi(x)=e^{\rho_\lambda(r)}, & x=(r,y) \in [T,T+\varepsilon]\times Y,\\
e^\lambda \varphi(x), & x \in X_R.
\end{cases}
\end{equation}
where $\rho_\lambda \colon \R \to [0,\lambda]$ is a smooth interpolation function for which $\rho_\lambda' \ge 0$ has compact support, while it satisfies $\rho_\lambda(T)=0$ and $\rho_\lambda(T+\varepsilon)=\lambda.$ Note that $\varphi_\lambda$ is a Hofer function (in particular, see Part (2) of Definition \ref{dfn:varphi}). Moreover, there is an inclusion
$$\mathcal{J}(L_0,L_1,\varphi_\lambda) \subset \mathcal{J}(L_0,L_1,\varphi)$$
of subsets of admissible almost complex structures.
 Recall Section \ref{sec:AdmissibleACS}. 

Let $C_i = f^{\varphi}_i(T, \cdot)$ denote the value of the primitive $ f^{\varphi}_i$ at $L_i \cap Y \cap \{r = T\},$ which is constant by Lemma \ref{lma:ActionComp}, together with the fact that $L_i \cap Y$ is connected.
 Recall $f^{\varphi}_i$ is a primitive of the one-form $\varphi \eta$ pulled back to $L_i.$ See Section \ref{sec:LagrangianCobordismsAndConcordances}.
\begin{lma}
\label{lma:StretchedAction}
Consider intersection points $p_L,p_R \in L_0 \cap L_1$ contained inside $X_L$ and $X_R$, respectively. After neck-stretching we have
$$\mathfrak{a}_{\varphi_{\lambda}}(p_R)=e^\lambda(\mathfrak{a}_\varphi(p_R)-(C_1-C_0))+(C_1-C_0)$$
while the action of $p_L$ is unaffected, i.e.~$\mathfrak{a}_{\varphi_\lambda}(p_L)=\mathfrak{a}_\varphi(p_L).$
\end{lma}
The effect of replacing $\varphi$ with $\varphi_\lambda$ will be said to be a {\em{neck-stretching with parameter $\lambda$}}.

\begin{proof}
The primitives of $\varphi_\lambda\eta$ pulled-back to $L_i$ are given by
\begin{equation*}
f^{\varphi_\lambda}_i = \begin{cases}
 f^\varphi_i(x), & x \in L_i \cap X_L, 
\\
e^\lambda (f^\varphi_i(x)-C_i)+C_i, & x \in L_i \cap X_R,
\end{cases}
\end{equation*}
as can be seen by a direct computation.
\end{proof}

\subsection{Splashing}
\label{sec:splashing} 

Consider the autonomous Hamiltonian
$$h^{\OP{Spl}}_{ \varepsilon,Z_-,Z_+}=\int_0^r e^s\beta^{\OP{Spl}}_{\varepsilon,Z_-,Z_+}(s)ds \colon (X,d\eta) \to \R$$
which only depends on the symplectization coordinate $r \in \R$, and where $\beta^{\OP{Spl}}_{\varepsilon,Z_-,Z_+}$ is as shown in Figure \ref{fig:wrap}. The corresponding Hamiltonian vector field is given by $\beta^{\OP{Spl}}_{\varepsilon,Z_-,Z_+}R_\alpha,$ whose support inside $[T-3\varepsilon,T+3\varepsilon] \times Y$.

We define
$$L^{\OP{Spl}}_0 :=\phi^1_{h^{\OP{Spl}}_{\varepsilon,Z_-,Z_+}}(L_0)$$
to be a Hamiltonian deformation of $L_0,$ supported in $[T-3\varepsilon,T+3\varepsilon] \times Y$. As in \cite{FuchsRutherford}, we call this a ``splash.''
We assume $L^{\OP{Spl}}_0 \cap [T-3\varepsilon,T+3\varepsilon] \times Y$ is cylindrical outside of $\{|r - (T \pm 2\varepsilon)| < \varepsilon^2\}$ and $\{|r - T| < \varepsilon^2\}.$ The following result is straightforward.

\begin{lma} For any fixed numbers $Z_- < 0 < Z_+$, we can make the above Hamiltonian 
$ h^{\OP{Spl}}_{\varepsilon,Z_-,Z_+}$ arbitrarily small in the uniform $C^0$-norm by shrinking $\varepsilon>0$. After either increasing
$Z_+$ or decreasing $Z_-,$ we may moreover assume that $h^{\OP{Spl}}_{\varepsilon,Z_-,Z_+} \equiv 0$ holds outside of $[T-3\varepsilon,T+3\varepsilon] \times Y$.
\end{lma}

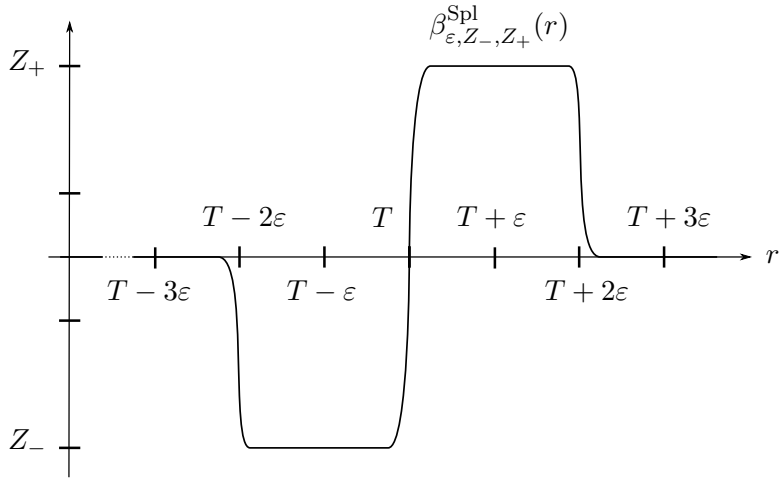
\begin{figure}[htp]
\centering
\labellist
\pinlabel $\beta^{\OP{Spl}}_{\varepsilon,Z_-,Z_+}(r)$ at 170 170
\pinlabel $T-2\varepsilon$ at 74 98
\pinlabel $T+\varepsilon$ at 167 98
\pinlabel $r$ at 273 83
\pinlabel $T-\varepsilon$ at 103 70
\pinlabel $T$ at 126 98
\pinlabel $T+2\varepsilon$ at 203 70
\pinlabel $T+3\varepsilon$ at 234 98
\pinlabel $T-3\varepsilon$ at 38 70
\pinlabel $Z_+$ at -8 156
\pinlabel $Z_-$ at -8 13
\endlabellist
\includegraphics{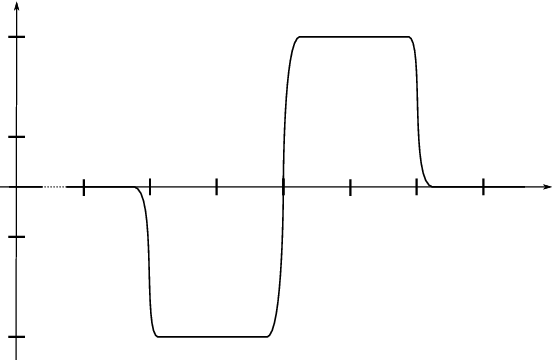}
\caption{The splashing construction is performed by applying the time-$1$ flow of the Hamiltonian vector field 
 $\beta^{\OP{Spl}}_{\varepsilon,Z_-,Z_+}(r) R_\alpha$ 
supported in $\{ r \in [T-3\varepsilon,T+3\varepsilon]\}$ to the Lagrangian $L_0$ which is cylindrical in the same subset.}
\label{fig:wrap}
\end{figure}

\begin{lma}
\label{lma:IntersectionComplex}
The intersection points
$$\{p^-_i\} =L^{\OP{Spl}}_0 \cap L_1 \cap [T-\varepsilon^2,T) \times Y$$
are in bijective correspondence with a subset of the Reeb chords from $\Lambda_1$ to $\Lambda_0,$ while the intersection points 
$$\{p^+_i\} = L^{\OP{Spl}}_0 \cap L_1 \cap (T,T+\varepsilon^2] \times Y$$
are in bijective correspondence with a subset of the Reeb chords from $\Lambda_0$ to $\Lambda_1.$ 
 For appropriate $\beta^{\OP{Spl}}_{\varepsilon,Z_-,Z_+},$ 
the $\varphi$-action of the intersection point $p^\pm_i$ and the length of the corresponding Reeb chord $c_i^\pm$ satisfy
$$ |\pm e^T \ell(c_i^\pm) + (C_0 - C_1) - \mathfrak{a}_\varphi(p_i^\pm) | < K {\varepsilon},$$
where the constant $ K = K(Z_-, Z_+)$ depends continuously on its  two  parameters. (The value of these constants $K$ depend on the particular choice of family $\beta^{\OP{Spl}}_{\varepsilon,Z_-,Z_+}$ of functions; more precisely, they are taken to satisfy $\varepsilon^2|\frac{d}{dr}\beta^{\OP{Spl}}_{\varepsilon,Z_-,Z_+}| \le K(Z_-,Z_+)$ .) 
\end{lma}

\begin{proof}
By the construction of
$$L^{\OP{Spl}}_0=\phi^1_{ \beta^{\OP{Spl}}_{\varepsilon,Z_-,Z_+} R_\alpha}(L_0)$$
each intersection point $p \in L^{\OP{Spl}}_0 \cap L_1$ corresponds to a chord of the time-one flow of the rescaled Reeb vector field $\beta^{\OP{Spl}}_{\varepsilon,Z_-,Z_+}R_\alpha$, where the function $\beta^{\OP{Spl}}_{\varepsilon,Z_-,Z_+}$ only depends on $r.$ The correspondence between intersection points and Reeb chords is hence straightforward. (Also, see Figure \ref{fig:chords2} for a similar correspondence.) By the same reasons, the pull-back of $e^r\alpha_+$ to $L^{\OP{Spl}}_0$ takes the form $e^r\frac{d}{dr}\beta^{\OP{Spl}}_{\varepsilon,Z_-,Z_+}(r)dr$; see Part (2) of Lemma \ref{lma:ActionComp}.

 We see from Figure \ref{fig:wrap} that we can choose $\beta^{\OP{Spl}}_{\varepsilon,Z_-,Z_+}$ such that 
$\left|\frac{d}{dr}\beta^{\OP{Spl}}_{\varepsilon,Z_-,Z_+}\right| < (Z_+ - Z_-)/\varepsilon^{2}.$ So, in particular, for
 
 $\{|r-T| \le \varepsilon^2\}$ we can assume
$$\left|e^r\frac{d}{dr}\beta^{\OP{Spl}}_{\varepsilon,Z_-,Z_+}-e^T\frac{d}{dr}\beta^{\OP{Spl}}_{\varepsilon,Z_-,Z_+}\right| \le |e^r-e^T|K/\varepsilon^2 \le K$$
 for a constant $K(Z_+,Z_-).$ The integral 
$$ \int_T^r \left(e^r\frac{d}{dr}\beta^{\OP{Spl}}_{\varepsilon,Z_-,Z_+}-e^T\frac{d}{dr}\beta^{\OP{Spl}}_{\varepsilon,Z_-,Z_+} \right) dr$$
then satisfies the estimate
$$|(f^\varphi_0-C_0) -(e^T\beta^{\OP{Spl}}_{\varepsilon,Z_-,Z_+}(r)-0)| \le K\varepsilon.$$
Since $f^\varphi_1$ is constantly equal to $C_1$ inside this region, the statements concerning the action thus hold. (Recall that $f_0^\varphi(p_i^\pm)-f_1^\varphi(p_i^\pm)=\mathfrak{a}_\varphi(p_i^\pm),$ while $e^T\beta^{\OP{Spl}}_{\varepsilon,Z_-,Z_+}(r(p_i^\pm))=\pm e^T \ell(c_i^\pm)$.)
\end{proof}

 Note that 
there are no intersection points of $ L^{\OP{Spl}}_0 \cap L_1$ in the region where $L^{\OP{Spl}}_0 \cap [T-3\varepsilon,T+3\varepsilon] \times Y$ is cylindrical, which includes the slices $\{r = T \pm \varepsilon\},$
 because $L_1$ is also cylindrical.

 In Section \ref{sec:ContinuationMaps}, we will define $\overline{M}_\pm$ in terms of $M_\pm$ from equation (\ref{eq:norms}).
For now, just consider $\overline{M}_\pm$ as constant. Also, we recall the definition of $\varphi_\lambda$ in equation (\ref{eq:VarphiNeckStretch}).

\begin{prp}
\label{prp:WrapObstruction}

Fix $Z_+>0>Z_-$ satisfying
$$Z_+ + (C_0 - C_1)>\overline{M}_+ > \overline{M}_- > Z_- - (C_0 - C_1).$$
After possibly shrinking $\varepsilon>0,$ increasing $\lambda,$ increasing $Z_+>0$ or decreasing $Z_-<0,$ the following can be made to hold. Consider the set of intersection points $p \in L^{\OP{Spl}}_0 \cap L_1$ such that $\overline{M}_- \le \mathfrak{a}_\varphi(p) \le \overline{M}_+.$ For any almost complex structure in
$$\mathcal{J}(L_0,L_1,\varphi_\lambda) \subset \mathcal{J}(L_0,L_1,\varphi),$$
there exists no pseudoholomorphic strip whose input is an intersection point as above located in $\{ r \le T+ \varepsilon\}$ (resp. $\{ r \ge T - \varepsilon\}$) and whose output is an intersection point as above located in $\{ r \ge T+ \varepsilon\}$ (resp. $\{ r \le T - \varepsilon\}$). 
\end{prp}

\begin{proof}
Let $\lambda = \lambda(L^{\OP{Spl}}_0,L_1, T, \varphi)>0$ be some positive constant and perform a neck-stretching with parameter $\lambda$ to one of the slices $\{r = T \pm \varepsilon\},$ with the corresponding functions $\varphi_\lambda^\pm \colon X \to \R_{>0}.$
Let $f^\pm_0,$ $f^\pm_1$ denote the primitives of the corresponding pull-backs of $\varphi_\lambda^\pm \eta$ to $L^{\OP{Spl}}_0,$ and $L_1$, respectively, and $f^{\OP{Spl}}_0$ denote a primitive of the pull-back of $\varphi\eta$ to $L^{\OP{Spl}}_0$ uniquely determined by $f^{\OP{Spl}}_0(T) = C_0 = f^\varphi_0(T).$
The induced actions will be denoted by $\mathfrak{a}_{\varphi_\lambda^\pm}(p) = f^\pm_0(p) - f^\pm_1(p).$ 

{\em Claim 1:} For sufficiently large $\lambda,$ any intersection point $p \in L^{\OP{Spl}}_0 \cap L_1 \cap \{r \ge T + \varepsilon\}$ satisfies $\mathfrak{a}_{\varphi_\lambda^+}(p) < Z_-.$

\emph{Proof of Claim 1:} Lemma \ref{lma:StretchedAction} implies that
\begin{eqnarray*}
\mathfrak{a}_{\varphi_\lambda^+}(p)
& = & e^\lambda ((f^{\OP{Spl}}_0(p) - f^\varphi_1(p))- (f^{\OP{Spl}}_0(T+\varepsilon) - f^\varphi_1(T+\varepsilon)))\\
& + & 
 (f^{\OP{Spl}}_0(T+\varepsilon) - f^\varphi_1(T+\varepsilon)) \\
& <& e^\lambda \nu_1 + \nu_2
\end{eqnarray*} 
for some constants $\nu_1, \nu_2$ independent of $\lambda.$
It suffices to verify $\nu_1 <0,$ that is, $f^{\OP{Spl}}_0(p) - f^\varphi_1(p) < f^{\OP{Spl}}_0(T+\varepsilon) -f^\varphi_1(T+\varepsilon).$ Consider first the case when $p$ is an intersection point in 
$\{T+\varepsilon \le r \le T+ 3\varepsilon\}.$
In this case, Figure \ref{fig:wrap} illustrates 
$$ f^{\OP{Spl}}_0(p) - f^{\OP{Spl}}_0(T+\varepsilon) = \int_{T+\varepsilon}^p e^r\frac{d}{dr}\beta^{\OP{Spl}}_{\varepsilon,Z_-,Z_+}(r) dr 
 < 
0.$$ 
 Since $L_1$ is cylindrical in this region, $f^\varphi_1(p) = f^\varphi_1(T+\varepsilon).$
All other intersections points of $L^{\OP{Spl}}_0$ and $L_1$ in $\{r \ge T+ 3\varepsilon\}$ exist as intersections points of $L_0$ and $L_1.$
 For these, we list some relations and approximations that we then combine to get the desired result.
Among the parameters that we adjust, we fix $\lambda$ last, immediately preceded by $\varepsilon.$
 In particular the product of $kZ_+$ (which can be thought close to $3 e^T Z_+$) with $\varepsilon$ can be made small enough to satisfy the last inequality in the second line below, as well as (once $T>0$ is fixed) the last inequality in the fourth line below. 
\begin{eqnarray*}
 f^{\OP{Spl}}_0(p) - f^\varphi_0(p) & = & (f^{\OP{Spl}}_0(T+3\varepsilon) -C_0) - (f^\varphi_0(T+3\varepsilon) -C_0) \\
& = & 
 \int_T^{T+3\varepsilon} e^r\frac{d}{dr}\beta^{\OP{Spl}}_{\varepsilon,Z_-,Z_+}(r) dr < Z_+ (e^{T+3\varepsilon} - e^T) < k Z_+\varepsilon, 
 \end{eqnarray*}
$$ f^{\varphi}_0(p) - f^\varphi_1(p) = \mathfrak{a}_\varphi(p) \le \overline{M}_+ < Z_+ - kZ_+\varepsilon + (C_0 -C_1), $$
$$ f^{\OP{Spl}}_0(T+\varepsilon) - f^{\OP{Spl}}_0(T+3\varepsilon) = \int^{T+\varepsilon}_{T+3\varepsilon} e^r\frac{d}{dr}\beta^{\OP{Spl}}_{\varepsilon,Z_-,Z_+}(r) dr > Z_+ e^{T+\varepsilon} > Z_+ + k Z_+\varepsilon, $$
$$ (C_0 - C_1) - (f^{\OP{Spl}}_0(T+ 3 \varepsilon) - f^\varphi_1(T+3 \varepsilon))= \int^T_{T+3\varepsilon} e^r\frac{d}{dr}\beta^{\OP{Spl}}_{\varepsilon,Z_-,Z_+}(r) dr < Z_+ (e^{T+3\varepsilon} - e^T) < k Z_+\varepsilon, $$
$$ f^\varphi_1(T+\varepsilon) = f^\varphi_1(T+3 \varepsilon). $$
\begin{eqnarray*}
f^{\OP{Spl}}_0(p) - f^\varphi_1(p) & < & f^\varphi_0(p) - f^\varphi_1(p) + k Z_+ \varepsilon \\
& < & Z_+ + (C_0 - C_1) \\
 &= &   Z_+ + (C_0 - C_1) + (kZ_+\varepsilon - kZ_+\varepsilon) -  (f_1^{\varphi}(T+\varepsilon) - f_1^{\varphi}(T+ 3\varepsilon)) \\
& < & f^{\OP{Spl}}_0(T+ \varepsilon) - f^\varphi_1(T+\varepsilon) \\
&-& ( f^{\OP{Spl}}_0(T+ 3 \varepsilon) - f^\varphi_1(T+3 \varepsilon)) + (C_0 - C_1) - k Z_+ \varepsilon \\
& < & f^{\OP{Spl}}_0(T+ \varepsilon) - f^\varphi_1(T+\varepsilon) - (C_0 - C_1) + (C_0 - C_1).
\end{eqnarray*}

{\em Claim 2:} For sufficiently large $\lambda,$ any intersection point $p \in L^{\OP{Spl}}_0 \cap L_1 \cap \{r \ge T - \varepsilon \}$ satisfies $\mathfrak{a}_{\varphi_\lambda^-}(p) > Z_+.$

\emph{Proof of Claim 2:} Lemma \ref{lma:StretchedAction} implies that
\begin{eqnarray*}
\mathfrak{a}_{\varphi_\lambda^-}(p) & = & e^\lambda ((f^{\OP{Spl}}_0(p) - f^\varphi_1(p))- (f^{\OP{Spl}}_0(T-\varepsilon) 
 + f^\varphi_1(T-\varepsilon)))\\
 & + &(f^{\OP{Spl}}_0(T-\varepsilon) - f^\varphi_1(T-\varepsilon)) \\
& >& e^\lambda \nu_1 + \nu_2
\end{eqnarray*}
It suffices to verify $\nu_1 >0,$ that is, $f^{\OP{Spl}}_0(p) - f^\varphi_1(p) > f^{\OP{Spl}}_0(T-\varepsilon) -f^\varphi_1(T-\varepsilon).$ Consider first the case when $p$ is an intersection point in 
$\{T-\varepsilon \le r \le T+ 3\varepsilon\}.$
In this case, Figure \ref{fig:wrap} illustrates $f^{\OP{Spl}}_0(p) > f^{\OP{Spl}}_0(T-\varepsilon),$ while $f^\varphi_1(p) = f^\varphi_1(T+\varepsilon).$
All other intersections points in $\{r \ge T+ 3\varepsilon\}$ exist as intersections points of $L_0 \cap L_1.$
 To complete the verification for these points, repeat the string of inequalities as in the proof of Claim 1 with the following changes: (1) replace $\le \overline{M}_+ < Z_+ - kZ_+\varepsilon + (C_0 -C_1)$ with 
 $\ge \overline{M}_- > Z_- + kZ_-\varepsilon - (C_0 -C_1);$ 
(2) replace $T+\varepsilon$ with $T-\varepsilon$ and $T+3\varepsilon$ with $T-3\varepsilon;$
(3) reflect all integration through $r=T.$

 For any $J \in \mathcal{J}(L_0,L_1, \varphi_\lambda),$ a pseudoholomorphic 
 strip must increase action (from input to output) for any action, including $\mathfrak{a}_{\varphi_\lambda^\pm}$; see Lemma \ref{lma:varphienergy}. Note that neck-stretching does not change the actions for any chords or intersection points to the left of the neck.
Therefore, by Claims 1 and 2, there is no pseudoholomorphic strip with Lagrangian boundary conditions $L^{\OP{Spl}}_0$ and $L_1$ 
 whose input is an intersection point as in the proposition's statement, located in $\{ r \le T+ \varepsilon\}$ (resp. $\{ r \ge T - \varepsilon\}$) and whose output is an intersection point as in the proposition's statement, located in $\{ r \ge T+ \varepsilon\}$ (resp. $\{ r \le T - \varepsilon\}$)

\end{proof}

\section{Pushing off the Lagrangian concordance}
\label{sec:PushOff}

In 
 Sections \ref{sec:PushOff} and \ref{sec:Proof}, 
 we assume that we are given a complete Lagrangian concordance $L \subset (X,d\eta)$ inside a complete symplectic cobordism from $(Y_-,\alpha_-)$ to $(Y_+,\alpha_+)$; c.f.~Definition \ref{dfn:LagCob}. The $(\pm)$-ends of $L$ are denoted by $\Lambda_\pm \subset (Y_\pm,\alpha_\pm)$, where $\Lambda_+=\Lambda$. The assumption that $L$ is a concordance means that 
$L$ is diffeomorphic to $\R \times \Lambda$ and that, hence, $\Lambda_-$ is diffeomorphic to $\Lambda.$

In this sections we fix two numbers
\[ 0 < \epsilon<\sigma \]
where the former is supposed to be sufficiently small and the latter can be arbitrary.

Let $\Lambda_\pm^\epsilon:=\phi^{-\epsilon}_{\alpha_\pm,1}(\Lambda_\pm) \subset (Y_\pm,\alpha_\pm)$ denote the image of the ends of $L$ under the time-$(-\epsilon)$ Reeb flow associated to the contact forms $\alpha_\pm$. The goal here is to construct a smooth family of Lagrangian cobordisms $L_{\epsilon,N,s}$ from $\Lambda_-^\epsilon$ to $\Lambda_+^\epsilon$ depending smoothly on the parameters $\epsilon,N,$ and $s \in [\epsilon, \sigma]$. Each of these Lagrangian cobordisms will be given as the image
\[ L_{\epsilon,N,s}=\phi^1_{H_{\epsilon,N,s}}(L) \]
of $L$ under a Hamiltonian diffeomorphism, where $H_{\epsilon,N,s} \colon X \to \R$ is a smooth family of autonomous Hamiltonians. 

 Our construction will be made to satisfy the following crucial properties:
\begin{itemize}
\item The above Hamiltonian flow is of the form $\frac{d}{dt}\phi^t_{H_{\epsilon,N,s}} \equiv -\epsilon R_{\alpha_\pm}$ outside of the compact subset $\{|r| \ge N+\epsilon \}$ of the respective cylindrical ends of $(X,d\eta)$. In particular, the cobordisms $L_{\epsilon,N,s}$ are all cylinders over $\Lambda_\pm^\epsilon$ outside the same subset.
\item The above Hamiltonian vector field is of the form $\beta_{\epsilon,N,s}(r) R_{\alpha_\pm}$ inside the complement
$$X \setminus (\{ | r-N| \le 2\epsilon\} \cup \overline{X}),$$
where $\beta_{\epsilon,N,s}(r)=\beta_{\epsilon,N,\epsilon}(r)$ moreover holds inside the complement
$$X \setminus (\{ | r-N| \le 2\epsilon\} \cup \{ -N+\epsilon \le r \le -\epsilon \}\cup \overline{X}).$$
\item For $s=\epsilon$ there exists a Weinstein neighborhood of $L \subset (X,d\eta)$, i.e.~a neighborhood symplectically identified with a neighborhood of the zero-section of $(T^*L,d\theta_L)$, in which all $\phi^t_{H_{\epsilon,N,\epsilon}}(L)$, $t \in [0,1]$, are graphs of the exact one-forms $t\,dg_{\epsilon,N}$ for a smooth family of functions
\[ g_{\epsilon,N} \colon L \to \R.\]
This family of functions is moreover required to be a Morse perturbation of a Morse--Bott function $\widetilde{g}_{\epsilon,N} \colon L \to \R$ that satisfies:
\begin{itemize}
\item The inequality $d\widetilde{g}_{\epsilon,N}(\partial_r)>0$ holds when restricted to the subsets $\{|r| \ge N+\epsilon \}$ of the cylindrical ends; and
\item The form $d\widetilde{g}_{\epsilon,N}$ is non-vanishing outside of the two critical submanifolds $\{r=\pm N\} \cap L \subset L$, each of which is non-degenerate in the Bott sense and diffeomorphic to $\Lambda$;
\end{itemize}
The function $g_{\epsilon,N}$ will be taken to be $ C^2$-close to $\widetilde{g}_{\epsilon,N}$, and to differ with the latter function only in a small neighborhood of the above critical submanifolds.
\end{itemize}
 See Figures \ref{fig:epsilon} and \ref{fig:bulge2} for illustrations of the deformations in the first two points. In fact, if the deformation described by the figures are applied to the whole concordance, then the intersection with the original concordance will be of Bott type; the third point is achieved after a generic Morse perturbation of this situation.

In order to construct the sought Hamiltonian isotopy, we choose the following strategy. First, we construct an appropriate Weinstein neighborhood of $L$, second, we construct the above function $g_{\epsilon,N}$ and, third, we construct the Hamiltonian vector field $\beta_{\epsilon,N,s}(r) R_{\alpha_\pm}$ inside the subset $\{ -N+\epsilon \le r \le -\epsilon \}$.

\subsection{A Weinstein neighborhood of the Lagrangian concordance}
\label{sec:wein}
We fix a parametrization $\psi \colon \R \times \Lambda \to L$ which maps $(-\infty,-1) \times \Lambda$ and $(1,+\infty) \times \Lambda$ into the concave and convex cylindrical ends of $X$, respectively. More specifically, using $q$ to denote the standard coordinate on $\R$, outside of $(-1,1) \times \Lambda$ we ask that $\psi$ takes the form
\[\psi(q, x) = \begin{cases}
(q +1,\widetilde{\psi}_-(x)) \in (-\infty,0] \times \Lambda_- \subset (-\infty,0] \times Y_-,& q \le -1,\\
(q-1,\widetilde{\psi}_+(x)) \in [0,+\infty) \times \Lambda_+ \subset [0,+\infty) \times Y_+,& q \ge 1,
\end{cases}\]
where $\widetilde{\psi}_-$ and $\widetilde{\psi}_+=\id_\Lambda$ are parametrizations of $\Lambda_-$ and $\Lambda_+=\Lambda$, respectively.

 First we use the Legendrian normal neighborhood theorem (see \cite[Theorem 6.2.2]{Geiges}) in order to extend the above Legendrian embeddings $\widetilde{\psi}_\pm$ to \emph{contact-form preserving} identifications
\begin{gather*}
\widetilde{\Psi}_\pm \colon (V_\pm,dz+\theta_\Lambda) \hookrightarrow (Y_\pm,\alpha_\pm),\\
\widetilde{\Psi}_\pm|_{0_\Lambda} = \widetilde{\psi}_\pm,
\end{gather*}
defined on neighborhoods
\[V_\pm \subset (J^1\Lambda=T^*\Lambda \times \R,dz+\theta_\Lambda)\]
of the zero-section.

 Then we use Weinstein's Lagrangian neighborhood theorem (see e.g.~\cite[Theorem 3.33]{SympTop}) in order to extend the above Lagrangian embedding $\psi \colon \R \times \Lambda \hookrightarrow (X,d\eta)$ to a (non-exact) symplectic embedding
\begin{gather*}
\Psi \colon (U,d\theta_{\R \times \Lambda}) \hookrightarrow (X,d\eta),\\
\Psi|_{0_{\R \times \Lambda}}=\psi,
\end{gather*}
defined on a neighborhood
\[ U \subset (T^*(\R \times \Lambda)=T^*\R \times T^*\Lambda,d\theta_{\R \times \Lambda}=d\theta_{\R} \oplus d\theta_\Lambda)\]
of the zero section.

Again we use $q$ to denote the standard coordinate on the $\R$-factor, while $\mathbf{q}$ is used to denote local coordinates on $\Lambda$. The coordinates $p$ and $\mathbf{p}$ are then used to denote the corresponding (locally defined) canonical conjugate momenta.

 These locally defined coordinates can be used to express \emph{globally well-defined} symplectomorphisms
\begin{gather*}
F_a \colon (T^*(\R \times \Lambda),d\theta_{\R \times \Lambda}) \xrightarrow{\cong} (\R \times J^1 \Lambda,d(e^r(dz+\theta_\Lambda))),\:\:a \in \R,\\
((q,\mathbf{q}),(p,\mathbf{p})) \mapsto (q+a,(\mathbf{q},e^{-(q+a)}\mathbf{p},-e^{-(q+a)}p)).
\end{gather*}
In the below lemma we use this identification in order to construct a symplectomorphism $\Psi$ of a particularly convenient form.

\begin{lma}
 The symplectomorphism $\Psi$ above can be constructed (for a possibly smaller domain $U$) so that it takes the form
$$ \Psi((q,\mathbf{q}),(p,\mathbf{p})) = \begin{cases}
 \widetilde{\Psi}\circ F_1, & q \le -1,\\
 \widetilde{\Psi}\circ F_{-1}, & q \ge 1,
\end{cases}$$
on the cylindrical ends of $(X,d\eta)$.
\end{lma}

\begin{proof}
 
We need to examine the construction of the symplectomorphism in the proof of the Weinstein Lagrangian neighborhood theorem to show that the sought properties can be achieved. In the proof one starts with an extension of $\psi$ to the cotangent bundle of $\R \times \Lambda$ by a \emph{diffeomorphism} (i.e.~smooth but not necessarily symplectic), which moreover is assumed to be symplectic along the zero-section (i.e.~infinitesimally). 
After an application of Moser's trick, this diffeomorphism is deformed to a symplectomorphism, while fixing the map set-wise along the zero-section.

In the case when the original diffeomorphism already is an exact symplectomorphism in some subset $U \subset T^*(\R \times \Lambda),$ in the sense that it preserves fixed global choices of primitives of the symplectic forms in that subset, it is readily seen that the application of Moser's trick can be assumed to not deform the map inside $U.$ The reason is that the time-differential of an interpolating path of symplectic forms then admits a path of primitives all which \emph{vanish} inside $U.$

In view of the above, the sought symplectomorphism $\Psi$ can be constructed by starting with a diffemorphic embedding of the cotangent bundle which is equal to $F_{\pm1}$ in the cylindrical ends. To that end, note that these symplectomorphisms become exact for a suitable choice of Liouville form on the cotangent bundle $T^*(\R \times \Lambda).$
\end{proof}

Observe that there is a possibly smaller neighborhood $U' \subset U$ of the zero-section $0_{\R \times \Lambda} \subset T^*(\R \times \Lambda)$ which is preserved by the $\R$-action
\[((q,\mathbf{q}),(p,\mathbf{p})) \mapsto ((q+t,\mathbf{q}),(e^tp,e^t\mathbf{p})), \:\: t \in \R,\]
on $T^*(\R \times \Lambda)$. This flow rescales the symplectic form by $e^t$, and corresponds to the Liouville vector field $\partial_r$ on the convex and concave ends of $(X,d\eta)$

\subsection{Explicitly defined push-offs along the ends}
\label{sec:explicit}

First, we consider the function
\[\beta^+_{\epsilon,N} \colon \R \to [-\epsilon,\epsilon]\]
shown in Figure \ref{fig:epsilon} which satisfies
\begin{itemize}
\item $\frac{d}{d r}\beta^+_{\epsilon,N}(r) \le 0$;
\item $\beta^+_{\epsilon,N}(r)=\epsilon$ for $r \le N-\epsilon$; 
\item $\beta^+_{\epsilon,N}(r)=-\epsilon$ for $r \ge N+\epsilon$;
\item $\beta^+_{\epsilon,N}(N)=0$ and $\frac{d}{d r}\beta^+_{\epsilon,N}(N) < 0$.
\end{itemize}
We moreover require this function to satisfy
\[\beta^+_{\epsilon,N}(r):=\epsilon\beta^+((r-N)/\epsilon)\]
for a suitable fixed smooth function
\[\beta^+ \colon \R \to [-1,1]\]
satisfying $\beta^+(-r)=-\beta^+(r)$.

Second, we define the function
\begin{eqnarray*}
& & \beta^-_{\epsilon,N} \colon \R \to [-\epsilon,\epsilon], \\
& & \beta^-_{\epsilon,N}(r):=\beta^+_{\epsilon,N}(-r),
\end{eqnarray*}
which is shown in Figure \ref{fig:bulge1}.

We also consider the function
$\beta^-_{\epsilon,N,\sigma} \colon \R \to [-\epsilon,\sigma]$
which coincides with $\beta^-_{\epsilon,N}$ outside of some \emph{compact} subset of $(-N+\epsilon,0) \subset \R$, while inside $(-N+\epsilon,0)$ it satisfies
\begin{itemize}
\item $\frac{d}{d r}\beta^-_{\epsilon,N,\sigma}(r) \ge 0$ for $-N+\epsilon \le r \le -N+2\epsilon$;
\item $\frac{d}{d r}\beta^-_{\epsilon,N,\sigma}(r) \le 0$ for $-\epsilon \le r \le 0$;
\item $\beta^-_{\epsilon,N,\sigma}(r)=\sigma$ for $-N+2\epsilon \le r \le -\epsilon$;
\end{itemize}

For any number $s$ we now define the function
\[ \beta^-_{\epsilon,N,s} := \beta^-_{\epsilon,N} + \frac{s-\epsilon}{\sigma-\epsilon}(\beta^-_{\epsilon,N,\sigma}-\beta^-_{\epsilon,N}).\]
In the case $s=\sigma$ the function is shown in Figure \ref{fig:bulge2}. In the following, we will only be considering this parameter with values $s \in [\epsilon,\sigma].$

Finally, note that the vector fields $\beta^+_{\epsilon,N}(r)R_{\alpha_+}$ and $\beta^-_{\epsilon,N,s}(r)R_{\alpha_-}$ constructed above on the convex and concave cylindrical ends of $(X,d\eta)$, respectively, are both Hamiltonian vector fields.

\begin{figure}[htp]
\centering
\labellist
\pinlabel $\beta^+_{\epsilon,N}(r)$ at 65 83
\pinlabel $\epsilon$ at -3 72
\pinlabel $-\epsilon$ at -7 24
\pinlabel $r$ at 203 47
\pinlabel $N-\epsilon$ at 78 35
\pinlabel $N+\epsilon$ at 137 35
\pinlabel $N$ at 110 60
\endlabellist
\includegraphics{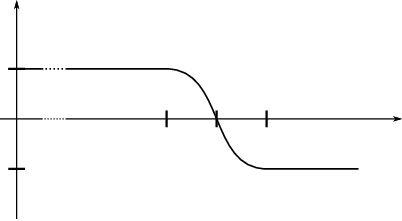}
\caption{ On the convex end, all Lagrangian cobordisms $L_{\epsilon,N,s}$ are obtained from $L$ by applying the time-1 flow of the Hamiltonian vector field $\beta^+_{\epsilon,N}(r) R_{\alpha_+}$.}
\label{fig:epsilon}
\end{figure}
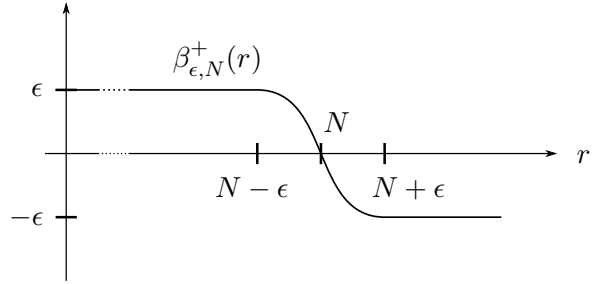

\begin{figure}[htp]
\centering
\labellist
\pinlabel $\beta^-_{\epsilon,N}(r)$ at 100 75
\pinlabel $-N-\epsilon$ at 15 53
\pinlabel $-N$ at 50 28
\pinlabel $-N+\epsilon$ at 73 53
\pinlabel $\sigma$ at 176 92
\pinlabel $\epsilon$ at 176 63
\pinlabel $-\epsilon$ at 179 15
\pinlabel $r$ at 198 39
\endlabellist
\includegraphics{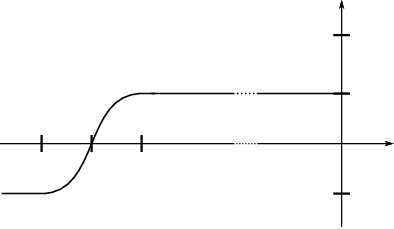}
\caption{The function $\beta^-_{\epsilon,N}(r).$}
\label{fig:bulge1}
\end{figure}
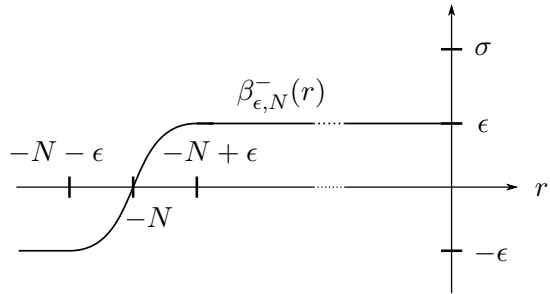

\begin{figure}[htp]
\centering
\labellist
\pinlabel $\beta^-_{\epsilon,N,\sigma}(r)$ at 100 104
\pinlabel $-N-\epsilon$ at 15 53
\pinlabel $-N$ at 50 28
\pinlabel $-N+\epsilon$ at 73 53
\pinlabel $-N+2\epsilon$ at 90 28
\pinlabel $-\epsilon$ at 140 28
\pinlabel $\sigma$ at 176 92
\pinlabel $\epsilon$ at 176 63
\pinlabel $-\epsilon$ at 179 15
\pinlabel $r$ at 198 39
\endlabellist
\includegraphics{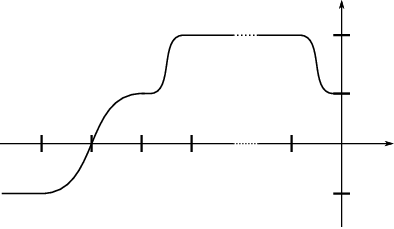}
\caption{On the concave end, the exact Lagrangian cobordism $L_{\epsilon,N,\sigma}$ is obtained from $L$ by applying the time-1 flow of the Hamiltonian vector field $\beta^-_{\epsilon,N,\sigma}(r) R_{\alpha_-}$.}
\label{fig:bulge2}
\end{figure}

\subsection{A Bott push-off}
We now use the coordinates provided by the Weinstein neighborhood defined in Section \ref{sec:wein}, identifying a neighborhood of $L$ with $U \subset (T^*(\R \times \Lambda),d\theta_{\R \times \Lambda})$. In these coordinates we can find a smooth family of functions
\[\widetilde{g}_{\epsilon,N} \colon (\R \setminus (-1,1)) \times \Lambda \to \R\]
for which the sections $t\,d\widetilde{g}_{\epsilon,N}$, $t \in [0,1]$, in the above Weinstein neighborhood are identified with the corresponding time-$t$ flow of $L \setminus \overline{X} \subset X$ under the vector fields $\beta^+_{\epsilon,N}(r)R_{\alpha_+}$ and $\beta^-_{\epsilon,N,s}(r)R_{\alpha_-}$ constructed in Section \ref{sec:explicit} with $s=\epsilon$. In these coordinates, we may moreover assume that $ \widetilde{g}_{\epsilon,N}|_{[-1-\delta,-1]}=-\epsilon e^{q+1},$
while $ \widetilde{g}_{\epsilon,N}|_{[1,1+\delta]}=-\epsilon e^{q-1}-\epsilon$ for some small $\delta>0$.

Finally, the above function $\widetilde{g}_{\epsilon,N}$ defined on $(\R \setminus (-1,1)) \times \Lambda$ extends smoothly to a function defined on all of $\R \times \Lambda$, without introducing any critical points inside $[-1,1] \times \Lambda$. After making the latter extension sufficiently $C^1$-small, this finishes the construction of the sought functions $\widetilde{g}_{\epsilon,N} \colon \R \times \Lambda \to \R$. Note that the critical points of these functions consist of the two critical manifolds $\{q=\pm (N+1)\}$, each of which is non-degenerate in the Bott sense and diffeomorphic to $\Lambda$.

\subsection{A Morse--Bott perturbation}
\label{sec:MorseBott}
The push-offs $t\,d\widetilde{g}_{\epsilon,N}$ defined on the cylindrical ends of $(X,d\eta)$ intersect $L$ precisely in the manifolds $\{\pm N\} \times \Lambda_\pm$ contained inside the respective cylindrical end. Here we provide a generic perturbation of the above function, making these intersections transverse.

Again consider the coordinates defined in Section \ref{sec:wein}. Fix a generic Morse function $h\colon \Lambda \to [0,1]$. We use $h$ to create a $ C^2$-small perturbation of $\widetilde{g}_{\epsilon,N}$ which near the critical manifold $\{q=\pm(N+1)\}$ takes the form
\[\widetilde{g}_{\epsilon,N} +\epsilon^{10} \rho(q\mp(N+1))h \colon \R \times \Lambda \to \R.\]
Here $\rho \colon \R \to [0,1]$ is a fixed smooth bump-function that is equal to 1 in some neighborhood of $0$, while its support is contained inside $(-\epsilon^3,\epsilon^3) \subset \R$. It follows that these perturbations have support contained in the subsets $\{ |q \mp (N+1)| \le \epsilon^3\}$.

Finally, the produced perturbation of $\widetilde{g}_{\epsilon,N}$ is our sought Morse function $g_{\epsilon,N} \colon \R \times \Lambda \to \R$.

\begin{rmk}
 Recall that the above construction is an instance of the standard technique used in order to obtain a Morse function by perturbing a Morse--Bott function: roughly speaking, add a Morse function on the critical manifold to the Morse-Bott function.
\end{rmk}

\subsection{The ``bulge'' inside the concave end}
Note that the construction of $L_{\epsilon,N,s}$ now is complete in the case when $s=\epsilon$. In order to finish the construction, what remains is thus the cases $s \in [\epsilon, \sigma]$. Recall that these cobordisms all agree outside of the subset $ \{-N+\epsilon \le r \le -\epsilon\}$ by construction. Inside the latter subset, we prescribe $L_{\epsilon,N,s}$ to be equal to the time-$1$ map of $L$ under the flow of the Hamiltonian vector field $\beta_{\epsilon,N,s} R_{\alpha_-}$ as constructed in Section \ref{sec:explicit}. Note that the whole cobordism $L_{\varepsilon,N,s}$ obtained this way is Hamiltonian isotopic to $L$. This finishes the construction.

\section{Proof of Theorem \ref{thm:main}}
\label{sec:Proof}

In this section we prove our main theorem. This is done by studying the Lagrangian Floer homology of a pair of cobordisms (Section \ref{sec:FloerHomology}) together with the technique of splashing (Section \ref{sec:SplashNeck}). The pair of Lagrangians that we start with consists of the concordance $L$ from $\Lambda_-$ to $\Lambda$, together with its Hamiltonian deformation $L_{\epsilon,N,s}$ as constructed in Section \ref{sec:PushOff}. The latter is a Lagrangian concordance from $\Lambda_-^\epsilon$ to $\Lambda^\epsilon$, two Legendrians obtained by applying the time-$(-\epsilon)$ Reeb flow to $\Lambda_-$ and $\Lambda,$ respectively.

\subsection{Passing to an indefinite contact Hamiltonian}
First we must replace the given contact Hamiltonian $H_t$ with one which is \emph{indefinite}, in the sense that $H_t$ attains the value zero for all $t \in [0,1].$

For a given contact isotopy $\phi^t_{\alpha,H_t} \colon (Y,\xi) \to (Y,\xi)$, taking a suitable family $c_t$ of constant functions as described Lemma \ref{lma:factor}, this indefiniteness is satisfied for the contact Hamiltonian that generates the composition
$$ \phi^{-C_t}_{R_\alpha} \circ \phi^t_{\alpha,H_t}=\phi^t_{\alpha,-c_t+H_t \circ \phi^{C_t}_{R_\alpha}} \:\:\: \text{for} \:\:\: C_t:=\int_0^t c_t ds.$$
In order to see that we may replace $H_t$ by the latter indefinite contact Hamiltonian $-c_t+H_t \circ \phi^{C_t}$ when proving Theorem \ref{thm:main}, we note the following elementary result:
\begin{lma}
\label{lma:ReebChordCorrespondence}
\begin{enumerate}
\item Since the Reeb flow $\phi^t_{R_\alpha}$ preserves the contact form $\alpha,$ the numbers $A$ and $B$ associated to the contact isotopies
$$\phi^{-C_t}_{R_\alpha} \circ \phi^t_{\alpha,H_t} \:\: \text{and} \:\: \phi^t_{\alpha,H_t},$$
as well as their oscillatory norms coincide; and
\item For $a \le b,$ there is a bijection between the subsets
$$\mathcal{Q}_\alpha(\Lambda,\phi^{-C_1}_{R_\alpha}\circ\phi^1_{\alpha,H_t}(\Lambda);a-C_1,b-C_1) \:\:\: \text{and} \:\:\: \mathcal{Q}_\alpha(\Lambda,\phi^1_{\alpha,H_t}(\Lambda);a,b),$$
of Reeb chords.
\end{enumerate}
\end{lma}
From now on, we will therefore assume $H_t$ to be indefinite in the above sense.
\subsection{Fixing constants}

Recall the strict Inequality \eqref{eq:nobreaking} in the assumptions of the theorem:
$$0< e^A \|H\|_{\OP{osc}} < \sigma(\Lambda_-,\alpha_-).$$

The parameters $s \in [\epsilon, \sigma]$ and $N>0$ in the construction of the Hamiltonian push-offs $L_{\epsilon,N,s}$ will be chosen so that the following properties hold.

First we fix $\sigma\ge\epsilon$ and $\delta>0$ for which
\begin{equation}
\label{eq:sigma}
0< e^{A+2\delta} \|H\|_{\OP{osc}} < \sigma < \sigma(\Lambda_-,\alpha_-)
\end{equation}
is satisfied, and where $\delta>0$ is sufficiently small.

Second, we shrink $\epsilon>0$, so that it becomes shorter than the length of the smallest Reeb chord from $\Lambda$ to itself
(see Lemma \ref{lma:ChordBijection} which requires a bijection between the set of Reeb chords from $\Lambda_-$ to itself with the set of Reeb chords from $\Lambda_-$ to $\Lambda_-^\epsilon$)
and so that we have
\begin{equation}
\label{eq:hbar}
 \sigma < e^{-2\epsilon}(\sigma(\Lambda_-,\alpha_-)-\epsilon).
\end{equation}

\subsection{The homotopy class of ``contractible'' chords}
We now need to pinpoint a distinguished connected component of the space $\Pi(X;L,L_{\epsilon,N,s})$ of paths in $X$ with starting point in $L$ and endpoint in $L_{\epsilon,N,s}$. Denote by $\mathfrak{p}_0 \in \pi_0(\Pi(X;L,L_{\epsilon,N,s}))$ the component of paths from $L$ to $L_{\epsilon,N,s}$ containing the Hamiltonian chord
$$[0,\epsilon] \ni t \mapsto (-N-2\epsilon,\phi^{-t}_{e^r}(x)) \in (-\infty,0] \times Y_- \subset X, \:\: x \in \Lambda_-.$$
 This Hamiltonian chord lives in the concave end and has endpoints on the cylinder over $\Lambda_- \cup \Lambda_-^\epsilon$; in fact, it corresponds to a very short $\alpha_-$-Reeb chord from $\Lambda_-^\epsilon$ to $\Lambda_-$,
but traversed in \emph{reverse} time. The following relationship between Reeb chords from $\Lambda_-$ to itself and Reeb chords from $\Lambda_-$ to $\Lambda_-^\epsilon$ holds.

\begin{lma}
\label{lma:ChordBijection}
There is a bijective correspondence between the Reeb chords
$$[0,\ell] \ni t \mapsto \phi^t_{\alpha_-,1}(x), \:\: x\in\Lambda_-, \phi^\ell_{\alpha_-,1}(x) \in \Lambda_-,$$
from $\Lambda_-$ to itself of length $\ell>0$, and
\begin{enumerate}
\item the Reeb chords
$$[0,\ell-\epsilon] \ni t \mapsto \phi^{t}_{\alpha_-,1}(x), \:\: x \in \Lambda_-,\phi^{\ell-\epsilon}_{\alpha_-,1}(x) \in \Lambda_-^\epsilon,$$
from $\Lambda_-$ to $\Lambda_-^\epsilon$ of length $\ell-\epsilon.$ Furthermore, the former chord is trivial inside $\pi_1(Y_-,\Lambda_-)$ if and only if the latter chord is contained in the class $\mathfrak{p}_0 \in \pi_0(\Pi(X;L,L_{\epsilon,N,s}))$ under the canonical embedding $Y_-=\{-N-2\epsilon\} \times Y_- \subset X,$ and
\item the Reeb chords
$$[0,\ell+\epsilon] \ni t \mapsto \phi^{t}_{\alpha_-,1}(x), \:\: x \in \Lambda_-^\epsilon,\phi^{\ell+\epsilon}_{\alpha_-,1}(x) \in \Lambda_-,$$
from $\Lambda_-^\epsilon$ to $\Lambda_-$ of length $\ell+\epsilon.$ Furthermore, the former chord is trivial inside $\pi_1(Y_-,\Lambda_-)$ if and only if the latter chord is contained in the class $\mathfrak{p}_0^{-1} \in \pi_0(\Pi(X;L_{\epsilon,N,s},L))$ (i.e.~inverting the time of the paths in class $\mathfrak{p}_0$) under the canonical embedding $Y_-=\{-N-2\epsilon\} \times Y_- \subset X.$
\end{enumerate}
\end{lma}
The proof is straightforward; see Figure \ref{fig:chords1} for an illustration.
\begin{lma}
\label{lma:ContractibleClass}
\begin{enumerate}
\item Any intersection point $L \cap L_{\epsilon,N,s}$ contained outside of the subset $\{|r \pm N | \ge \epsilon\}$ is contained in the region $[-N+\epsilon,-\epsilon] \times Y$ of the ``bulge.'' (See Figure \ref{fig:bulge2}.) There is a one-to-two correspondence between the Reeb chords from $\Lambda_-^\epsilon$ to $\Lambda_-$ of length $\epsilon< \ell < s+\epsilon$ and such intersection points; one is contained near $r=-N+\epsilon$ while one is contained near $r=0.$ When considered as paths inside $\Pi(X;L_{\epsilon,N,s},L)$ (note the order!), these intersection points and corresponding chords moreover live in the same component.
\item For numbers $s$ satisfying $s \in [\epsilon, \sigma]$, there is no intersection point
$$L \cap L_{\epsilon,N,s} \cap \{|r \pm N | \ge \epsilon\}$$
contained in the component $\mathfrak{p}_0 \in \pi_0(\Pi(X;L,L_{\epsilon,N,s})).$
\end{enumerate}
\end{lma}
The correspondence of the above lemma is schematically depicted in Figure \ref{fig:chords2}.
\begin{proof}

(1): This follows from the construction of $L_{\epsilon,N,s}$ in Section \ref{sec:explicit}. Namely, the cobordism $L_{\epsilon,N,s}$ is obtained from $L$ by creating a ``bulge'' in the region $[-N+\epsilon,-\epsilon] \times Y,$ where this bulge is created by an application of the positive Reeb flow up to time at most equal to $s.$ Using this fact it is now an easy matter to homotope the constant path at such an intersection point to a Reeb chord inside $\{-N-\epsilon\} \times Y_-$ going from $\{-N-\epsilon\} \times \Lambda_-^\epsilon$ to $\{-N-\epsilon\} \times \Lambda_-$ and being of of length $>\epsilon.$

(2): The statement follows from Part (1) of this lemma combined with Part (2) of Lemma \ref{lma:ChordBijection} above.
\end{proof}

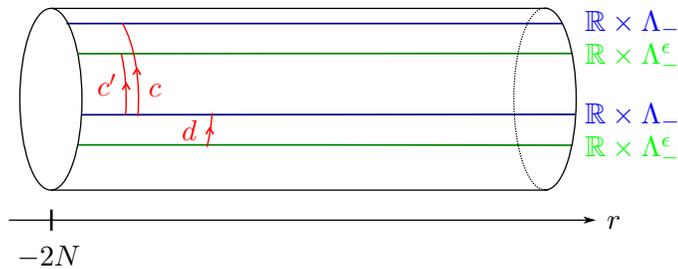
\begin{figure}[htp]
\centering
\labellist
\pinlabel $-2N$ at 15 -10
\pinlabel $\color{green}\R\times\Lambda_-^\epsilon$ at 235 30
\pinlabel $\color{green}\R\times\Lambda_-^\epsilon$ at 235 65
\pinlabel $\color{blue}\R\times\Lambda_-$ at 235 43
\pinlabel $\color{blue}\R\times\Lambda_-$ at 235 78
\pinlabel $r$ at 228 3
\pinlabel $\color{red}c'$ at 37 55
\pinlabel $\color{red}c$ at 55 53
\pinlabel $\color{red}d$ at 68 38
\endlabellist
\includegraphics{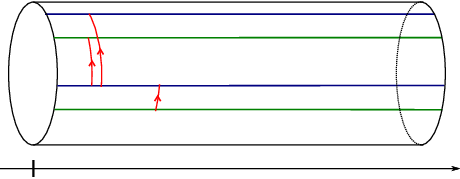}
\vspace{5mm}
\caption{The correspondence between the Reeb chords provided by Lemma \ref{lma:ChordBijection}: the Reeb chord $c$ on $\Lambda_-$ corresponds to a chord $c'$ from $\Lambda_-$ to $\Lambda_-^\epsilon.$ The chord $d$ of length $\epsilon$ from $\Lambda_-^\epsilon$ to $\Lambda_-$ lives in the component $\mathfrak{p}_0 \in \pi_0(\Pi(X;L,L_{\epsilon,N,s}))$ when traversed backwards in time.}
\label{fig:chords1}
\end{figure}

\begin{figure}[htp]
\centering
\labellist
\pinlabel $-N+\epsilon$ at 97 -10
\pinlabel $0$ at 168 -10
\pinlabel $\color{green}\R\times\Lambda_-^\epsilon$ at 235 30
\pinlabel $\color{green}\R\times\Lambda_-^\epsilon$ at 235 65
\pinlabel $\color{blue}\R\times\Lambda_-$ at 235 43
\pinlabel $\color{blue}\R\times\Lambda_-$ at 235 78
\pinlabel $\color{red}c$ at 50 55
\pinlabel $\color{red}c_1$ at 110 90
\pinlabel $\color{red}c_2$ at 164 90
\pinlabel $\color{red}d$ at 82 51
\pinlabel $\color{red}d_1$ at 129 51
\pinlabel $\color{red}d_2$ at 143 51
\pinlabel $r$ at 228 3
\endlabellist
\includegraphics{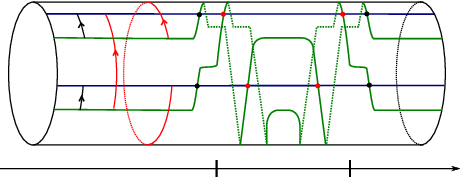}
\vspace{5mm}
\caption{The correspondence between the Reeb chords provided by Lemma \ref{lma:ContractibleClass}. For every Reeb chord $c$ from $\Lambda_-^\epsilon$ to $\Lambda_-$ of a (suitably) bounded length, there are precisely two corresponding intersection points $c_1$ and $c_2$ in the bulge region; one is contained near $r=-N+\epsilon$ while one is contained near $r=0.$}
\label{fig:chords2}
\end{figure}
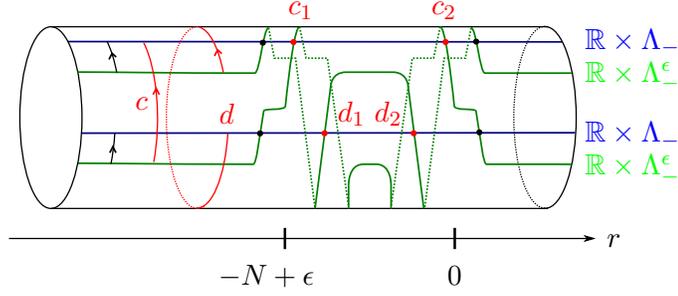

\subsection{Fixing the actions of the Hamiltonian chords}
\label{sec:actions}
Recall the action conventions from Section \ref{sec:action}, where an action is associated to the primitives of the one-form $\varphi\eta$ pulled back to the two Lagrangian cobordisms. Here $\varphi \colon X \to \R_{>0}$ is a choice of an auxiliary piecewise smooth function. In order to get optimal results, we need to be careful when choosing this function. We will take 
 the Hofer function $\varphi$ to be constantly equal to $1$ on $\overline{X}$, while it takes the values
\begin{equation}
\label{eq:varphi}
\varphi(r)=\begin{cases} e^{-r-2\epsilon}, & r \le -N-\epsilon, \\
e^{N-\epsilon}, & -N-\epsilon \le r \le -N+\epsilon, \\
e^{-r}, & -N+\epsilon \le r \le 0, \\
1, & 0 \le r \le N+\epsilon, \\
e^{-r+N+\epsilon}, & r \ge N+\epsilon,
\end{cases}
\end{equation}
on the cylindrical ends.
 In other words, with the notation from Definition \ref{dfn:varphi}, we have the associated constant $r_\varphi=2\epsilon$. It now follows from Lemma \ref{lma:ChordBijection} that:
\begin{lma}
\label{lma:hbar}
We have an equality
$$\hbar = \hbar(\varphi,\mathfrak{p}_0,\Lambda_-,\Lambda_-^\epsilon,\alpha_-) =e^{-2\epsilon}(\sigma(\Lambda_-,\alpha_-)-\epsilon),$$
where the left-hand side was defined in equation (\ref{eq:hbar_def}).
\end{lma}

Although the continuous function $\varphi$ only is piecewise smooth, such a function is sufficient for our purposes: recall that pseudoholomorphic curves have positive $d(\varphi \eta)$-energy by Lemma \ref{lma:varphienergy}. To that end, we must use an admissible almost complex structure in the sense of Section \ref{sec:AdmissibleACS}, which in particular is cylindrical in the complement of the subset $\{\varphi'(r)=0\}$ of the cylindrical ends.

We will use $d(\varphi \eta)$-energy (resp.~$d\eta$-energy) and the $\mathfrak{a}_\varphi$-action (resp.~$\mathfrak{a}$-action)
of Floer generators, in order to obstruct the existence of certain pseudoholomorphic curves when we prove Lemma \ref{lma:cone} and Proposition
\ref{prp:refinedinvariance}.

Fix $s \in [\epsilon, \sigma]$ and choose primitives $f_0$ and $f_1$ of the pull-backs of $\eta$ to $L$ and $ L_{\epsilon,N,s}$, respectively, and primitives $f_0^\varphi$ and $f_1^\varphi$ of the pull-backs of $\varphi\eta$ to $L$ and $L_{\epsilon,N,s},$ respectively. All primitives here are required to vanish at $r=-\infty$, and they are thus uniquely determined. In the following we will also assume that
$$f_0=f_0^\varphi \equiv 0$$
 outside of the non-cylindrical part $\overline{X}$ (i.e.~this equality in particular holds inside the subset where all intersection points reside) which, in view of Lemma \ref{lma:equalaction}, causes no restriction. Now we estimate the actions defined in Section \ref{sec:action} for the 0-Hamiltonian chords using the above primitives.
\begin{lma}[The main action computation]
\label{lma:ActionSummary}
Let $\mathbf{O}(\epsilon)>\mathbf{o}(\epsilon) > 0 $ be fixed, but unspecified, positive functions of $\epsilon, N$, defined for all $N,\epsilon \ge 0,$ that tend to $0$ as $\epsilon \to 0$ for any fixed choice of $N>0.$ Let $p,q \in L \cap L_{\epsilon,N,s}$ be 0-Hamiltonian chords. 
Let $s \ge \epsilon.$
Under the assumption that $N > 0$ is sufficiently large and $\epsilon>0$ is sufficiently small, we may assume the following:
\begin{enumerate}
\item For $p$ near $\{ r = -N\},$
$$- \mathbf{O}(\epsilon)< \mathfrak{a}(p) <0 \:\: \text{ and } \:\: -\mathbf{O}(\epsilon) < \mathfrak{a}_\varphi(p)<0;$$
\item For $p$ near $\{ r = N\},$ 
$$
 \mathbf{o}(\epsilon) < \mathfrak{a}_\varphi(p) < \mathbf{O}(\epsilon)
 \:\: \text{ and } \:\:
|\mathfrak{a}(p)-s| < \mathbf{O}(\epsilon);$$
\item
Inside the subset $[0,N-\epsilon] \times Y$ both potentials $f_1, f_1^\varphi \colon L_1 \to \R$ are constant, where they moreover satisfy
$$0 < f_1^\varphi < \frac{1}{2}\mathbf{o}(\epsilon), \:\: \text{ and } \:\:s-\mathbf{O}(\epsilon)<f_0-f_1<s+\mathbf{O}(\epsilon).$$
\end{enumerate}
\end{lma}

\begin{proof}
For $i = 1,$ respectively $i=2,$ let $p_i$ be a $0$-Hamiltonian chord sitting in the contact slice $\{r = r_i\}$ in the concave, respectively convex, end.
The Morse--Bott perturbation in Section \ref{sec:MorseBott} used to produce $L_{\epsilon,N,s}$ implies that 
$|r_i - (-1)^i N| < \epsilon^3.$ 
Moreover, to within an error bounded by 
$$ 
||g_{\epsilon,N,s} - \widetilde{g}_{\epsilon,N,s}||_{C^1} + \int_{r = \pm N - \epsilon^3}^{\pm N + \epsilon^3}
|e^r \beta'(r) \max\{\varphi(r), 1\}| d r 
< 
\epsilon^2
$$
$\mathfrak{a}(p_i)$ and $\mathfrak{a}_\varphi(p_i)$
have the following integral approximations:
\[
\mathfrak{a}(p_i) \approx -\int_{-\infty}^{(-1)^i N} e^r \beta'_{\epsilon,N,s}(r) d r , \quad
\mathfrak{a}_\varphi(p_i) \approx -\int_{-\infty}^{(-1)^i N} e^r \varphi(r) \beta'_{\epsilon,N,s}(r) d r.
\]
Here we use that both primitives $f_0$ and $f_0^\varphi$ on $L$ vanish identically in the region containing the intersection points (in fact, everywhere outside of $\overline{X}$), together with Lemma \ref{lma:ActionComp} for the computation of the potentials on $L_1.$

We make several observations to obtain the required estimates.
First, 
$$
\mbox{supp}(\beta'_{\epsilon,N,s}) \subset [-N-\epsilon, -N+2 \epsilon] \cup [-\epsilon, 0] \cup [N-\epsilon, N+ \epsilon].
$$
Second, 
$$\left|\int_{r_0}^{r_0'} \beta'_{\epsilon,N,s}(r) d r\right|= |\beta_{\epsilon,N,s}(r_0') - \beta_{\epsilon,N,s}(r_0)|
 = 
 \begin{cases} \epsilon \,, &\mbox{if} \,\, (r_0,r_0') \in \{(\pm N - \epsilon, \pm N), (\pm N, \pm N + \epsilon) \} \,,\\
 s - \epsilon \,, &\mbox{if} \,\, (r_0,r_0') \in \{(-N + \epsilon, -N+2\epsilon), (-\epsilon, 0)\}.
 \end{cases}
$$
Third, on any such interval $(r_0,r_0')$, all factors $I(r)$ in front of the $ \beta'_{\epsilon,N,s}(r)$-term in the integrand are monotonic; thus,
\begin{eqnarray*}
-I(r_i)(\beta_{\epsilon,N,s}(r_0') - \beta_{\epsilon,N,s}(r_0)) & \le & -\int_{r_0}^{r_0'} I(r) \beta'_{\epsilon,N,s}(r) d r \\
& \le & -I(r_j)(\beta_{\epsilon,N,s}(r_0') - \beta_{\epsilon,N,s}(r_0))
\end{eqnarray*}
 where $r_0 = r_i < r_j=r_0'$ or $r_0 = r_j < r_i=r_0'.$

We apply these observations (along with the error bounds on the approximations) to several integral computations to deduce the inequalities of the lemma. 
As we will see below, the lower bound for $\mathfrak{a}_\varphi(p_2)$ will motivate setting
\begin{equation}
\label{eq:o}
 \mathbf{o}(\epsilon) = (-e^{-\epsilon} - e^{0} + e^{N-3\epsilon}) \epsilon- \epsilon^2
\end{equation}
which we claim is positive because we can assume without loss of generality that $N-3\epsilon \ge 1$ and $-2+e^1> \epsilon.$
We omit the precise exponential-polynomial formulation of $ \mathbf{O}(\epsilon)$ in terms of $\epsilon, N,$ since $ \mathbf{O}(\epsilon)$ only appears as an upper bound (in magnitude).
The $(\pm \epsilon^2)$-terms included below will cover the integral approximations made above.

The bounds for $\mathfrak{a}(p_1)$ follow from 
$$
 -\epsilon^2 > -e^{-N - \epsilon} \epsilon > -\int_{-N-\epsilon}^{-N} e^r \beta'_{\epsilon,N,s}(r) d r > -e^{-N} \epsilon > -\mathbf{O}(\epsilon) +\epsilon^2.
 $$
The bounds for $\mathfrak{a}_\varphi(p_1)$ follow from
$$
 -\epsilon^2 > -e^{-2\epsilon} \epsilon > - \int_{-N-\epsilon}^{-N} e^r \varphi (r) \beta'_{\epsilon,N,s}(r) d r >- e^{-\epsilon} \epsilon > -\mathbf{O}(\epsilon) +\epsilon^2.
$$
The bounds for $\mathfrak{a}(p_2)$ follow from 
\begin{eqnarray*}
&s - \frac{1}{2} \mathbf{O}(\epsilon) - \epsilon^2 > (s - \epsilon)(e^0 - e^{-N+\epsilon}) >
 &\\
& -\left(\int_{-N+ \epsilon}^{-N+2\epsilon} + \int_{-\epsilon}^0\right) e^r \beta'_{\epsilon,N,s}(r) d r
> (s- \epsilon) (e^{-\epsilon} - e^{-N+2\epsilon}) > s - \mathbf{O}(\epsilon) &
\end{eqnarray*}
together with
\begin{eqnarray*}
\frac{1}{2} \mathbf{O}(\epsilon) 
& > & (-e^{-N-\epsilon} - e^{-N} + e^N) \epsilon \\
&> &- \left(\int_{-N- \epsilon}^{-N} + \int_{-N}^{-N+\epsilon}+\int_{N-\epsilon}^{N} \right) e^r \beta'_{\epsilon,N,s}(r) d r \\
& > & (-e^{-N} - e^{-N+\epsilon} + e^{N-\epsilon}) \epsilon > \epsilon^2.
\end{eqnarray*} 
The bounds for $\mathfrak{a}_\varphi(p_2)$ follow from 
$$0 = -\left(\int_{-N+ \epsilon}^{-N+2\epsilon} + \int_{-\epsilon}^0\right) e^r \varphi\beta'_{\epsilon,N,s}(r) d r $$
together with
\begin{eqnarray*}
\mathbf{O}(\epsilon) - \epsilon^2 \
& > & ( -e^{-2\epsilon} - e^{-\epsilon} + e^{N-2\epsilon}) \epsilon \\
& > & -\left(\int_{-N- \epsilon}^{-N} + \int_{-N}^{-N+\epsilon}+\int_{N-\epsilon}^{N} \right) e^r \varphi(r) \beta'_{\epsilon,N,s}(r) d r \\
&>& (-e^{-\epsilon} - e^{0} + e^{N-3\epsilon}) \epsilon 
= \mathbf{o}(\epsilon) + \epsilon^2.
\end{eqnarray*} 

First we note that, since $\beta'|_{[0, N - \epsilon]} =0,$ the claim made in Part (3) that the potentials $f_1$ and $f_1^\varphi$ are constant in the concerned region holds.

To compute the bounds, we use the same integration estimates as for $\mathfrak{a}(p_2)$ and $\mathfrak{a}_\varphi(p_2)$ above, except that we omit the fifth (last) integral as it occurs when $r > N -\epsilon.$ 
We see that for $0 \le r \le N - \epsilon,$ 
\begin{eqnarray*}
s+ \mathbf{O}(\epsilon)
& > & (s - \epsilon)(e^0 - e^{-N+\epsilon})+(-e^{-N-\epsilon} - e^{-N}) \epsilon + \epsilon^2 \\
& > & -(f_1 - f_0)(r) \\
& > & (s- \epsilon) (e^{-\epsilon} - e^{-N+2\epsilon})+ (-e^{-N} - e^{-N+\epsilon}) \epsilon - \epsilon^2 > s - \mathbf{O}(\epsilon),
\end{eqnarray*}
as well as
\begin{eqnarray*}
\lefteqn{0 > ( -e^{-2\epsilon} - e^{-\epsilon} ) \epsilon + \epsilon^2 > -(f_1^\varphi - f_0^\varphi)(r)} \\
& > & (-e^{-\epsilon} - e^{0} ) \epsilon- \epsilon^2 > -\frac{1}{2}\mathbf{o}(\epsilon).
\end{eqnarray*} 
The last inequality holds for sufficiently large $N$ in \eqref{eq:o} and all arbitrarily small $\epsilon>0.$
Plugging in $f_0^\varphi = 0 = f_0$ gives the answer.
\end{proof}

 Since $\hbar$ is independent of $\epsilon$ and $N$ from Lemma \ref{lma:ActionSummary}, 
after
possibly shrinking $\epsilon>0$, we can assume that $\mathbf{O}(\epsilon)$ in Lemma \ref{lma:ActionSummary} satisfies\begin{equation}
\label{eq:epsilon1}
2 \mathbf{O}(\epsilon) < \hbar.
\end{equation}
Further, by using Part (3) of Lemma \ref{lma:ActionSummary} together with Inequality \eqref{eq:sigma}, and after possibly shrinking $\delta>0$, we may moreover assume that
\begin{equation}
\label{eq:epsilon2}
e^{A+2\delta} \|H\|_{\OP{osc}} < \sigma-\mathbf{O}(\epsilon) <(f_0 - f_1)|_{[0,N-\epsilon]\times Y} < \sigma + \mathbf{O}(\epsilon)
\end{equation}
is satisfied when $\epsilon>0$ is sufficiently small and $s = \sigma$ (that is, when $f_1$ is the primitive of the pull back of $\eta$ to $L_{\epsilon, N, s} = L_{\epsilon, N, \sigma}$).

\subsection{The Floer homology of the push-off}
\label{sec:FloerHomologyPush}

By Lemma \ref{lma:ContractibleClass} we conclude that
\[CF^{\mathfrak{p}_0}_*(L,L_{\epsilon, N, s};0,J) = C_*^- \oplus C_*^+\]
holds for all $\epsilon \le s \le \sigma,$ where $C_*^+,$ resp. $C_*^-,$ denotes the vector subspace generated by the intersection points at the convex, resp. concave, end. Recall the superscript introduced in Section \ref{sec:LagIntHF} which indicates that the generating intersection points lie in the same contractible homotopy class. Also, recall the choice of Morse function $h: \Lambda \rightarrow [0, 1]$ made in Section \ref{sec:MorseBott}. By the construction of $L_{\epsilon,N,\sigma}$ in the same section, there is an isomorphism $C_*^-=C_{*-1}^+=C^{\mathrm{Morse}}_*(h)$ of graded vector spaces induced by a canonical identification of the generators.

\begin{lma}[Identifying Floer and Morse homology]
\label{lma:cone}
Using the above canonical identification of generators, the Floer complex
\[\left(CF^{\mathfrak{p}_0}_* (L,L_{\epsilon,N,\sigma};0,J_t)=C_*^- \oplus C_*^+,d=\begin{pmatrix} d_- & 0 \\ \psi & d_+ \end{pmatrix}\right)\]
is well-defined, has homotopy class independent of the choice of admissible almost complex structure, and is given as an acyclic mapping cone of $\psi$ (the acyclicity is equivalent to $\psi$ being a quasi-isomorphism). Moreover, there are homotopy equivalences
$$(C_*^-,d_-) \sim (C^{\OP{Morse}}_*(h),\partial_h) \:\: \text{and} \:\: (C_{*-1}^+,d_+)\sim(C^{\OP{Morse}}_{*}(h),\partial_h)$$
of the respective quotient and subcomplexes and the Morse homology complex of $\Lambda$.
\end{lma}

\begin{proof}
Inequality \eqref{eq:epsilon1} and Parts (1) and (2) of Lemma \ref{lma:ActionSummary} imply that, for sufficiently small $\epsilon >0,$ and for any generators $p,q \in CF^{\mathfrak{p}_0}_* (L,L_{\epsilon,N,s};0,J_t),$ $s \in [\epsilon,\sigma],$
the inequality
$$|\mathfrak{a}_\varphi(p) - \mathfrak{a}_\varphi(q)| < \hbar$$
is satisfied. Here we also rely on Lemma \ref{lma:ContractibleClass} in order to infer that $$CF^{\mathfrak{p}_0}_* (L,L_{\epsilon,N,s};0,J_t)=C_*^- \oplus C_*^+$$
holds on the level of vector spaces. Consequently, Part (1) of Proposition \ref{prp:bifurcation} applies, showing that
$$ CF^{\mathfrak{p}_0}_* (L,L_{\epsilon,N,s};0,J_t), \:\: s \in [\epsilon,\sigma],$$
are well defined Floer complexes for generic admissible almost complex structures.

The cone-structure of the complexes now follows from Parts (1) and (2) of Lemma \ref{lma:ActionSummary} combined with Lemma \ref{lma:varphienergy}. Specifically, for any two generators $p_\pm \in C_*^\pm,$
$$ \mathfrak{a}_\varphi(p_+) - \mathfrak{a}_\varphi(p_-) > 0.$$

Recall from Section \ref{sec:wein}, for $s$ sufficiently close to $\epsilon,$ $L_{\epsilon,N,s}$ is a graphical Lagrangian which sits in a Weinstein tubular neighborhood of $L$ symplectomorphic to $T^*L.$ 
So for sufficiently small $\epsilon$ and $s,$ we can use Floer's original work which proves that the pseudoholomorphic strips (for a suitable time-dependent $J_t$) with boundary on $L$ and $L_{\epsilon, N, s}$ converge to gradient flow lines for a suitable metric on $L$
\cite{Floer:MorseTheoryLagrangian}.
To understand the underlying Morse complexes, recall the construction of $\widetilde{g}$ and its close perturbation $g.$
From standard Morse--Bott analysis (see \cite[Proposition 1.9]{MorseNovikov}, for example), the negative gradient flows of $g$ leave the critical level sets 
$\{q = \pm(N+1)\} = \{r = \pm N\} \cap L$ invariant.
Moreover, no such negative gradient flows run from $\{q = -(N+1)\}$ to $\{q = +(N+1)\}.$
Thus the result holds for sufficiently small $\epsilon$ and $s$ (and suitable $J$). 

Next consider an arbitrary admissible almost complex structure (see Section \ref{sec:AdmissibleACS}) and an arbitrary $L_{\epsilon,N,s},$ $s \in [\epsilon, \sigma]$ as in the statement of the lemma. Floer's computation is valid after a suitable compactly supported deformation of this almost complex structure and for $s=\epsilon.$ Using the invariance properties of the Floer complex we will deduce the statement for an arbitrary $s \in [\epsilon,\sigma]$ as well as a general admissible almost complex structure.

The family $L_{\epsilon,N,s},$ $s \in [\epsilon,\sigma],$ of Lagrangian cobordisms is generated by a Hamiltonian isotopy by their construction in Section~\ref{sec:PushOff}. Note that, according to Lemma \ref{lma:ContractibleClass}, no births/deaths of intersection points $L \cap L_{\epsilon,N,s}$ corresponding to chords in the homotopy class $\mathfrak{p}=0$ occur during this isotopy. Thus, we can continue to identify the vector spaces of $0$-Hamiltonian chords as $C_*^\pm.$ Part (2) of Proposition \ref{prp:bifurcation}, proven using bifurcation analysis, now shows that the homotopy class of the complexes is invariant. (Here we make use of the fact that the Hamiltonian is compactly supported, and that the Lagrangians are cylindrical outside of a compact subset, in order to obtain a well-defined identification of the ``homotopy class'' of an intersection point.)

Finally, by Stokes' theorem together with (1) and (2) of Lemma \ref{lma:ActionSummary}, no handle-slide strip in the bifurcation analysis can originate from a generator of $C_*^+$ and terminate at a generator of $C_*^-.$ The aforementioned chain homotopy equivalence thus descends to a homotopy equivalence on the sub and quotient complexes, as claimed.
\end{proof}

\subsection{Action properties when applying Usher's trick}
\label{sec:ContinuationMaps}

Given the contact Hamiltonian $H_t \colon Y \to \R$ and the constant $A \ge 0$ in the assumption (see Equation (\ref{eq:norms})), Proposition \ref{prp:usher} produces an associated Hamiltonian $G_t \colon X \to \R$ for which $\phi^1_{G_t}(L)$ and $\phi^1_{e^r H_t}(L)$  have coinciding images  in the subset $[A+\delta/2,A+3\delta/2] \times Y.$ In particular, both of these images are cylinders over $\phi^1_{\alpha,H_t}(\Lambda)$ in that subset.

Let $M_\pm$ be as defined as in \eqref{eq:norms}. Recall the choice of a small number $\delta>0$ made above (see Inequalities \eqref{eq:sigma} and \eqref{eq:epsilon2}) and denote
\begin{equation}
\label{eq:Mbar}
\overline{M}_- := e^{2\delta}M_- < 0 < e^{2\delta}M_+ =: \overline{M}_+.
\end{equation}
The Hamiltonian $G_t \colon X \to \R$ may be assumed to satisfy
\begin{equation}
\label{eq:Mbar2}
\overline{M}_- < -\int_0^1 \max_{x \in X} G_t \, dt \le -\int_0^1 \min_{x \in X} G_t \, dt < \overline{M}_+
\end{equation}
according to Proposition \ref{prp:usher}.

Here we investigate properties of the action of the generators of the Floer complex when the latter Hamiltonian is turned on.

\begin{lma} 
\label{lma:Cartan}
The primitives $\overline{f}_0, \overline{f}_0^\varphi \colon \phi^1_{G_t}(L) \to \R$ of the pull backs of $\eta$ and $\varphi\eta,$ respectively, to $\phi^1_{G_t}(L),$ both vanish in the complement of the subset
$$[0,A+\delta/2] \times Y \: \cup \: [A+ 3\delta/2,N-2\epsilon] \times Y,$$
given that the primitives are chosen to vanish at $r=-\infty.$ (Recall that the support of $G_t$ can be assumed to be contained inside $[0,N-2\epsilon]$ when $N \gg 0$ is sufficiently large; see Proposition \ref{prp:usher}.)
\end{lma}

\begin{proof}
First we recall that the corresponding primitives $f_0, f_0^\varphi \colon L \to \R$ both vanish identically.

We compute the change in these primitives under the Hamiltonian isotopy $\phi^t_{G_t}$ in the following manner. By using Cartan's formula
\[\frac{d}{dt} (\phi^t_{G_t})^*(e^r\alpha)=d(\iota_{X_t} e^r\alpha)+\iota_{X_t}d(e^r\alpha)=d(\iota_{X_t} e^r\alpha)-dG_t,\]
the new primitives can be obtained by integrating the variable $t$ of the function $\iota_{X_t} e^r\alpha-G_t$.

Outside of $[0,N-2\epsilon] \times Y$ we have $G_t \equiv 0$ by construction, and the statement is an easy consequence.

For the image
\[\phi^1_{G_t}(L) \: \cap \: [A+\delta/2,A+3\delta/2] \times Y\]
we argue as follows. Any $p \in L$ for which $\phi^1_{G_t}(p) \in [A+\delta/2,A+3\delta/2] \times Y$ satisfies the property that
\[G_t \circ \phi^t_{G_t}(p)=e^r K_t \circ \phi^t_{e^r K_t}(p), \:\: t \in [0,1],\]
by the construction in Section \ref{sec:usher} (in particular, see Part (2) of Proposition \ref{prp:usher}). At such a point $p \in L$, the above application of Cartan's formula thus tells us that the infinitesimal change of the primitive is given by
\[ \iota_{X_t} e^r\alpha -e^r K_t =0\]
since $\iota_{X_t} \alpha=K_t$ (recall that $K_t \colon Y \to \R$ is a contact Hamiltonian).
\end{proof}

\subsection{Turning on the splashing}
\label{sec:TurnSplash}

In this section we introduce the splashing construction from Section \ref{sec:splashing}. Recall that $\phi^1_{G_t}(L)$ and $L_{\epsilon,N,\sigma}$ both are cylindrical and disjoint inside $[A+\delta/2,A+3\delta/2] \times Y$, each intersecting the hypersurface $\{A \} \times Y$ of contact type in a Legendrian submanifold. The splashing construction will be performed inside the cylindrical piece
$$([A+\delta/2,A+3\delta/2] \times Y,d(e^r\alpha)) \subset (X,d\eta)$$
of the symplectization using the data
$$
\begin{cases}
T=A+2\delta/2,\\
3\varepsilon=\delta/2,\\
L_0=\phi^1_{G_t}(L),\\
L_1=L_{\epsilon,N,\sigma},\\
Z_+>\overline{M}_+>0,\\
Z_-<\overline{M}_-<0.
\end{cases}$$

Recall the Hamiltonian $h^{\OP{Spl}}:=h^{\OP{Spl}}_{\varepsilon, Z_-, Z_+} \colon \overline{X} \to \R$ obtained in Section \ref{sec:splashing} which realizes the splashing. We can assume this Hamiltonian to be arbitrarily small in the uniform $C^0$-norm, and supported inside $[A+\delta/2,A+3\delta/2] \times Y$

\begin{lma}
\label{lma:GtildeG}
There exists a Hamiltonian $\widetilde{G}_t \colon X \to \R$ with support contained in $[0,N-\epsilon] \times Y$ satisfying
\begin{align*}
& L^{\OP{Spl}}:=\phi^1_{\widetilde{G}_t}(L)=\phi^1_{h^{\OP{Spl}}} \circ \phi^1_{G_t}(L),\\
& \left|\int_0^1 \max \widetilde{G}_t dt - \int_0^1 \max {G}_t dt\right|< \nu/2, \\
& \left|\int_0^1 \min \widetilde{G}_t dt - \int_0^1 \min {G}_t dt\right|<\nu/2,
\end{align*}
for any choice of $\nu>0.$ In particular
$$\left| \| \widetilde{G}_t \|_{\OP{osc}} - \|G_t\|_{\OP{osc}} \right| < \nu.$$
\end{lma}
\begin{proof}
The Hamiltonian $\widetilde{G}_t$ is constructed as follows. First, it causes no restriction to assume that $G_t$ vanishes identically near $t=0,1$; this can be done with an arbitrarily small effect on the Hofer norm. Second, we may replace $h^{\OP{Spl}}$ by a non-autonomous Hamiltonian $h^{\OP{Spl}}_t$ also vanishing identically near $t=0,1$; this can be done while still having an arbitrarily small Hofer norm. Finally, we consider the concatenation
$$ \widetilde{G}_t = \begin{cases}
2 G_{2t}, & 0 \le t \le 1/2,\\
2h^{\OP{Spl}}, & 1/2 \le t \le 1,
\end{cases}
$$
for which
$$L^{\OP{Spl}}:=\phi^1_{\widetilde{G}_t}(L)=\phi^1_{h^{\OP{Spl}}_\varepsilon} \circ \phi^1_{G_t}(L)$$
is satisfied. 
\end{proof}

Although $CF^{\mathfrak{p}_0}_*(L,L_{\epsilon,N,\sigma};\widetilde{G}_t)$ may not be a chain complex due to ``unwanted bubbling," it still defines a vector space whose generators have an associated action. 
Fix the choice of action induced by $\varphi\eta$, where $\varphi \colon X \to \R_{\ge 0}$ is as defined in Equality \eqref{eq:varphi}. With such an action one defines the Floer complex
\[
C_*:=CF^{\mathfrak{p}_0}_*(L,L_{\epsilon,N,\sigma};\widetilde{G}_t)^{\overline{M}_+}_{\overline{M}_-}
\]
with generators in the action range 
 $[\overline{M}_-,\overline{M}_+),$
as described in Section \ref{sec:LagIntHF}. Note that this indeed is a well-defined complex by Proposition \ref{prp:welldefcomplex}, which follows since (\ref{eq:Mbar}), (\ref{eq:norms}), (\ref{eq:sigma}), (\ref{eq:hbar}),
and Lemma \ref{lma:hbar} (in this very order) combine to give
\begin{equation}
\label{eq:Mhbar}
\overline{M}_+-\overline{M}_+ = e^{2\delta} (M_+ - M_-) = e^{A+2\delta} || H_t ||_{\OP{osc}} <\sigma < \hbar.
\end{equation}
Since the initial intersection points $L_{\epsilon,N,\sigma} \cap L$ all lie outside of the support of $G_t$ (see Proposition \ref{prp:usher}), and hence also outside of the support of $\widetilde{G}_t$ (see Lemma \ref{lma:GtildeG}), these intersections all remain, forming a subset of the generators of $CF^{\mathfrak{p}_0}_*(L,L_{\epsilon,N,\sigma}; \widetilde{G}_t )^{\overline{M}_+}_{\overline{M}_-}.$
In other words, there are inclusions
$$C_*^\pm \subset C_*$$
on the level of vector spaces.

We define the ``left'' $L_*,$ ``right'' $R_*,$ and ``intersection'' $I_*$ vector subspaces as follows.
\begin{itemize}
\item $L_* \subset C_*$ is generated by those chords (in the appropriate action range) corresponding to the intersection points
$$L^{\OP{Spl}} \cap L_{\epsilon,N,\sigma} \cap \{r \le T+\varepsilon\},$$
and thus in particular $C^-_* \subset L_*;$
\item $R_* \subset C_*$ is generated by those chords (again in the appropriate action range) corresponding to the intersection points
$$L^{\OP{Spl}} \cap L_{\epsilon,N,\sigma} \cap \{r \ge T-\varepsilon\},$$
and thus in particular $C^+_* \subset R_*;$
\item
$I_*:= L_* \cap R_* \subset C_*.$
\end{itemize}
Here we have made heavy use of the naturality, as described in Section \ref{sec:Naturality}.

\begin{lma}
\label{lma:reebchords}
 For admissible $ J \in \mathcal{J}(L,L_{\epsilon,N,\sigma}, \varphi_\lambda)$ with $\lambda \gg 0$ sufficiently large, 
$L_*,$ $R_*$ are subcomplexes of $C_*.$ 
For small enough constants $\epsilon,\delta,\nu>0,$ there is a bijective correspondence between the generators of $I_*$ and the Reeb chords inside
$$\mathcal{Q}_\alpha(\phi^1_{\alpha,H_t}(\Lambda),\Lambda;M_-,M_+)=\mathcal{Q}_\alpha(\Lambda,\phi^1_{\alpha,H_t}(\Lambda);-M_+,-M_-)$$ (c.f.~Remark \ref{rmk:qa}).
\end{lma}
\begin{rmk}
\label{rmk:rabinowitz}
The intersection complex $I_* = L_* \cap R_*$ is related to the Rabinowitz--Floer homology complex as first defined in \cite{Rabinowiz} by Cieliebak--Frauenfelder in the non-relative case, and later in \cite{Merry}, \cite{LagrangianRabinowitz} by Merry in the relative case.
\end{rmk}

\begin{proof}
Again we will implicitly make use of the naturality property from Section \ref{sec:Naturality}. Two crucial identities that we need are that
\begin{eqnarray*}
\phi^1_{G_t}(L) \cap [A+\delta/2,A+3\delta/2] \times Y & =& [A+\delta/2,A+3\delta/2] \times \phi^1_{\alpha,H_t}(\Lambda),\\
L_{\epsilon,N,\sigma} \cap [A+\delta/2,A+3\delta/2] \times Y &=&[A+\delta/2,A+3\delta/2] \times \Lambda^\epsilon,
\end{eqnarray*}
both are cylindrical inside $[A+\delta/2,A+3\delta/2] \times Y,$ the first one being a consequence of Proposition \ref{prp:usher}.

Consider the primitives $\overline{f}^\varphi_0$ and $\overline{f}^\varphi_1=f^\varphi_1$ of $\varphi\eta$ pulled back to $\phi^t_{G_t}(L)$ and $L_{\epsilon,N,\sigma},$ respectively. We have $\overline{f}^\varphi_0 \equiv 0$ inside $[A+\delta/2,A+3\delta/2] \times Y$ by Lemma \ref{lma:Cartan}, while Part (3) of Lemma \ref{lma:ActionSummary} implies that $0 < f^\varphi_1 < \frac{1}{2}\mathbf{o}(\epsilon).$ Recall that we set $C_i := \overline{f}^\varphi_i(T)$ in Section \ref{sec:neckstretching}, where
$T = A + 2\delta/2$ was specified at the beginning of this subsection. Since $C_0=0$ and $0 < C_1 <\frac{1}{2}\mathbf{o}(\epsilon),$ the assumptions of Proposition \ref{prp:WrapObstruction} are satisfied for the constants $Z_\pm$ chosen at the beginning of this subsection
 when $\lambda \gg 0$ is sufficiently large. 
The claim concerning the subcomplex property then follows from Proposition \ref{prp:WrapObstruction}.

The bijection between the relevant generators and the Reeb chords inside $\mathcal{Q}_\alpha(\phi^1_{\alpha,H_t}(\Lambda),\Lambda;M_+,M_-)$ finally follows from Lemma \ref{lma:IntersectionComplex}. (Note that the splash is performed to $\phi^1_{G_t}(L)$ inside $[A+\delta/2,A+3\delta/2] \times Y,$ i.e.~to the cylinder over $\phi^1_{\alpha,H_t}(\Lambda)$.) Since our set-up is generic, see Section \ref{sec:Results}, there is an interval $(e^{-A}M_+,e^{-A}M_+ +\mu)$ (resp.~$(e^{-A}M_--\mu,e^{-A}M_-)$), for $\mu >0$ sufficiently small (but possibly large compared to $\delta$), which contains no length of a Reeb chord from $\Lambda$ to $\phi^1_{\alpha,H_t}(\Lambda)$ (resp.~from $\phi^1_{\alpha,H_t}(\Lambda)$ to $\Lambda$).
\end{proof}

Since Lemma \ref{lma:reebchords} implies that $I_*$ is a subcomplex of both $L_*$ and $R_*$, in combination with $L_* + R_* = C_*,$ we get the short exact sequence whose maps are induced by inclusions
\begin{equation}
\label{eq:MVS}
0 \rightarrow I_* \rightarrow L_* \oplus R_* \rightarrow C_* \rightarrow 0.
\end{equation}

Recall the definition of $\Phi_{0,\widetilde{G}_t}$ and $\Phi_{\widetilde{G}_t,0}$ induced by the Hamiltonian $\widetilde{G}_t$ in Section \ref{sec:LagIntHF}.

\begin{lma}
\label{lma:leftinverse}
The maps
\begin{eqnarray*}
&& \Phi_{0,\widetilde{G}_t} \colon CF^{\mathfrak{p}_0}_*(L,L_{\epsilon,N,\sigma};0) \to CF^{\mathfrak{p}_0}_*(L,L_{\epsilon,N,\sigma};\widetilde{G}_t)^{\overline{M}_+}_{\overline{M}_-}, \\
&& \Phi_{\widetilde{G}_t,0} \colon CF^{\mathfrak{p}_0}_*(L,L_{\epsilon,N,\sigma};\widetilde{G}_t)^{\overline{M}_+}_{\overline{M}_-} \to CF^{\mathfrak{p}_0}_*(L,L_{\epsilon,N,\sigma};0),
\end{eqnarray*}
are well-defined chain maps, whose composition $\Phi_{\widetilde{G}_t,0} \circ \Phi_{0,\widetilde{G}_t}$ is chain homotopic to the identity.
\end{lma}
\begin{proof}
Let $m'_\pm := \pm \mathbf{O}(\epsilon)$ and $m_\pm := \overline{M}_\pm.$ With \eqref{eq:Mbar2} and Lemma \ref{lma:GtildeG} we now see that, for sufficiently small $\nu>0$ (defined in Lemma \ref{lma:GtildeG}), there exists a sufficiently small $\epsilon>0$ such that the constants $m_- \le m'_- \le m'_+ \le m_+$ satisfy the hypotheses of Proposition \ref{prp:invariance}. (Here the Hamiltonian $G_t$ of that proposition is replaced with $\widetilde{G}_t$ of this section.) The claims are now direct consequences of this proposition.
\end{proof}

Denote by $C^{\ge 0}_*$ the vector subspace spanned by all generators of $C_*$ contained outside of the concave end. We have a canonical decomposition
$$C_* = C^{-}_* \oplus C^{\ge 0}_*$$
as vector spaces.
\begin{prp}
\label{prp:refinedinvariance}
For a suitable admissible almost complex structure,
\[ C^+_* \subset C^{\ge 0}_* \subset CF^{\mathfrak{p}_0}_*(L,L_{\epsilon,N,\sigma};\widetilde{G}_t)^{\overline{M}_+}_{\overline{M}_-} \]
is a sequence of subcomplexes satisfying the following properties:
\begin{enumerate}
\item $\Phi_{0,\widetilde{G}_t}(C^+_*) \subset C^+_*$ and $\Phi_{\widetilde{G}_t,0}(C^+_*) \subset C^+_*;$
\item $\Phi_{\widetilde{G}_t,0}|_{C^+_*}$ is a left-sided homotopy inverse of $\Phi_{0,\widetilde{G}_t}|_{C^+_*}$ which induces an isomorphism on the homology level; and
\item $\Phi_{\widetilde{G}_t,0}(C^{\ge 0}_*) \subset C^+_*.$
\end{enumerate}
\end{prp}
\begin{proof}

The claim that the vector subspace
$$C^+_* \subset CF^{\mathfrak{p}_0}_*(L,L_{\epsilon,N,\sigma};\widetilde{G}_t)^{\overline{M}_+}_{\overline{M}_-}$$
in fact is a subcomplex will be proven together with Part (1) below, while
the claim that the vector subspace
$$C^{\ge 0}_* \subset CF^{\mathfrak{p}_0}_*(L,L_{\epsilon,N,\sigma};\widetilde{G}_t)^{\overline{M}_+}_{\overline{M}_-}$$
is a subcomplex will be proven together with Part (3) below.

In the following we will consider the primitives $\widetilde{f}^\varphi_0$ and $\widetilde{f}^\varphi_1=f^\varphi_1$ of $\varphi\eta$ pulled back to $\phi^t_{\widetilde{G}_t}(L)$ and $L_{\epsilon,N,\sigma},$ respectively. Note that $\widetilde{f}^\varphi_0$ coincides with the primitive $\overline{f}^\varphi_0$ of $\varphi\eta$ pulled back to $\phi^t_{G_t}(L)$ outside of the ``splashing region,'' i.e.~ $[A+\delta/2,A+3\delta/2] \times Y.$

(1): Let $p$ be a generator of $C^+_*.$ By Parts (2) and (3) of Lemma \ref{lma:ActionSummary}, together with Lemma \ref{lma:Cartan}, the constants $C_1 := \widetilde{f}_1^\varphi(N - 2\epsilon) < \frac{1}{2} \mathbf{o}(\epsilon)$ and $C_0 := \widetilde{f}_0^\varphi(N - 2\epsilon) = 0$ satisfy
$\mathfrak{a}_\varphi(p) - (C_1 - C_0)> (1-\frac{1}{2})\mathbf{o}(\epsilon).$
Since both $\widetilde{f}_i^\varphi,$ $i=0,1,$ are constant inside $[N - 2\epsilon, N - \epsilon],$
we can stretch the neck at $\{ N - 2\epsilon \le r \le N-\epsilon\}$ to define $\varphi_\lambda$ as in Formula (\ref{eq:VarphiNeckStretch}). 

Note that this stretching imposes a specific construction of the almost complex structure in $\{ N - 2\epsilon \le r \le N-\epsilon\}.$
However, the transversality arguments of Proposition \ref{prp:Transversality} still hold since all pseudoholomorphic strips that intersect 
 $\{ N - 2\epsilon \le r \le N-\epsilon\}$ also must limit to their Hamiltonian chords contained somewhere outside of this region. The latter subset is where the generic perturbation carried out by Proposition \ref{prp:Transversality} can be taken to occur.

After the neck stretching, Lemma \ref{lma:StretchedAction} shows that $\mathfrak{a}_{\varphi_\lambda}(p)$ can be assumed to be arbitrarily large for the intersection points $p$ that generate $C^+_*$ as a vector subspace of 
$CF^{\mathfrak{p}_0}_*(L,L_{\epsilon,N,\sigma};0)$ (resp. $CF^{\mathfrak{p}_0}_*(L,L_{\epsilon,N,\sigma};\widetilde{G}_t)$ ). 
(Note that we first fix $\mathbf{o}(\epsilon)>0,$ and hence all the other parameters involved in the construction, such as $\delta$ etc.
Then we increase $\lambda$ arbitrarily high.)
Let $q$ be an arbitrary generator of the orthogonal complement of $C^+_*$ in $CF^{\mathfrak{p}_0}_*(L,L_{\epsilon,N,\sigma};\widetilde{G}_t)$ (resp. $CF^{\mathfrak{p}_0}_*(L,L_{\epsilon,N,\sigma};0)$).
Note that $\mathfrak{a}_{\varphi_\lambda}(q) = \mathfrak{a}_{\varphi}(q).$
So Proposition \ref{prp:filtrationproperties} with $\widetilde{\varphi}=\varphi_\lambda$ can be applied to show that, for sufficiently large $\lambda,$ there are no strips which contribute to $\langle \Phi_{0,\widetilde{G}_t}(p), q \rangle$ 
 (resp. $\langle \Phi_{\widetilde{G}_t,0}(p), q \rangle$ ) .

(2): Recall Lemma \ref{lma:leftinverse}, which is based upon Proposition \ref{prp:invariance}, by which the composition $\Phi_{\widetilde{G}_t, 0} \circ \Phi_{0,\widetilde{G}_t}$ is chain homotopic to the identity. The same stretching argument as in the proof of Part (1) above shows that the Floer continuation homotopies in fact restrict to $C^+_*.$ In other words, the composition $\Phi_{\widetilde{G}_t, 0}|_{C^+_*} \circ \Phi_{0,\widetilde{G}_t}|_{C^+_*}$ of restrictions is also homotopic to the identity map $\id_{C^+_*}.$

The same stretching argument, combined with the invariance result Proposition \ref{prp:bifurcation} in terms of bifurcation analysis, shows that the homology of the subcomplex
$$C^+_* \subset CF^{\mathfrak{p}_0}_*(L,L_{\epsilon,N,\sigma};\widetilde{G}_t)^{\overline{M}_+}_{\overline{M}_-}$$
is isomorphic to the homology of the subcomplex
$$C^+_* \subset CF^{\mathfrak{p}_0}_*(L,L_{\epsilon,N,\sigma};0)^{\overline{M}_+}_{\overline{M}_-}.$$
 Since a left-inverse of a map between equidimensional vector spaces is an inverse, the isomorphism on the level of homology follows.

(3): Now we need to consider the primitives $\widetilde{f}_0$ and $\widetilde{f}_1$ of $\eta$ pulled back to $\phi^1_{\widetilde{G}_t}(L)$ and $L_{\epsilon,N,\sigma},$ respectively. Part (2) of Remark \ref{rmk:action} allows us to apply Part (3) of Lemma \ref{lma:ActionSummary} and Inequality \eqref{eq:epsilon2} in order to obtain 
$$
-(\widetilde{f}_1(N-\epsilon)-\widetilde{f}_0(N-\epsilon)) >e^{A+2\delta}\|H_t\|_{\OP{osc}}>-(\widetilde{f}^\varphi_1(N-\epsilon)-\widetilde{f}^\varphi_0(N-\epsilon))+e^{A+2\delta}\|H_t\|_{\OP{osc}}.$$
Roughly speaking, the ``bulge'' of $L_{\varepsilon,N,\sigma}$ of size $\sigma>0$ contained inside $[-N+\epsilon,-\epsilon] \times Y_-$ contributes an additional term of
$$ \sigma > e^{A+2\delta}\|H_t\|_{\OP{osc}} $$
to the primitive $\widetilde{f}_1$ of $\eta$ compared to the primitive $\widetilde{f}^\varphi_1$ of $\varphi\eta.$ In view of Inequality \eqref{eq:epsilon2}, it is important to choose $\varepsilon>0$ sufficiently small here.

Now, fix a generator $p\in C^{\ge 0}_*.$
If $p$ lies in the support of the Hamiltonian deformation $\widetilde{G}_t,$ Part (2) of Remark \ref{rmk:action} does not apply.
However, since for any $\tau,$ $\widetilde{G}_t$ changes $f_0(\tau)$ and $f^\varphi_0(\tau)$ by the same (possibly trivial) amount, while leaving $f_1(\tau), f^\varphi_1(\tau)$ fixed, we similarly obtain
\[ \mathfrak{a}(p)>\mathfrak{a}_\varphi(p)+e^{A+2\delta}\|H_t\|_{\OP{osc}}.\]
By definition of $C_*,$ $\mathfrak{a}_\varphi(p) \le \overline{M}_+$. 
Since
\[\overline{M}_\pm := e^{2\delta}M_\pm \:\: \text{and} \:\: e^{A+2\delta}\|H_t\|_{\OP{osc}}=e^{2\delta}(M_+-M_-),\]
we may hence assume that
\[0<-\int_0^1 \min_X \widetilde{G}_t dt <e^{2\delta}M_+\le \mathfrak{a}(p)\]
(here we have used Proposition \ref{prp:usher} and Lemma \ref{lma:GtildeG}).

The claim that
$$ C^{\ge 0}_* \subset CF^{\mathfrak{p}_0}_*(L,L_{\epsilon,N,\sigma};\widetilde{G}_t)^{\overline{M}_+}_{\overline{M}_-} $$
is a subcomplex is now a direct consequence of Lemma \ref{lma:ActionSummary} together with Part (2) of Remark \ref{rmk:action} (the differential is action increasing). Similarly, the action consideration in Proposition \ref{prp:filtrationproperties} using $\widetilde{\varphi}=1$ implies that the inclusion
$$ \Phi_{\widetilde{G}_t,0}(C^{\ge 0}_*) \subset C^+_* \subset CF^{\mathfrak{p}_0}_*(L,L_{\epsilon,N,\sigma};0) $$
is satisfied. Indeed, any $q \in C^-_* \subset CF^{\mathfrak{p}_0}_*(L,L_{\epsilon,N,\sigma};0)$ satisfies $\mathfrak{a}(q) < 0.$
\end{proof}

\begin{cor}
\label{cor:invariance}
For an almost complex structure as used in Proposition \ref{prp:refinedinvariance} the following hold.
\begin{enumerate}
\item The homology of the subcomplex
\[C^+_* \subset CF^{\mathfrak{p}_0}_*(L,L_{\epsilon,N,\sigma};\widetilde{G}_t)^{\overline{M}_+}_{\overline{M}_-}\]
is isomorphic to the Morse homology complex $(C^{\OP{Morse}}_{*-1}(h),\partial_h)$ of $\Lambda$;
\item Consider the inclusions
$$C^+_* \subset C^{\ge 0}_* \subset CF^{\mathfrak{p}_0}_*(L,L_{\epsilon,N,\sigma};\widetilde{G}_t)^{\overline{M}_+}_{\overline{M}_-}$$
of subcomplexes. The first inclusion is injective, while the composition of inclusions vanishes in homology.
\end{enumerate}
\end{cor}
\begin{proof}
(1): This follows from Part (2) of Proposition \ref{prp:refinedinvariance} together with Lemma \ref{lma:cone}.

(2): Parts (2) and (3) of Proposition \ref{prp:refinedinvariance} imply the existence of a commutative diagram of the form
\[\xymatrix{
H(C^+) \ar[r] \ar[d]^{\simeq}_{[\Phi_{\widetilde{G}_t,0}|_{C^+}]} & \ar[d]^{[\Phi_{\widetilde{G}_t,0}|_{C^{\ge0}}]} H(C^{\ge0})\\
H(C^+) \ar@{=}[r] & H(C^+),
}\]
where the horizontal maps are induced by the canonical inclusions of subcomplexes. This shows the first claim.

Lemma \ref{lma:cone} and Part (2) of Proposition \ref{prp:refinedinvariance} produces a commutative diagram of the form
\[\xymatrix{
H(C^+) \ar[r] \ar[d]^{\simeq}_{[\Phi_{0,\widetilde{G}_t}|_{C^+}]} & \ar[d]^{[\Phi_{0,\widetilde{G}_t}]} HF(L,L_{\epsilon,N,\sigma};0)=0\\
H(C^+) \ar[r] & HF(L,L_{\epsilon,N,\sigma};\widetilde{G}_t)^{\overline{M}_+}_{\overline{M}_-},
}\]
where the horizontal maps are induced by the canonical inclusions of subcomplexes. This implies the vanishing of the bottom map, as sought.

\end{proof}

\subsection{Finishing the proof of Theorem \ref{thm:main}}
 Just before Lemma \ref{lma:reebchords}, we introduced the subcomplex $L_*$ (resp. $R_*$) generated by Reeb chords of appropriate action to the left (resp. right) of $\{ r \le T + \varepsilon\}$ (resp. $\{ r \ge T - \varepsilon\}$). 
Recall Lemma \ref{lma:reebchords} and Proposition \ref{prp:refinedinvariance}, which imply that
$$C^+_* \overset{\iota}{\subset} R_* \overset{\iota_R}{\subset} C_*=CF^{\mathfrak{p}_0}_*(L,L_{\epsilon,N,\sigma};\widetilde{G}_t)^{\overline{M}_+}_{\overline{M}_-}$$
is a sequence of inclusions of subcomplexes.
So Part (2) of Corollary \ref{cor:invariance} implies that this composition
of inclusions vanishes on the level of homology, while the first inclusion is an inclusion on the homology level. Also, let
$$L_* \overset{\iota_L}{\subset} C_*=CF^{\mathfrak{p}_0}_*(L,L_{\epsilon,N,\sigma};\widetilde{G}_t)^{ \overline{M}_+}_{ \overline{M}_-}$$
 denote the inclusion of the subcomplex $L_*$. 
 We then have a commutative diagram in homology
\[ \xymatrix@C=4em{
& & H_*(C^+) \ar@{^(->}[d]_{[0\oplus\iota]} \ar[dr]^{[\iota_R\circ\iota]=0} & \\
\hdots \ar[r] & H_*(I) \ar[r] & H_*(L) \oplus H_*(R) \ar[r]^{[\iota_L+\iota_R]} & H_*(C) \ar[r] & \hdots
}\]
where the horizontal sequence is the exact Mayer--Vietoris sequence from \eqref{eq:MVS}. Combining this diagram with Part (1) of Corollary \ref{cor:invariance} we deduce that
\[\dim H(I) \ge \dim \ker([\iota_L+\iota_R]) \ge \dim H(C^+)=\dim H_*(\Lambda).\]
This inequality, combined with Lemmas \ref{lma:ReebChordCorrespondence} and \ref{lma:reebchords}, then finishes the proof. \qed

\bibliographystyle{plain}
\bibliography{references}
\end{document}